%% file: tracking.tex
\pdfoutput=1
\documentclass{jnsao}
\usepackage[utf8]{inputenc}
\usepackage[finnish,english]{babel}
\usepackage[inline]{enumitem}
\usepackage[nameinlink,capitalise]{cleveref}
\usepackage{mdframed}
\usepackage{subfiles}
\usepackage{tikz}
\usepackage{pgfplots}
\usetikzlibrary{colorbrewer}
\usepgfplotslibrary{groupplots}

\mdfdefinestyle{examplebg}{
    backgroundcolor=structure!7,%
    hidealllines=true,%
    innertopmargin=-.3em,%
    innerbottommargin=.7em,%
    innerleftmargin=.7em,%
    innerrightmargin=.7em,%
}

\surroundwithmdframed[style=examplebg]{example}
\surroundwithmdframed[style=examplebg]{algorithm}

\theoremstyle{definition}
\newtheorem{assumption}[definition]{Assumption}
\counterwithin{assumption}{section}

\theoremstyle{theorem}
\newtheorem{claim}[theorem]{Claim}

\crefname{assumption}{Assumption}{Assumptions}

\manuscriptcopyright{}
\manuscriptlicense{}
\manuscripteprint{2412.01481}
\manuscripteprinttype{arxiv}

\title{Differential estimates for fast first-order multilevel nonconvex optimisation}
\shorttitle{Differential estimates for multilevel optimisation}

\author{%
    Neil Dizon\thanks{School of Mathematics and Statistics, University of New South Wales, Sydney, Australia. \mbox{\email{n.dizon@unsw.edu.au}}, \orcid{0000-0001-8664-2255}}
    \and
    Tuomo Valkonen\thanks{MODEMAT Research Center in Mathematical Modeling and Optimization, Quito, Ecuador \emph{and} Department of Mathematics and Statistics, University of Helsinki, Finland. \email{tuomo.valkonen@iki.fi}, \orcid{0000-0001-6683-3572}}
}

\date{2024-12-02; revised 2026-03-31}

\acknowledgements{%
    This research has been supported by the Academy of Finland grants 314701 and 345486.
}

\input{symbols}

\def\ignorelegendentry#1{}
\pgfplotsset{
    ignore legend/.style={
        every axis legend/.code={\let\addlegendentry\ignorelegendentry}
    },
}

\begin{document}

\maketitle

\begin{abstract}
    With a view on bilevel and PDE-constrained optimisation, we develop iterative estimates $\estdiff F(\thisx)$ of $F'(\thisx)$ for composite functions $F \defeq J \circ S$, where $S$ is the solution mapping of the inner optimisation problem or PDE, the latter seen in this work as a special case of the former.
    The idea is to form a single-loop method by interweaving updates of the iterate $\thisx$ by an outer optimisation method, with updates of the estimate by single steps of standard optimisation methods and linear system solvers.
    When the inner methods satisfy simple tracking inequalities, the differential estimates can almost directly be employed in standard convergence proofs for general forward-backward type methods.
    We adapt those proofs to a general inexact setting in normed spaces, that, besides our differential estimates, also covers mismatched adjoints and unreachable optimality conditions in measure spaces.
    As a side product of these efforts, we provide improved convergence results for nonconvex Primal-Dual Proximal Splitting (PDPS).
    We numerically evaluate our methods on Electrical Impedance Tomography (EIT) and minimal surface control.
\end{abstract}

\section{Introduction}
\label{sec:intro}

First-order methods are slow.
To be precise, they require a high number of iterations, but if those iterations are fast, they have the chance to practically overpower second-order methods with expensive iterations.
In bilevel or PDE-constrained optimisation---the latter seen in this work as a special case of the former---the steps of basic first-order methods are very expensive, involving the solution of the inner problem or PDE and its adjoint.
To make first-order methods fast, it is, therefore, imperative to reduce the cost of solving these subproblems---for instance, by employing inexact solution schemes.

Consequently, especially in the machine learning community, an interest has surfaced in \term{single-loop} methods for bilevel optimisation; see \cite{suonpera2022bilevel} and references therein. Many of these methods are very specific constructions. In \cite{jensen2022nonsmooth} we started work on a more general approach to PDE-constrained optimisation: we showed that on each step of an outer primal-dual optimisation method, we can take \term{single steps} of standard linear system splitting schemes for the PDE constraint and its adjoint, and still obtain a convergent method that is computationally significantly faster than solving the PDEs exactly.
The adjoint equation is used to compute differentials of the PDE or inner problem solutions with respect to its parameters.
In \cite{suonpera2024general} we then presented an approach to bilevel optimisation that allowed general inner and adjoint algorithms that satisfy certain \term{tracking inequalities}.
These were proved for standard splitting schemes for the adjoint equation, and for forward-backward splitting and the Primal-Dual Proximal Splitting (PDPS) of \cite{chambolle2010first} for the inner problem.
The overall analysis was still tied to bilevel optimisation in Hilbert spaces, with forward-backward splitting as the outer optimisation method.

\emph{Our goal, in this work, is to develop a general theory of optimisation methods for bilevel and, by extension, multilevel problems.}
This theory is based on the idea of algorithmic single-loop differential estimation, with a \emph{flexible choice of outer, inner, and adjoint algorithms}. The inner and adjoint methods only need to satisfy the aforementioned \emph{tracking inequalities}. The outer method only has to tolerate \emph{controlled inexactness} of differentials involving the inner solution mapping.

\paragraph{Problem setup}

Write
\[
    F=J \circ S_u
\]
for a mapping $S_u: X \to U$ and a differentiable function $J: U \to \R$, on normed spaces $X$ and $U$.
Typically, but not necessarily, $S_u(x)$ is an \emph{inner solution mapping} that arises from the satisfaction of
\begin{equation}
    \label{eq:intro:tsu}
    0 = T(S_u(x), x)
    \quad\text{for a}\quad
    T: U \times X \to W_*
    \quad\text{with}\quad W_* \text{ a normed space}.
\end{equation}
The idea is that $T$ encodes the optimality conditions of an inner problem, parametrised by $x$, or a PDE, likewise parametrised by $x$.
We are then interested in the solution of composite optimisation problems of the form
\begin{equation}
    \label{eq:intro:simple-problem}
    \min_{x \in X}~ F(x) + G(x),
\end{equation}
or, more generally, the solution of optimality conditions
\begin{equation}
    \label{eq:intro:oc}
    0 \in F'(x) + \subdiff G(x) + \Xi x,
\end{equation}
for $G$ convex but possibly nonsmooth, and $\Xi \in \linear(X; X^*)$ skew-adjoint.
If $\Xi=0$, then this optimality condition is typically necessary for  \eqref{eq:intro:simple-problem}.
More generally, the operator allows the modelling of primal-dual problems, and treating the PDPS and Douglas–Rachford splitting as generalised forward-backward splitting methods \cite{clason2020introduction,tuomov-proxtest}.
We will discuss such formulations in more detail in \cref{sec:fb}.

\paragraph{Contributions}

Our contributions are as follows.
In \cref{sec:tracking,sec:inner-adjoint}, which form our \emph{inner theory},
\begin{enumerate}[label=(\alph*)]
    \item we show in \emph{general normed spaces} that we can approximate in a single-loop fashion the differentials of compositions $F = J \circ S_u$, given abstract inner and adjoint algorithms for $S_u$, satisfying certain \term{tracking inequalities}.
    We then illustrate how standard algorithms satisfy these inequalities.
\end{enumerate}
The corresponding results on sequences of scalars, on which these results are based on, are relegated to \cref{sec:scalar-tracking}.
To work in normed spaces, we use distances defined by semi-norms that satisfy the Pythagoras' identity, and corresponding support functions. We introduce these in \cref{sec:semicon}.

In contrast to \cite{suonpera2024general} and, indeed, all single-loop bilevel optimisation methods that we are aware of, our approach can also work with the adjoint dimension reduction trick typically employed in PDE-constrained optimisation.
We show that, subject to additive error terms with a bounded sum, the differential estimates $\estdiff F(\thisx)$ satisfy standard smoothness properties, such as Lipschitz differential and the two- and three-point descent inequalities \cite{tuomov-proxtest,clason2020introduction}.
Based on this, in \cref{sec:fb}, which forms our \emph{outer theory},
\begin{enumerate}[resume*]
    \item we prove various forms of convergence of general inexact splitting methods for \eqref{eq:intro:oc}.
\end{enumerate}

To facilitate the analysis of outer primal-dual methods, and even the basic forward-backward splitting when growth properties have to be combined from multiple sources, we introduce in \cref
{sec:operator-reg} operator-relative variants of the descent inequality.
We interpret the conditions of \cref{sec:fb} for outer forward-backward splitting and outer PDPS in \cref{sec:outer-examples}.

Not content to merely adapt existing proofs to inexact steps and normed spaces,
we also present some improvements to the nonconvex PDPS of \cite{tuomov-nlpdhgm}
(see also the review \cite{tuomov-firstorder}).
Treating a slightly simplified problem,
\begin{enumerate}[resume*]
    \item we show that, for the nonconvex PDPS, the values of the convex envelope of the objective function at ergodic iterates locally converge to a minimum.
    Through a more refined analysis, we are also able to avoid the requirement of dual strong monotonicity (e.g., Moreau–Yosida regularisation of total variation) of previous works \cite{jensen2022nonsmooth,tuomov-firstorder}.
\end{enumerate}

We finish with numerical illustrations and new application examples (electrical impedance tomography and minimal surface control) in \cref{sec:numerical}, confirming and further improving upon the related results in \cite{jensen2022nonsmooth,suonpera2024general}. Through our work, the specific algorithms presented in those works can be understood through a clean and generic differential estimation approach.

Some readers may wish to use \cref{sec:numerical} as a guide to this work, skipping to it after this introduction, and reading other parts as needed.

\paragraph{Examples}

The following examples illustrate the kind of algorithms that can be constructed, and whose convergence can be proved using our general theory. Our overall theory is, however, much more general than these two specific algorithms, as it allows the overall method to be composed of arbitrary inner, adjoint, and outer algorithms. Instead of gradients in Hilbert spaces, we will, moreover, work with differentials in general normed spaces.

\begin{example}[Forward-backward–linear system splitting–forward-backward]%
    \label{ex:intro:bilevel}
    Consider the bilevel optimisation problem
    \[
        \min_{x, u} J(u) + G(x) \quad\text{subject to}\quad u \in \argmin_x f(u; x) + g(u; x).
    \]
    Assume for simplicity that $f$ and $g$ are twice differentiable, convex in the first variable, and all the spaces are Hilbert. Then we can rewrite the problem as \cref{eq:intro:simple-problem,eq:intro:tsu} by setting
    \[
        T(u, x)=\grad_u f(u; x) + \grad_u g(u; x).
    \]
    A standard forward backward method with step length parameter $\tau>0$ would update
    \[
        \nextx \defeq \prox_{\tau G}(\thisx - \grad F(\thisx)),
    \]
    however the computation of $\grad F(x) = \grad S_u(x)^*\grad J(S_u(x))$ is generally expensive.
    We therefore propose to form the iteratively updated single-loop differential estimate $\nextestgrad$, as in the overall algorithm
    \begin{equation}
        \label{eq:intro:method}
        \left\{
        \begin{aligned}
        \nextu & \defeq \prox_{\sigma \grad g(\freevar; \thisx)}(\thisu - \sigma \grad f(\thisu; \thisx)),
        & \text{(inner step)}
        \\
        \nexxt w & \defeq - \inv N(M\this w + \grad J(\nextu))
            \quad\text{where}\quad N+M=\grad_u T(\nextu, \thisx),
        & \text{(adjoint step)}
        \\
        \nextestgrad & \defeq \grad_x T(\nextu, \thisx)\nexxt w,
        & \text{(diff.~transform)}
        \\
        \nextx & \defeq \prox_{\tau G}(\thisx - \nextestgrad).
        & \text{(outer step)}
        \end{aligned}
        \right.
    \end{equation}
    The first line is a standard forward-backward step for the inner variable.
    The second line applies, e.g., Gauss–Seidel or block-Gauss–Seidel splitting (depending on the choice of the operators $N$ and $M$) to the linear equation
    \begin{equation*}
        \grad_u T(\nextu, \thisx)\nexxt w + \grad J(\nextu) = 0.
    \end{equation*}
    Of course, the choice $M=0$ and $N = \grad_u T(\nextu, \thisx)$ would solve this equation exactly.
    The third line of \eqref{eq:intro:method} transforms the inner and adjoint iterates $\nextu$ and $\nexxt w$ into the differential estimate $\nextestgrad$.
    As we later elaborate in  \cref{ex:tracking:pde-adjoint}, this construction is motivated by the reduced adjoint equation
    \[
        \grad_u T(S_u(x), x)w_x + \grad J(S_u(x)) = 0
        \quad \text{whose solution satisfies}\quad
        \grad F(x) = \grad_x T(S_u(x), x)w_x.
    \]
    We treat specific inner and adjoint methods and differential transformations in \cref{sec:inner-adjoint} after the general theory in \cref{sec:tracking}.
    The fourth line of  \eqref{eq:intro:method} is an outer inexact forward-backward step.
    We present a general theory of outer/inexact methods in \cref{sec:fb}, and provide examples in \cref{sec:outer-examples}.
\end{example}

\begin{example}[A simple PDE-constrained optimisation problem]
    \label{ex:intro:pde}
    With $U \subset H^1(\Omega)$ and $X \subset L^2(\Omega)$, define for admissible $x \in X$ the linear operator $A_x \in \linear(U; U^*)$ by $A_x u \defeq \grad^* [x \cdot \grad u]$, and consider the PDE-constrained optimisation problem
    \[
        \min_{x, u} J(u) + G(x) \quad\text{subject to}\quad A_x u = b,
    \]
    where we use $G$ to restrict $x$ to a domain where the PDE is well-posed.
    Setting $T(u, x) = A_x u - b$, we are again in the setup of \cref{eq:intro:simple-problem,eq:intro:tsu}.
    Observing that $\grad_u T(\nextu, \thisx) = A_x$ and $\grad_x T(\nextu, \thisx)\nexxt w = \grad\nextu \cdot \grad\nexxt w$, similarly to \cref{ex:intro:bilevel}, but this time also applying a linear system splitting scheme to the PDE itself, we arrive at the algorithm
    \[
        \left\{
        \begin{aligned}
        \text{split } & N + M \defeq A_{\thisx},
        & \text{(operator splitting)}
        \\
        \nextu & \defeq -\inv N(M\this u - b),
        & \text{(PDE step)}
        \\
        \nexxt w & \defeq - \inv N(M\this w + \grad J(\nextu)),
        & \text{(adjoint step)}
        \\
        \estgrad F(\thisx) & \defeq \grad\nextu \cdot \grad\nexxt w,
        & \text{(diff.~transform)}
        \\
        \nextx & \defeq \prox_{\tau G}(\thisx - \estgrad F(\thisx)).
        & \text{(outer step)}
        \end{aligned}
        \right.
    \]
    To use standard schemes such as Gauss–Seidel, we need to work in finite-dimensional subspaces. However, in infinite dimensions, the abstract splitting $A_x=N+M$ can applied on the block structure of $A_x$ in more complex problems.
\end{example}

\paragraph{Further related work}

Through our general approach to inexactness in \cref{sec:fb}, independent of the differential estimation theory of \cref{sec:tracking}, besides differential estimates for multilevel problems, we can model mismatched adjoints \cite{lorenz2023mismatch}, and difficult-to-solve-exactly optimality conditions in measure spaces \cite{tuomov-pointsource}. We also adopt the approach of \cite{tuomov-pointsource} to optimisation in normed spaces: instead of Bregman divergences, we construct an inner product structure with a self-adjoint $M \in \linear(X; X^*)$.
Our work is related to the study of gradient oracles for smooth convex optimisation in \cite{devolder2013firstorder}, and for nonconvex composite optimisation in \cite{dvurechensky2022gradient,nabou2024proximal}, both in finite-dimensional Euclidean spaces.
Based on sufficient descent and the Kurdyka–Łojasiewicz property, \cite{ochs2019unifying} also study inexact methods in $\R^n$.
Moreover, \cite{baraldi2022proximal} introduce approaches to control model inexactness in proximal trust region methods, and \cite{salehi2024adaptivelyinexactfirstordermethod} in non-single-loop gradient methods for bilevel optimisation.

\paragraph{Notation and basic concepts}

We write $\linear(X; Y)$ for the space of bounded linear operators between the normed spaces $X$ and $Y$, and $\Id$ for the identity operator.
$X^*$ stands for the dual space of $X$.
When $X$ is Hilbert, we identify $X^*$ with $X$.
We write $\iprod{x}{y}$ for an inner product, $\dualprod{x^*}{x}_{X^*,X}$ for a dual product.
The open ball in the standard norm of $X$  is denoted by $\openball(x,r)$.
We also write $M \ge N$ if $M-N$ is positive semi-definite.
We extensively use the vectorial Young's inequality
\[
    \dualprod{x^*}{x}_{X^*,X} \leq \frac{a}{2}\norm{x}_X^2 + \frac{1}{2a}\norm{x^*}_{X^*}^2 \qquad \text{for all } x\in X,\, x^* \in X^*\, a > 0.
\]
For sets $A$ and $B$, we often write $\dualprod{A-B}{x} \ge 0$, which means that $\dualprod{y-z}{x} \ge 0$ for all $y \in A$ and $z \in B$.

For $F: X \to \R$, we write $DF(x)$ for the Gâteaux and $F'(x) \in X^*$ for the Fr\'{e}chet derivative at $x$, if they exist. If $X$ is Hilbert, $\grad F(x) \in X$ stands for the Riesz representation of $F'(x)$, i.e., the gradient.
For partial derivatives, we use the notation $\diffwrt{F}{x}(u, x)$.
We also write $\sublev_c F \defeq \{ x \in X \mid F(x) \le c\}$ for the $c$-sublevel set.
With $\extR := [-\infty, \infty]$, for a convex $G: X \to \extR$, we write $\Dom G$ for the effective domain, $\subdiff G(x)$ for the subdifferential at $x$, and $G^*:X^*\to \extR$ for the Fenchel conjugate.
We write $F^*$ for the Fenchel conjugate of $F$.
When $X$ is a Hilbert space, we write $\prox_F$ for the proximal map and, with a slight abuse of notation, identify $\subdiff G(x)$ with the set of Riesz representations of its elements.

\section{Support functions of semi-norm balls}
\label{sec:semicon}

Let $X$ be a normed space.
We call $M \in \linear(X; X^*)$ \term{self-adjoint} if the restriction $M^*|X=M$, and \term{positive semi-definite} if $\dualprod{Mx}{x}_{X^*,X}\ge 0$ for all $x \in X$.
If both hold, $\norm{x}_{M} \defeq \sqrt{\dualprod{Mx}{x}}$ defines a semi-norm, which \emph{satisfies the Pythagoras' identity}
\begin{equation}
    \label{eq:norms:pythagoras}
    \dualprod{M(x-\tilde x)}{x-\bar x}_{X^*,X} = \frac{1}{2}\norm{x-\tilde x}_M^2 + \frac{1}{2}\norm{x-\bar x}_M^2 - \frac{1}{2}\norm{\tilde x - \bar x}_M^2
\end{equation}
familiar from Hilbert spaces.
We write $\openball_M(x,r)$ for the radius-$r$ open ball at $x$ in this semi-norm.

We will often equip the space $X$ of the outer iterate with such a semi-norm.
For example, $M$ may arise as the injection $H_0^1(\Omega) \hookrightarrow H^{-1}(\Omega)$, $x \mapsto \iprod{x}{\freevar}_{L^2(\Omega)}$.
As another example, in \cite{tuomov-pointsource}, a convolution construction is employed in spaces of measures.
Further examples arise from primal-dual methods, and the applications of \cref{sec:numerical}.
The key idea of using such operators in \cref{sec:fb} will, indeed, be to define the steps of outer algorithms in non-Hilbert spaces, while also encoding variable decoupling in primal-dual methods.

To motivate the construction ahead of time, the $M$-proximal map of $G: X \to \extR$ may be defined as
$
    \prox_G^M(x) = \argmin_{\tilde x \in X} G(\tilde x) + \frac{1}{2}\norm{\tilde x-x}_M^2,
$
For convex, proper, and lower semicontinous $G$, this is characterised by the Fermat principle $0 \in \subdiff G(\tilde x) + M(\tilde x - x)$ in $X^*$.
When $M$ is the trivial injection $H_0^1(\Omega) \hookrightarrow H^{-1}(\Omega)$ from above, the idea then is that we may have to work in $H_0^1(\Omega)$ for reasons of regularity, but want to perform $L^2$ steps for reasons efficiency or simplicity. This is achieved by the choice of $M$---subject to the proximal map nevertheless being solvable with $\tilde x \in H_0^1(\Omega)$.

Due to the various properties of the next lemma---many shared by arbitrary semi-norms---we will then want to measure distances in $X^*$ through the support function of the $M$-unit ball,
\[
    \norm{x^*}_{M,*} \defeq \sup\{ \dualprod{x^*}{x}_{X^*,X} \mid  x \in X,\, \norm{x}_M \le 1 \}.
\]
In particular, we will measure distances of differentials $F'(x) \in X^*$ through this construct.
Its possible infinite-valuedness will impose additional regularity on the differentials.

 \begin{lemma}
    \label{lemma:semicon:props}
    For any semi-norm $\norm{\freevar}_{\circ}$ on $X$, and $\norm{x^*}_{*} \defeq \sup_{ \norm{x}_{\circ} \le 1 } \dualprod{x^*}{x}_{X^*,X}$ on $X^*$, we have:
    \begin{enumerate}[label=(\roman*)]
        \item\label{item:semicon:props:triangle}
        $\norm{\freevar}_*: X^* \to \extR$ is non-negative, positively homogeneous, and satisfies the triangle inequality.
        \item\label{item:semicon:props:conjugacy}
        $\norm{x^*}_*  = [2(\tfrac{1}{2}\norm{\freevar}_\circ^2)^*(x^*)]^{1/2}$,
        i.e.,  $\frac{1}{2}\norm{x^*}_*^2 = (\tfrac{1}{2}\norm{\freevar}_\circ^2)^*(x^*)$ for all $x^* \in X^*$.
        \item\label{item:semicon:props:young}
        $\dualprod{x^*}{x}_{X^*,X} \le \frac{1}{2}\norm{x^*}_*^2 + \frac{1}{2}\norm{x}_\circ^2$ for all $x \in X$ and $x^* \in X^*$.
    \end{enumerate}
    For a self-adjoint positive semi-definite $M \in \linear(X; X^*)$, also
    \begin{enumerate}[resume*]
        \item\label{item:semicon:props:bound}
        $\norm{x^*}_{X^*} \le \norm{M}_{\linear(X;X^*)}^{1/2} \norm{x^*}_{M,*}$
        for all $x^*\in X^*$.
        \item\label{item:semicon:props:range}
        $\norm{x^*}_{M,*} = \norm{x}_M$ for all $x^* = M x \in \range M$.
    \end{enumerate}
    Moreover, letting $A \in \linear(X; Y)$ for a normed space $Y$, and defining $\norm{A}_{\circ,Y} \defeq \sup_{\norm{x}_\circ \le 1} \norm{Ax}_Y$,
    \begin{enumerate}[resume*]
        \item\label{item:semicon:props:op}
        $\norm{y^* A}_* = \norm{A^* y^*}_* \le \norm{A}_{\circ,Y}\norm{y^*}_{Y^*}$ for all $y^* \in Y^*$.
    \end{enumerate}
 \end{lemma}

 \begin{proof}
    \cref{item:semicon:props:triangle}:
    This is immediate from the common properties of support functions.

    \cref{item:semicon:props:conjugacy}:
    Consider the problem
    $\sup_{x \in X}~ \dualprod{x^*}{x}_{X^*,X} - \frac{1}{2}\norm{x}_\circ^2$.
    Taking $x=t \bar x$ for $t \in \R$ and $\bar x \ne 0$, and optimising with respect to $t$, we obtain $t=\dualprod{x^*}{\bar x}_{X^*,X} / \norm{\bar x}_\circ^2$. Hence, as required
    \[
        \begin{split}
        \left(\frac{1}{2}\norm{\freevar}_\circ^2\right)^*(x^*)
        &
        = \sup_{x \in X}~ \dualprod{x^*}{x}_{X^*,X} - \frac{1}{2}\norm{x}_\circ^2
        = \sup_{t \in \R, \norm{\bar x}_\circ \le 1}~ t\dualprod{x^*}{\bar x}_{X^*,X} - \frac{t^2}{2}\norm{\bar x}_\circ^2
        = \sup_{\norm{\bar x}_\circ \le 1} \frac{\dualprod{x^*}{\bar x}_{X^*,X}^2}{2\norm{\bar x}_\circ^2}
        \\
        &
        = \frac{1}{2}\left(\sup_{\norm{\bar x}_\circ \le 1} \frac{\dualprod{x^*}{\bar x}_{X^*,X}}{\norm{\bar x}_\circ}\right)^2
        = \frac{1}{2}\left(\sup_{\norm{\bar x}_\circ \le 1} \dualprod{x^*}{x_0}_{X^*,X}\right)^2
        = \frac{1}{2}\norm{x^*}_*^2.
        \end{split}
    \]

    \cref{item:semicon:props:young}:
    Due to \cref{item:semicon:props:conjugacy}, this claim is exactly the Fenchel–Young inequality; see, e.g., \cite[(5.1)]{clason2020introduction}.

    \cref{item:semicon:props:bound}:
    We have $f^* \le g^*$ for $f \ge g$ from the definition of the Fenchel conjugate. Abbreviating $m \defeq \norm{M}_{\linear(X;X^*)}$, moreover
    $
        \norm{x}_X^2 = \sup_{\norm{x^*}_{X^*} \le 1} \norm{x}_X\dualprod{x}{x^*} \ge \dualprod{x}{Mx}/m = \norm{x}_M^2/m.
    $
    (For the inequality, take $x^*=Mx/(m\norm{x}_X)$.)
    Hence, also using the scaling and dual-norm properties \cite[Lemma 5.8 and Lemma 5.5]{clason2020introduction} of Fenchel conjugates, it follows for any $x^* \in X^*$ that
    \[
        (\tfrac{1}{2}\norm{\freevar}_M^2)^*(x^*)
        \ge
        (\tfrac{m}{2}\norm{\freevar}_X^2)^*(x^*)
        =
        m(\tfrac{1}{2}\norm{\freevar}_X^2)^*(x^*/m)
        =
        \frac{m}{2}\norm{x^*/m}_{X^*}^2
        =
        \frac{1}{2m}\norm{x^*}_{X^*}^2.
    \]
    We finish by using \cref{item:semicon:props:conjugacy}.

    \cref{item:semicon:props:range}:
    By the definition of the Fenchel conjugate,
    $
        \left(\frac{1}{2}\norm{\freevar}_M^2\right)^*(x^*) = \sup_{x \in X} \dualprod{x^*}{x}_{X^*,X} - \frac{1}{2}\norm{x}_M^2.
    $
    By the Fermat principle, if $x^* = Mx$, i.e., $x^* \in \range M$, then the supremum is achieved by $x$. In this case, the expression gives $\left(\frac{1}{2}\norm{\freevar}_M^2\right)^*(x^*) = \frac{1}{2}\norm{x}_{M}^2$. Now we apply \cref{item:semicon:props:conjugacy}.

    \cref{item:semicon:props:op}:
    By definition
    $
        \norm{A^*y^*}_*
        = \sup_{\norm{x}_\circ \le 1} \dualprod{y^*}{Ax}_{Y^*,X}
        \le \norm{y^*}_{Y^*}\norm{A}_{\circ,Y}.
    $
    Moreover---as a mere matter of notation---for any $x \in X$, we have $[A^*y^*]x = \dualprod{A^*y^*}{x}_{X^*,X} = \dualprod{y^*}{Ax}_{Y^*,Y}= [y^*A]x$. Thus, $A^*y^*=y^*A$.
\end{proof}

By \cref{item:semicon:props:range}, if $M$ is invertible, so that $\norm{\freevar}_M$ is a norm, then $\norm{x^*}_{M,*}=\norm{x^*}_{\inv M}$. If $M$ is the injection $H^1_0(\Omega) \hookrightarrow H^{-1}(\Omega)$,  $x \mapsto \iprod{x}{\freevar}_{L^2(\Omega)}$, then for $x^* \in L^2(\Omega)$, we have $\norm{x^*}_{M,*}=\norm{x}_M=\norm{x}_{L^2(\Omega)}$.
In this case, given $A \in \linear(H_0^1(\Omega); Y)$, $\norm{A}_{M,Y} = \sup_{x \in H_0^1(\Omega), \norm{x}_{L^2(\Omega)} \le 1} \norm{Ax}_Y$ evaluates the norm of the extension of $A$ into an operator from $L^2(\Omega)$ to $Y$. It may be infinite.

Thus, $\norm{\freevar}_{M,*}$ is also, almost, a semi-norm, but may take the value $+\infty$ outside $\range M$.

\section{Differential estimation from tracking inequalities}
\label{sec:tracking}

Let $J: U \to \R$ and $S_u: X \to U$ be Fréchet differentiable on normed spaces $X$ and $U$.
We consider
\[
    F(x) = J(S_u(x)).
\]
Typically, but not necessarily in the overall theory, $S_u(x)$ arises from the satisfaction of \eqref{eq:intro:tsu}.

As $S_u$ and its differential can be expensive to compute, given an iterate $\thisx$ of an arbitrary \emph{outer algorithm} for minimising an objective that involves $F$, such as \eqref{eq:intro:simple-problem}, we estimate $S_u(\thisx)$ by an \term{inner iterate} $\nexxt u \in U$, and $S_u'(\this x)$ by an \term{adjoint iterate} $\nexxt p \in \linear(X; U)$, that is, we estimate
\[
    F'(\this x) = J'(S_u(\thisx))S_u'(\this x)
    \quad\text{by}\quad
    \nextestdiff
    =
    J'(\nexxt u)\nexxt p.
\]
When $X$ is a Hilbert space, we write $\nextestgrad$ for the Riesz representation of $\nextestdiff$.
We do not provide a single explicit formula for $\nexxt u$ and $\nexxt p$.
Instead, we assume them to satisfy \emph{tracking estimates} as in \cite{jensen2022nonsmooth,suonpera2024general}.
We formulate these tracking estimates---that are essentially contractivity estimates with suitable penalties for parameter change---in \cref{sec:assumptions}.
Our goal is to derive, in \cref{sec:smoothness}, variants of standard descent inequalities and Lipschitz bounds for the estimate $\nextestdiff$.
In \cref{sec:inner-adjoint} we provide examples of \emph{inner and adjoint methods} that satisfy the tracking inequalities.

Although $\nextestdiff$ will have the above structure, we want to avoid constructing $\nexxt p \approx S_u'(\this x) \in \linear(X; U)$ directly due to its high dimensionality. Instead, we seek to only construct the necessary projections through a lower-dimensional variable $\nexxt w$.
We illustrate this idea in the following example.

\begin{example}[Adjoint equations]
    \label{ex:tracking:pde-adjoint}
    Suppose $S_u(x)$ arises from the satisfaction of \eqref{eq:intro:tsu}.
    By implicit differentiation, subject to sufficient differentiability and \eqref{eq:intro:tsu} holding in a neighbourhood of $x$, we obtain the \term{basic adjoint}
    \begin{equation}
        \label{eq:tracking:basic-adjoint}
        \diffwrt{T}{u}(S_u(x), x)S_u'(x) + \diffwrt{T}{x}(S_u(x), x)
        = 0 \in \linear(X; W),
    \end{equation}
    where
    $S_u'(x) \in \linear(X; U)$,
    $\diffwrt{T}{u}(S_u(x), x) \in \linear(U; W)$, and
    $\diffwrt{T}{x}(S_u(x), x) \in \linear(X; W)$.
    Hence, following the derivation of adjoint PDEs in, e.g., \cite[§1.6.2]{hinze2009pde} or \cite[§1.2]{delosreyes2015numerical}, assuming $\diffwrt{T}{u}(S_u(x), x)$ to be invertible, we solve from \eqref{eq:tracking:basic-adjoint} that
    \begin{gather}
         \label{eq:tracking:reduced-adjoint-diff-transformation}
         [J \circ S_u]'(x) = J'(S_u(x))S_u'(x) = S_w(x) \diffwrt{T}{x}(S_u(x), x),
    \intertext{for a $S_w(x) \in W$ satisfying the \term{reduced adjoint}}
        \label{eq:tracking:reduced-adjoint}
        S_w(x)\diffwrt{T}{u}(S_u(x), x) + J'(S_u(x)) = 0.
    \intertext{For $x=\thisx$, we will in practise take $\nexxt w$ as an operator splitting approximation to}
        \label{eq:tracking:reduced-adjoint-update}
        \nexxt w \diffwrt{T}{u}(\nexxt u, \thisx) + J'(\nextu) = 0,
    \shortintertext{and then set}
        \nonumber
        \nextestdiff \defeq \nexxt w \diffwrt{T}{x}(\nexxt u, \thisx) \approx J'(S_u(\thisx))S_u'(\thisx).
    \end{gather}
\end{example}

\subsection{The tracking assumption}
\label{sec:assumptions}

To track the inexact computations of inner and adjoint variables across iterations, we introduce verifiable conditions that quantify how closely the computed values follow the outputs of the exact inner and adjoint solution mappings evaluated at the current outer iterate. These tracking assumptions ensure that the accumulated errors remain controlled and that the approximate gradient remains meaningful for descent. The following assumption formalises this idea.

\begin{assumption}
    \label{ass:tracking:main}
    Let the abstract spaces $U$ and $W$ be equipped with the distance functions $d_U: U \times U \to [0, \infty]$ and $d_W: W \times W \to [0, \infty]$.
    Also let the normed space $X$ be equipped with a semi-norm $\norm{\freevar}_\circ$, and $X^*$ with the support function $\norm{\freevar}_*$ of the corresponding unit ball (see \cref{sec:semicon}).
    Finally, let the subset $\localset \subset X$, the \term{inner solution map} $S_u: X \to U$, and the \term{adjoint solution map} $S_w: X \to W$.
    Then, on an iteration $k \ge 0$, given $\{(x^n, u^n, w^n)\}_{n=0}^k \subset \localset \times U \times W$:
    \begin{enumerate}[label=(\roman*)]
        \item\label{item:tracking:main:inner-tracking}
        An \term{inner algorithm} produces $\nextu \in U$ satisfying for some $\pi_u>0$, $\kappa_u>1$, \emph{when $k \ge 1$}, the inner tracking inequality
        \[
            \kappa_u \InstnextDistU
            \le
            \InstthisDistU
            + \pi_u\InstthisDistXprev.
        \]

        \item\label{item:tracking:main:adjoint-tracking}
        An \term{adjoint algorithm} produces $\nexxt w \in W$ satisfying for some ${\primaldifffact}, \pi_w>0$, $\kappa_w>1$, \emph{when $k \ge 1$}, the adjoint tracking inequality
        \[
            \kappa_w \InstnextDistW
            \le
            \InstthisDistW
            + {\primaldifffact} \InstnextDistU
            + \pi_w\InstthisDistXprev.
        \]

        \item\label{item:tracking:main:differential-transformation}
        A \term{differential transformation} produces $\nextestdiff \in X^*$ that satisfies for some $\alpha_u,\alpha_w \ge 0$ the bound
        \[
            \trackingErrorThis
            \le
            \alpha_u \InstnextDistU
            + \alpha_w \InstnextDistW.
        \]
    \end{enumerate}
\end{assumption}

\begin{remark}[Distance functions]
    The distance functions $\dU$ and $\dW$ will in \cref{sec:fb} be given by norms, but in this section, we make no such requirement.
    They could be distances on a manifold.
    They could be squared norms, as the basic theory was formulated in \cite{jensen2022nonsmooth}.
    We write squared distances as $d_U^2(u, \tilde u) \defeq d_U(u, \tilde u)^2$.
\end{remark}

\begin{remark}[Meaning of the conditions]
    The inner and adjoint tracking conditions \cref{item:tracking:main:inner-tracking,item:tracking:main:adjoint-tracking}, which are not required to hold on the initial iteration $k=0$, are parameter change aware contractivity conditions for the inner and adjoint algorithms:
    if $\thisx=\prev x$, the former reduces to a standard contractivity condition.
    For examples of such algorithms, recall \cref{ex:intro:bilevel,ex:intro:pde}.
    The condition \cref{item:tracking:main:differential-transformation} allows transforming the construction error of $\nextestdiff$ into the tracking errors of the inner and adjoint algorithms.
    We provide more detailed examples of algorithms that satisfy these conditions in \cref{sec:inner-adjoint}.
\end{remark}

\begin{remark}[Initial iterates and first-step errors]
    The initial iterates are $x^0$, $u^0$, and $w^0$, but the tracking inequalities start from $x^0$, $u^1$, and $w^1$: the idea is that each $u^{k+1}$, which is computed using $x^k$, is compared against the exact inner solution $S_u(x^k)$, and likewise for the adjoint variable.
    The \emph{first-step errors} $\InstinitDistU$ and $\InstinitDistW$ will appear in our final estimates.
    They can be made small by solving the first step to a high precision.
    Contractive algorithms guarantee that they are smaller than the initialisation errors $\dU(u^0, S_u(x^0))$ and $\dW(w^0, S_w(x^0))$, indeed, this follows if the inner and adjoint tracking conditions hold also for $k=0$ with $x^{-1}=x^0$.
\end{remark}

\subsection{Smoothness of differential estimates}
\label{sec:smoothness}

We now derive descent- and Lipschitz-type inequalities for the approximate differential ${\estdiff F(\this x)}$, extending these classical smoothness concepts to account for differential errors under the tracking framework.
Most of these result are straightforward consequences of \cref{ass:tracking:main} and the scalar recurrence results of \cref{sec:scalar-tracking}.

We recall that if $F'$ is $L$-Lipschitz, it then satisfies the \emph{descent inequality} \cite[Theorem 7.1]{clason2020introduction}
\begin{equation}
    \label{eq:weaker:descent}
    \dualprod{F'(\this x)}{x-\this x}_{X^*,X}
    \ge
    F(x)-F(\this x)
    - \frac{L}{2}\demoDistXsq{x}{\thisx}
    \quad\text{for all}\quad
    x, \thisx \in X.
\end{equation}
The next theorem establishes a bound on dual pairings of $\estdiff F(\thisx)-F'(\thisx)$.
If, for simplicity, $\dX=\norm{\freevar}_X$, then taking $\bar x=\thisx$ and $\tilde\gamma=\trackingressum$ in the theorem, and combining with the descent inequality \eqref{eq:weaker:descent}, we obtain the \emph{inexact descent inequality} at $x=\nextx$,
\begin{equation}
    \label{eq:weaker:descent:inexact}
    \dualprod{\estdiff F(\this x)}{\nextx-\this x}_{X^*,X}
    \ge
    F(\nextx)-F(\this x)
    - \frac{L + 2\trackingressum}{2}\demoDistXsq{\nextx}{\thisx}
    - \frac{1}{2\trackingressum}e_{p,k}.
\end{equation}
According to the theorem, the final error term has a bounded sum, which we will be able to manage in convergence proofs of optimisation methods (in \cref{sec:fb}).

In such convergence proofs, it is frequently convenient to use the \emph{three-point descent inequality} (see \cite[Corollaries 7.2 and 7.7]{clason2020introduction} for the convex case, or \cite[Appendix B]{tuomov-proxtest} for the non-convex case)
\begin{equation}
    \label{eq:weaker:descent3}
    \dualprod{F'(\this x)}{x-\bar x}_{X^*,X}
    \ge
    F(x)-F(\bar x)
    + \frac{\beta}{2}\demoDistXsq{x}{\bar x}
    - \frac{\lambda}{2}\demoDistXsq{x}{\thisx}
    \quad\text{for all}\quad
    x, \thisx, \bar x \in X,
\end{equation}
where $\beta \in \R$ models second-order growth, and $\lambda \ge 0$ models smoothness.
If, again, $\dX=\norm{\freevar}_X$, then combining the next theorem with this inequality, we obtain the inexact version at $x=\nextx$,
\begin{equation}
    \label{eq:weaker:descent3:inexact}
    \dualprod{\estdiff F(\this x)}{\nextx-\bar x}_{X^*,X}
    \ge
    F(\nextx)-F(\bar x)
    + \frac{\beta-\tilde\gamma}{2}\demoDistXsq{\nextx}{\bar x}
    - \frac{\trackingressum^2\inv{\tilde\gamma} + \lambda}{2}\demoDistXsq{\nextx}{\thisx}
    - \frac{1}{2\tilde\gamma}e_{p,k}.
\end{equation}

We write
\begin{equation}
    \label{eq:tracking:kappa}
    \kappa \defeq \min (\kappa_u, \kappa_w)
    \quad\text{and}\quad
    \overline\kappa \defeq \max(\kappa_u, \kappa_w).
\end{equation}

\begin{theorem}
    \label{thm:weaker:smoothness}
    Suppose \cref{ass:tracking:main} holds up to an iteration $k \in \N$.
    Let $\{x^n\}_{n=0}^k \subset \localset$, and pick  $p \in [1, \kappa)$.
    Then, for any $\tilde\gamma>0$ and $\bar x \in X$, we have
    \[
        \dualprod{\estdiff F(\this x)-F'(\thisx)}{\nextx - \bar x}_{X^*,X}
        \ge
        -\frac{\tilde\gamma}{2} \normalDistXsq{\nextx}{\bar x} - \frac{\trackingressum^2}{2\tilde\gamma}\InstnextDistXthisSq - \frac{1}{2\tilde\gamma}e_{p,k}
    \]
    for some $\trackingressum \ge 0$ and $e_{p,k} \in \R$ that satisfy
    \begin{equation*}
        \sum_{n=0}^k p^n e_{p,n}
        \le
            \Psi_p
            \defeq
            \frac{\InstinitDistUsq}{\pi_u} \bigg(\frac{\trackingressum\alpha_u\kappa}{\kappa-1} + \frac{\trackingressum\alpha_w{\primaldifffact}}{(\kappa-1)^2}\bigg)
            +
            \frac{\InstinitDistWsq}{\pi_w} \bigg(\frac{\trackingressum\alpha_w\kappa}{\kappa-1}\bigg)
    \end{equation*}
    and
    \begin{equation}
        \label{eq:tracking:ressum}
        \trackingressum
        \le
        \frac{(\alpha_u\pi_u+\alpha_w\pi_w)\kappa\overline\kappa}{p(\kappa-p)}
        + \frac{\alpha_w {\primaldifffact}\pi_u\overline\kappa}{p^2(\kappa-p)^2}.
    \end{equation}
\end{theorem}

\begin{proof}
    By \cref{lemma:semicon:props}\,\cref{item:semicon:props:young},
    \[
        \dualprod{\estdiff F(\this x) - F'(\this x)}{\nextx - \bar x}_{X^*,X}
        \ge
        - \frac{1}{2\tilde\gamma}\trackingErrorSq{\nextestdiff}{F'(\this x)} - \frac{\tilde\gamma}{2}\normalDistXsq{\nextx}{\bar x}.
    \]
    Let $\thisDistXprev \defeq \InstthisDistXprev$, $\thisDistU \defeq \InstthisDistU$,  $\thisDistW \defeq \InstthisDistW$, and $\scalarTrackingErrorThis=\trackingErrorThis$.
    Then \cref{ass:tracking:main} and $\{x^n\}_{n=0}^k \subset \localset$ imply \cref{ass:scalar-tracking:main} for the index $k$.
    Now an application of \cref{lemma:scalar-tracking:error-sum} establishes
    \[
        \trackingErrorThisSq
        \le
        \trackingressum^2 \InstnextDistXthisSq
        +
        e_{p,k},
    \]
    where $\trackingressum$ and $e_{p,k}$ defined in \cref{eq:scalar-tracking:ressum,eq:scalar-tracking:ek} satisfy the claimed bounds.
    Combining these two inequalities establishes our claim.
\end{proof}

\begin{remark}
    \label{rem:tracking:error-sum}
     $\Psi_p$ can be made arbitrarily small by taking good-quality first inner and adjoint steps.
     The principal penalty from subsequent inexact steps is, therefore, $\trackingressum$.
\end{remark}

It will also be useful to have a Lipschitz-type property.
Combining the next theorem with $F'$ being $L$-Lipschitz with respect to $\dXstar$ and $\dX$, we can get the Lipschitz property with error for $\estdiff F$,
\[
    \trackingError{\nextestdiff}{F'(x)} \le L\normalDistX{\thisx}{x} + \sqrt{\errLip{k}}
    \qquad (x \in X).
\]
According to the theorem, the error terms again have bounded sum, if the squares of residual terms $\InstnextDistXthis[k]$ from the tracking inequalities have a bounded sum.
Many standard convergence proofs, include the ones of \cref{sec:fb}, will generally ensure the latter.
The theorem also shows that, in this case, the distance between the inner and adjoint iterates and the inner and adjoint solutions goes to zero.

\begin{theorem}
    \label{thm:weaker:erroneous-lipschitz}
    Suppose \cref{ass:tracking:main} holds up to an iteration $k \in \N$, and that
    $\{x^n\}_{n=0}^k \subset \localset$.
    Then
    \begin{equation}
        \label{eq:weaker:erroneous-lipschitz}
        \trackingErrorThisSq
        \le
        \bigl(\alpha_u \InstnextDistU + \alpha_w \InstnextDistW\bigr)^2
        \le
        \errLip{k},
    \end{equation}
    where each $\errLip{k} \ge 0$ is such that
    \begin{equation}
        \label{eq:weaker:erroneous-lipschitz:error-bound}
        \sum_{n=0}^{k-1} \errLip{n}
        \le \Psi_1 + \trackingressum[1]\sum_{n=0}^{k-1}\InstnextDistXthisSq[n].
    \end{equation}
\end{theorem}

\begin{proof}
    Let $\thisDistXprev \defeq \InstthisDistXprev$, $\thisDistU \defeq \InstthisDistU$,  $\thisDistW \defeq \InstthisDistW$, as well as $\scalarTrackingErrorThis=\trackingErrorThis$.
    Then \cref{ass:tracking:main} and $\{x^n\}_{n=0}^k \subset \localset$ imply \cref{ass:scalar-tracking:main} for the index $k$.
    Now \cref{lemma:scalar-tracking:inner-product-error-estimate} with $p=1$ gives \eqref{eq:weaker:erroneous-lipschitz} for $\errLip{k} \defeq \breve e_{1,k}$ as defined in \eqref{eq:scalar-tracking:ek:0}.
    The properties $\errLip{k}$ are a consequence of \cref{lemma:scalar-tracking:error-one}.
\end{proof}

\section{Inner and adjoint algorithms}
\label{sec:inner-adjoint}

We next provide brief examples of inner and adjoint methods that satisfy the corresponding parts of \cref{ass:tracking:main}; hence iterative ways to construct differential estimates $\nextestdiff$ that satisfy the inexact smoothness-type properties of \cref{sec:smoothness}.
This section is largely based on \cite{suonpera2024general}, with some streamlined proofs, and some omitted proofs.
The framing of the proofs in our differential estimation framework, is obviously new.

As in \cref{ass:tracking:main}, we equip $X$ with the arbitrary semi-norm $\norm{\freevar}_\circ$, and $X^*$ with $\norm{\freevar}_*$ constructed in \cref{sec:semicon}.
However, unless otherwise indicated, we now fix
\begin{align*}
    \dU(u, \tilde u)
    =
    \norm{u - \tilde u}_U
    \quad\text{and}\quad
    \dW(w, \tilde w)
    =
    \norm{w - \tilde w}_W.
\end{align*}
Equivalent norms can be used, as long as any relevant factors are adjusted correspondingly.

We generally assume that $S_u$, and similarly $S_w$, is $\lipSu\bX$-Lipschitz within $\localset$, i.e, $\norm{S_u(x)-S_u(\tilde x)}_U \le \lipSu\bX(x, \tilde x)$ for $x, \tilde x \in \localset$.

\subsection{Inner algorithms and the inner tracking inequality}
\label{sec:inner-adjoint:inner}

\paragraph{Forward-backward splitting}

On a Hilbert space $U$ and a normed space $X$, consider the parametric inner problem
\[
    S_u(x) = \argmin_u f(u; x) + g(u; x)
\]
for $f$ and $g$ convex and proper and lower semicontinuous in $u$, with $f$ differentiable in $u$.
If also $g$ is differentiable in $u$, this is an instance of \eqref{eq:intro:tsu} with
\[
    T(u, x) = \grad f(u; x) + \grad g(u; x).
\]
If $g$ is nondifferentiable, we still obtain an instance of \eqref{eq:intro:tsu} by taking for any $\theta>0$ \cite[Lemma 6.22]{clason2020introduction}
\begin{equation}
    \label{eq:tracking:fb-t-alt}
    T(u, x) = \prox_{\theta g(\freevar; x)}(u - \theta \grad f(u; x)) - u.
\end{equation}

\begin{theorem}
    \label{thm:tracking:inner-fb}
    Suppose $\grad f(\freevar; x)$is  $L$-Lipschitz and $\gamma_f$(-strongly) convex, and $g(\freevar; x)$ is $\gamma_g$(-strongly) convex (and proper and lower semicontinuous), both uniformly in $x \in \localset$ for an $\localset \subset X$, where we allow $\gamma_f=0$ or $\gamma_g=0$ as long as $\gamma_f+\gamma_g>0$.
    If $S_u$, as defined above, is $\lipSu\bX$-Lipschitz in $\localset$, and $\theta L \le 2$ for a step length parameter $\theta>0$, then  \cref{ass:tracking:main}\,\cref{item:tracking:main:inner-tracking} holds for the inner forward-backward splitting updates
    \[
        \nextu \defeq \prox_{\theta g(\freevar; \thisx)}(\thisu - \theta \grad f(\thisu; \thisx))
    \]
    with
    $\kappa_u = \sqrt{(1+2\theta\gamma_g)/(1-2\theta\lambda\gamma_f)}>1$ for $\lambda = 1-\theta L/2 \in (0, 1]$.
\end{theorem}

\begin{proof}
    Let $k \in \N$ and $\thisx \in \localset$.
    Write for brevity $f = f(\freevar; \thisx)$ and $g = g(\freevar; \thisx)$, as nothing in the initial part of the proof depends on the parametricity.
    By the strong convexity, properness, and lower semicontinuity, $\hat u \defeq S_u(\thisx)$ exists; in fact, there exists $\hat q \in \subdiff g(\hat u)$ such that $0 = \hat q + \grad f(\hat u)$ \cite[Theorems 3.8, 4.2, 4.5 \& 4.14]{clason2020introduction}.
    Interpolating between the $\gamma_f$-strong monotonicity and the $\inv L$-cocoercivity of $\grad f$ (see \cite[Chapter 7]{clason2020introduction}) with $\lambda$, and using Young's inequality, we deduce
    \[
        \begin{split}
        \iprod{\grad f(\thisu)& - \grad f(\hat u)}{\nextu-\hat u}
        =
        \iprod{\grad f(\thisu) - \grad f(\hat u)}{\thisu-\hat u}
        +
        \iprod{\grad f(\thisu) - \grad f(\hat u)}{\nextu-\thisu}
        \\
        &
        \ge
        \lambda \gamma_f \norm{\thisu-\hat u}^2 + (1-\lambda)\inv L \norm{\grad f(\thisu) - \grad f(\hat u)}^2
        +  \iprod{\grad f(\thisu) - \grad f(\hat u)}{\nextu-\thisu}
        \\
        &
        \ge
        \lambda \gamma_f \norm{\thisu-\hat u}^2
        - \frac{1}{2\theta}\norm{\nextu-\thisu}^2.
        \end{split}
    \]
    Since $g(\freevar, \thisx)$ is $\gamma_g$-strongly monotone,
    $
        \iprod{\nexxt q - \hat q}{\nexxt u - \hat u} \ge \gamma_g\norm{\nextu-\optu}^2.
    $
    Summing these two inequalities, we obtain
    \begin{equation}
        \label{eq:inner:fb:growth}
        \iprod{\grad f(\thisx) + \nexxt q}{\nextu - \opt u} - \lambda \gamma_f\norm{\thisu-\hat u}^2
        \ge
        \gamma_g\norm{\nextu-\hat u}^2 - \frac{1}{2\theta}\norm{\nextu-\thisu}^2.
    \end{equation}
    Applying the implicit update $0 = \nexxt q + \grad f(\thisu) + \inv\theta(\nextu-\thisu)$ for some $\nexxt q \in  g(\nextu)$ and using Pythagoras' identity yields $(\frac{1}{2\theta}-\lambda\gamma_f)\norm{\thisu-\optu}^2 \ge \left(\frac{1}{2\theta} + \gamma_g\right)\norm{\nextu-\optu}^2$.
    Let now $k \ge 1$ with $\thisx, \prevx \in \localset$.
    Rearranging and applying the assumed Lipschitz continuity of $S_u$, we obtain the required
    \[
        \kappa_u
        \norm{\nextu-S_u(\thisx)}_U
        \le
        \norm{\thisu-S_u(\thisx)}_U
        \le
        \norm{\thisu-S_u(\prevx)}_U + \lipSu\InstthisDistXprev.
        \qedhere
    \]
\end{proof}

\begin{remark}[Lipschitz solution mapping]
    The Lipschitz assumption on $S_u$ is guaranteed in sufficiently smooth cases by the classical implicit function theorem applied to the equation $T(u, x)=0$; see \cite[Appendix B]{suonpera2024general}.
    Nonsmooth implicit function theorems and the Aubin or pseudo-Lipschitz property of the set-valued mapping $S_u$ are studied in, e.g., \cite{dontchev2014implicit,ioffe2017variational} as well as \cite[Theorem 28.3]{clason2020introduction}.
    If $S_u$ has the Aubin property, it will be Lipschitz if we assume, e.g., strict convexity to ensure the uniqueness of solutions.
    For the specific case $f(u; x)=\bar f(u)$ and $g(u; x)= x\bar g(u)$ with a scalar $x$, we refer to \cite[Theorem 28.5]{clason2020introduction}.
    A case where $g$ is a constraint is studied in \cite[Theorem 4.51]{bonnans2000perturbation}.
\end{remark}

\paragraph{Primal-dual proximal splitting}

On a Hilbert space $Z$ and a normed space $X$, consider the inner problem
\[
    \min_z f(z; x) + g^*(Kz; x).
\]
for $K \in \linear(Z; Y^*)$ linear and bounded to a Hilbert space $Y^*$, both $f$ and $g$ convex in the first parameter, and differentiable in both parameters.
As an instance of \eqref{eq:intro:tsu}, we represent the Fenchel–Rockafellar primal-dual optimality conditions of this problem as a root $u$ of the mapping
\[
    T(u, x) = (\grad f(z; x) + K^* y, \grad g(y; x) - Kz)
    \quad\text{where}\quad u=(z, y) \in U = Z \times Y.
\]

\begin{theorem}
    \label{thm:tracking:inner-pdps}
    Suppose $g(\freevar; x)$ and $f(\freevar; x)$ are $\gamma$-strongly convex uniformly in $x \in \localset$ for a $\localset \subset X$.
    If $S_u(x)=\inv T(\freevar; x)(0)$ is $\lipSu\bX$-Lipschitz in $\localset$, and the step length parameters $\sigma_p,\sigma_d>0$ satisfy $\sigma_p\sigma_d\norm{K} \le 1$, then \cref{ass:tracking:main}\,\cref{item:tracking:main:inner-tracking} holds for the inner PDPS updates
    \[
        \nextz = \prox_{\sigma_p f(\freevar; \thisx)}(\thisz - \sigma_p K^*\thisy)
        \quad\text{and}\quad
        \nexty = \prox_{\sigma_d g(\freevar; \thisx)}(\thisy + \sigma_d K(2\nextz-\thisz))
    \]
    with
    \[
        \dU(u, \tilde u) = \norm{u - \tilde u}_N
        \quad\text{for}\quad
        N = \begin{pmatrix}
            \inv\sigma_p \Id & -K^* \\
            -K & \inv\sigma_d \Id
        \end{pmatrix}.
    \]
\end{theorem}

The proof in \cite[Theorem 3.6]{suonpera2024general} is fundamentally similar to the forward-backward splitting in \cref{thm:tracking:inner-fb}, but requires working with operator-induced norms and monotone operators.

\paragraph{Linear system splitting}

We now cover algorithms for PDEs as the inner problems.
For $U, X$, and $W_*$ normed spaces, let both $A_x \in \linear(U; W_*)$, modelling a linear PDE parametrised $x$, and the right hand side $b_x \in W_*$ be Lipschitz in $x \in X$.
Consider the inner constraint of $u=S_u(x)$ satisfying
\begin{equation}
    \label{eq:tracking:linear-inner}
    A_x u = b_x.
\end{equation}
This is again an instance of  \eqref{eq:intro:tsu} when we set
\[
    T(u, x) = A_x u - b_x.
\]
Two solve \eqref{eq:tracking:linear-inner} inexactly and efficiently, we split $A_x$ into two components, as in standard Gauss–Seidel or Jacobi splittings.

\begin{assumption}[Admissible splitting]
    \label{ass:primal-admissible-splitting}
    Let $X, U$, and $W_*$ be normed spaces, and $A_x \in \linear(U; W_*)$ for all $x$ in a set $\localset \subset X$.
	We assume to be given splittings $A_x=N_x+M_x$ with $N_x$ invertible, satisfying
	$\kappa_u \norm{\inv N_x M_x}_{\linear(U; U)} \le 1$  for some $\kappa_u>1$, for all $x \in \localset$.
\end{assumption}

\begin{theorem}
    \label{thm:tracking:inner-linear-system-splitting}
    Suppose \cref{ass:primal-admissible-splitting} holds and that $S_u(x)=\inv A_x b_x$ is $\lipSu\bX$-Lipschitz in $\localset$.
    Then the updates $\nextu = \inv N_{\thisx}(b_{\thisx} - M_{\thisx}\thisu)$ satisfy \cref{ass:tracking:main}\,\cref{item:tracking:main:inner-tracking} in $\localset$
\end{theorem}

\begin{proof}
    Let $k \ge 1$ with $x^k, x^{k-1} \in \localset$.
    We have
    \[
        \nextu - S_u(\thisx)
        =\inv N_{\thisx}(b_{\thisx} - M_{\thisx}\thisu) - \inv N_{\thisx}(b_{\thisx} - M_{\thisx}S_u(\thisx))
        =-\inv N_{\thisx}M_{\thisx}(\thisu - S_u(\thisx)).
    \]
    Hence
    \[
        \kappa_u \norm{\nextu - S_u(\thisx)}_U
        \le
        \norm{\thisu - S_u(\thisx)}_U
        \le
        \norm{\thisu - S_u(\prev x)}_U
        + \lipSu \InstthisDistXprev.
        \qedhere
    \]
\end{proof}

To provide examples of \cref{ass:primal-admissible-splitting}, we extend the condition to splittings that are admissible for the adjoint equations that we treat next.

\begin{assumption}[Adjoint admissible splitting]
    \label{ass:adjoint-admissible-splitting}
    Let $X, U$, and $W_*$ be normed spaces, and $A_x \in \linear(W; U^*)$ for all $x$ in a set $\localset \subset X$.
	We assume to be given splittings $A_x=N_x+M_x$ with $N_x$ invertible and satisfying for some $\kappa_w > 1$ and $\gamma_N>0$ the bounds
	\begin{equation}
		\label{eq:convergence:splitting:split-condition}
		\kappa_w \norm{\inv N_x M_x}_{\linear(W; W)} \le 1
		\quad\text{and}\quad
		\gamma_N\norm{\inv N_x}_{\linear(W; U^*)} \le 1
		\quad\text{for all}\quad x \in \localset.
	\end{equation}
\end{assumption}

\begin{example}[Jacobi splitting]
	\label{ex:splitting:jacobi}
	Let $A_x \in \R^{n \times n}$, and take $N_x$ as the diagonal of $A_x$.
	We have $\kappa \norm{\inv N_x M_x}_{\linear(U; U)} \le 1$ for some $\kappa>1$ when $A_x$ is either strictly diagonally dominant \cite[§10.1]{golub1996matrix}.
	We have $\gamma_N\norm{\inv N_x} \le 1$ when the main diagonal of $A_x$ has only positive entries. Then $\gamma_N$ is the minimum of the diagonal values.
\end{example}

\begin{example}[Gauss--Seidel]
	\label{ex:splitting:Gauss–Seidel}
	Let $A_x \in \R^{n \times n}$, and take $N_x$ as the upper triangle and diagonal of $A_x$.
	We have $\kappa \norm{\inv N_x M_x}_{\linear(U; U)} \le 1$ for some $\kappa>1$ when $A_x$ is either strictly diagonally dominant or symmetric and positive definite \cite[§10.1]{golub1996matrix}.
	We have $\gamma_N\norm{\inv N_x} \le 1$ when $N_x$ is invertible, i.e., has no zeros on the main diagonal.
\end{example}

Several other splittings are possible.
The trivial splitting of $A_x \in (U; U)$ takes $N_x=\theta\Id$ for some step length parameter $\theta$, and $M_x=A-N_x$.
In \cite{suonpera2024general} a “block-Gauss–Seidel” scheme is employed on an operator that has easily invertible diagonal blocks.
Such a scheme could also be applied to a domain decomposition.

\subsection{Adjoint algorithms and the adjoint tracking inequality; the differential transformation condition}
\label{sec:inner-adjoint:adjoint}

We now treat adjoint methods and the differential transformation, i.e., \cref{ass:tracking:main}\,\cref{item:tracking:main:inner-tracking,item:tracking:main:adjoint-tracking}.
We concentrate on the reduced adjoint; the adjoint tracking inequality for the basic adjoint is treated in \cite{suonpera2024general}, and the differential transformation condition is proved similarly to the reduced adjoint.

With $X$, $U$, and $W_*$ normed spaces, let, thus, $S_u$ and $T$ satisfy \eqref{eq:intro:tsu}, and define
\[
    \nextestdiff \defeq \nexxt w \diffwrt{T}{x}(\nexxt u, \thisx)
\]
for $\nexxt w \in W$ computed by taking (single or multiple) admissible splitting steps (see \cref{ass:adjoint-admissible-splitting,ex:splitting:Gauss–Seidel,ex:splitting:jacobi}) on the linear equation
\[
    w \diffwrt{T}{u}(\nexxt u, \thisx) + J'(\nextu) = 0
\]
with unknown $w$.
Correspondingly, let $S_w$ arise from the reduced adjoint \eqref{eq:tracking:reduced-adjoint}.
The next theorem shows that the scheme satisfies \cref{ass:tracking:main}\,\cref{item:tracking:main:adjoint-tracking,item:tracking:main:differential-transformation} in case of single steps; multiple steps follow immediately with larger $\kappa_w$.

For the theorem, we recall from \cref{sec:semicon} the (possibly infinite-valued) norm-like construct $\norm{\freevar}_{\circ,W_*}$ on $\linear(X; W_*)$.
If our base seminorm $\norm{\freevar}_\circ$ on $X$ is just the standard norm on $X$, then also $\norm{\freevar}_* = \norm{\freevar}_{X^*}$ and $\norm{\freevar}_{\circ,W_*}=\norm{\freevar}_{\linear(X; W_*)}$. We will need the flexibility of these more general constructs in the application examples of \cref{sec:numerical}.

\begin{theorem}
    \label{thm:tracking:reduced-adjoint-splitting}
    Suppose that $T|U \times \localset$ and $S_u|\localset$ are Lipschitz-continuously differentiable for a $\localset \subset X$, and that the adjoint admissible splitting \cref{ass:adjoint-admissible-splitting} holds in $U \times \localset$ for the linear operators $A_{(u, x)} \in \linear(W; U^*)$, $w \mapsto w\diffwrt{T}{u}(u, x)$.
    Suppose, moreover, that $u \mapsto \diffwrt{T}{u}(u, x) \in \linear(U; W_*)$ is $L_{\diffwrt{T}{u};u}$-Lipschitz for all $x \in \localset$; that $J': U \to U^*$ is $L_{J'}$-Lipschitz; that $S_w: X \to W$ is $\lipSw\bX$-Lipschitz in $\localset$; and that
    \[
        N_{S_w} \defeq \sup\{ \norm{S_w(x)}_W \mid x \in \localset \}< \infty.
    \]
    Let
    \[
        \nexxt w \defeq -\inv N_{(\nexxt u, \thisx)}(J'(\nextu)+M_{(\nexxt u, \thisx)}\this w).
    \]
    Then \cref{ass:tracking:main}\,\cref{item:tracking:main:adjoint-tracking} holds in $\localset$ with
    $\mu_u = \kappa_w\inv\gamma_N(L_{\diffwrt{T}{u};u}N_{S_w} + L_J)$.

    Equipping $\linear(X; W_*)$ with the $\norm{\freevar}_{\circ,W_*}$ distance, suppose further that $u \mapsto \diffwrt{T}{x}(u, x) \in  \linear(X; W_*)$ is $L_{\diffwrt{T}{x};u}$-Lipshitz for all $x \in \localset$ and that
    \[
        C_{\diffwrt{T}{x}} \defeq \sup\{ \norm{\diffwrt{T}{x}(\nextu, \thisx)}_{\circ,W^*} \mid k \in \N \} < \infty.
    \]
    Then the differential transformation \cref{ass:tracking:main}\,\cref{item:tracking:main:differential-transformation} holds for $\nextestdiff \defeq \nexxt w \diffwrt{T}{x}(\nexxt u, \thisx)$ with $\alpha_u=N_{S_w}L_{\diffwrt{T}{x}; u}$ and $\alpha_w=C_{\diffwrt{T}{x}}$.
\end{theorem}

\begin{proof}
    To prove \cref{ass:tracking:main}\,\cref{item:tracking:main:adjoint-tracking}, let $k \ge 1$ with $x^k, x^{k-1} \in \localset$.
    For brevity, set $v=(\nextu, \thisx)$, $A_vw = w\diffwrt{T}{u}(v)$, and $b_v=-J'(\nextu)$.
    Likewise let $\bar v=(S_u(\thisx), \thisx)$, $A_{\bar v}w = w\diffwrt{T}{u}(\bar v)$, and $b_{\bar v}=-J'(S_u(\thisx))$.
    Assuming that $\thisx \in \localset$, we proceed as in \cite[Lemma 4.1]{suonpera2024general}:
    Observing that $A_{\bar v}S_w(\thisx)=b_{\bar v}$ and $A_vS_w(\thisx)=(N_{v}+M_{v})S_w(\thisx)$, we rearrange
    \[
        \begin{split}
        \nexxt w - S_w(\thisx)
        &
        =
        \inv N_v(b_v - M_v \this w) -  S_w(\this x)
        \\
        &
        =
        \inv N_v(b_v -b_{\bar v}) - \inv N_v (A_v - A_{\bar v})S_w(\thisx)
        - \inv N_vM_v(\this w - S_w(\thisx)).
        \end{split}
    \]
    \Cref{ass:tracking:main}\,\cref{item:tracking:main:adjoint-tracking} now follows from
    \[
        \begin{split}
        \kappa_w\norm{\nexxt w - S_w(\this x)}_W
        &
        \le
        \kappa_w\inv\gamma_N(L_{J'} + L_{\diffwrt{T}{u};u}\norm{S_w(\thisx)}_W)\norm{v-\bar v}_{U \times X}
        +
        \norm{\this w - S_w(\thisx)}_W
        \\
        &
        \le
        \kappa_w\inv\gamma_N(L_{J'} + L_{\diffwrt{T}{u};u}N_{S_w})\norm{\nextu - S_u(\thisx)}_U
        \\
        \MoveEqLeft[-1]
        +
        \norm{\this w - S_w(\prev x)}_W
        +
        \lipSw\InstthisDistXprev.
        \end{split}
    \]

    We can write
    \[
        \begin{split}
        \nextestdiff - F'(\thisx)
        &
        =
        \nexxt w \diffwrt{T}{x}(\nexxt u, \thisx) - S_w(\thisx)\diffwrt{T}{x}(S_u(\thisx), \thisx)
        \\
        &
        = [\nexxt w -S_w(\thisx)]\diffwrt{T}{x}(\nexxt u, \thisx) - S_w(\thisx)[\diffwrt{T}{x}(S_u(\thisx), \thisx)-\diffwrt{T}{x}(\nexxt u, \thisx)].
        \end{split}
    \]
    Recall from \cref{lemma:semicon:props} that $\norm{\freevar}_*$ satisfies the triangle inequality and an operator norm type inequality with respect to $\norm{\freevar}_{\circ, W_*}$.
    \Cref{ass:tracking:main}\,\cref{item:tracking:main:differential-transformation} then follows, for any $k \in \N$ with $x^k \in \localset$, from
    \[
        \begin{split}
            \norm{\nextestdiff - F'(\thisx)}_*
            &
            \le
            \norm{\diffwrt{T}{x}(\nexxt u, \thisx)}_{\circ,W_*}
                \norm{\nexxt w -S_w(\thisx)}_W
            \\
            \MoveEqLeft[-1]
            + \norm{S_w(\thisx)}_W \norm{\diffwrt{T}{x}(S_u(\thisx), \thisx)-\diffwrt{T}{x}(\nexxt u, \thisx)}_{\circ, W_*}
            \\
            &
            \le
            C_{\diffwrt{T}{x}} \norm{\nexxt w - S_w(\thisx)}_W
            + N_{S_w} L_{\diffwrt{T}{x};u}\norm{\nextu- S_u(\thisx)}_{U}.
            \qedhere
        \end{split}
    \]
\end{proof}

\section{Nonconvex forward-backward type methods with inexact updates}
\label{sec:fb}

We now need to prove the convergence of \emph{outer methods} for the overall problem \eqref{eq:intro:oc}, given estimates $\estdiff F(\thisx)$ of  $F'(\thisx)$ by \emph{inner and adjoint methods}, the latter two satisfying the tracking theory of \cref{sec:tracking}.
In this section, we do this through a convergence theory for \emph{general inexact forward backward-type methods} in a normed space $X$.
Our treatment encompasses primal-dual methods, seen as forward-backward methods
with respect to appropriate operator-relative (semi-)norms, discussed in the previous section.
We introduce such methods in \cref{sec:fb:methods}.
Then in \cref{sec:fb:growth} we introduce abstract growth conditions, which we will use in \cref{sec:fb:subdiff,sec:fb:non-escape,sec:fb:values,sec:fb:weak} to prove various forms of convergence.

Afterwards, in \cref{sec:outer-examples}, we will verify the growth inequalities for forward-backward and primal-dual algorithms that use the tracking theory of \cref{sec:tracking} for (single-loop) updates of an inner problem.

\subsection{General inexact forward-backward type methods}
\label{sec:fb:methods}

For proper $F, G: X \to \extR$, consider the problem
\begin{equation}
    \label{eq:fb:problem}
    \min_{x \in X} F(x) + G(x).
\end{equation}
In this subsection, and in the examples of \cref{sec:outer-examples}, $G$ will be convex and lower semicontinuous, and $F$ Fréchet differentiable, but the general theory of \cref{sec:fb:growth,sec:fb:subdiff,sec:fb:non-escape,sec:fb:values,sec:fb:weak} will make no such assumption.

For an initial $x^0$, if $X$ is Hilbert, the iterates $\{\this{x}\}_{k=1}^\infty$ of the basic inexact forward-backward method are generated for some step length parameter $\tau>0$ and an estimate $\estgrad F(\thisx)$ of $\grad F(\thisx)$ (not necessarily the one from \cref{sec:tracking}) by
\begin{equation}
    \label{eq:fb:alg-basic}
    \nexxt x \defeq \prox_{\tau G}(\this x - \tau \estgrad F(\this x)).
\end{equation}
In implicit form the method reads
\begin{equation}
    \label{eq:fb:fb-inexact}
    -\inv\tau(\nextx-\thisx) \in \estgrad F(\thisx) + \subdiff G(\nextx).
\end{equation}

We generalise this problem and method by considering for a skew-adjoint $\Xi\in \linear(X; X^*)$, i.e., $\Xi^*|X=-\Xi$, the problem of finding $x \in X$ satisfying
\begin{equation}
    \label{eq:fb:generic-problem}
    0 \in H(x) \defeq F'(x) + \subdiff G(x) + \Xi x.
\end{equation}
Recalling the preliminary discussion in \cref{sec:semicon}, we pick a self-adjoint and positive semi-definite \term{preconditioning operator} $M \in \linear(X; X^*)$.
We then intent to solve this problem with an implicit method of the rough form
\begin{equation}
    \label{eq:fb:implicit}
    -M(\nextx-\thisx) =: \totalest{k+1} \approxin F'(\thisx) + \subdiff G(\nextx) + \Xi\nextx.
\end{equation}
Here the approximate inclusion “$\approxin$” generalises the inexact gradient $\estgrad F(\this x)$ to other forms of inexactness. We will make the---for the moment---imprecise notion more precise through the growth inequalities of \cref{sec:fb:growth}.
Besides being essential for constructing primal-dual methods, as we will shortly see, $M$ allows working in non-Hilbert spaces in \cref{sec:eit} or \cite{tuomov-predict}, or avoiding Riesz representations in Hilbert spaces, such as $H^{-1/2}(\Omega)$, in \cref{sec:minsurf}.
We could generalise $M$ to a Bregman divergence, but choose simplicity of presentation.

Algorithms of the form  \eqref{eq:fb:implicit} with an exact inclusion for $\totalest{k+1}$, cover many common splitting algorithms, such as Douglas–Rachford splitting (DRS) and the primal-dual proximal splitting (PDPS) of \cite{chambolle2010first}; see \cite{clason2020introduction,tuomov-proxtest}.
As we will see in the following examples, with an inexact inclusion, besides the inexact gradients of \cref{sec:tracking}, the approach also covers inexact proximal maps and mismatched adjoints \cite{lorenz2023mismatch} in primal-dual methods.
Inexact proximal maps were used, e.g., in \cite{tuomov-predict} for point source localisation in measure spaces.

Indeed, for general $M$, \emph{there does not necessarily exist an exact solution $\nextx \in X$ to \eqref{eq:fb:implicit}}. If $\Xi=0$, subject to standard regularity conditions, exact \eqref{eq:fb:implicit} arises as the first-order optimality conditions of the surrogate model
\[
    \min_{x \in X}~ G(x) + \dualprod{F'(\thisx)}{x-\thisx} + \frac{1}{2}\norm{x-\thisx}_M^2.
\]
If $M$ does not provide the coercitivity for this problem to have a solution in $X$, and a solution is required, the existence has to be verified otherwise. The same applies when we replace $F'(\thisx)$ by an estimate $\estdiff F(\thisx)$.

The next example demonstrates how \cref{eq:fb:generic-problem,eq:fb:implicit} accommodate primal-dual methods.

\begin{example}[Primal-dual proximal splitting]
    \label{ex:fb:pdps}
    On normed spaces $Z$ and $Y$, let $g: Z \to \extR$ and $h: Y^* \to \extR$ be convex, proper, and lower semicontinuous, $f: Z \to \R$ possibly non-convex but Fréchet differentiable, and $K \in \linear(Z; Y^*)$.
    Suppose $h=(h_*)^*$ for some $h_*: Y \to \extR$, and consider the problem
    \begin{equation}
        \label{eq:fb:pdps:problem}
        \min_{z \in Z}~ f(z) + g(z) + h(Kz)
        =
        \min_{z \in Z} \max_{y \in Y}~ f(z) + g(z) + \dualprod{y}{Kz}_{Y,Y^*} - h_*(y).
    \end{equation}
    If $f$ is convex, subject to the standard condition on the existence of $x_0 \in \interior \Dom [h \circ K] \isect \Dom [f+g] \ne \emptyset$ with $Kx_0 \in \interior \Dom h$,\footnotemark[1] the Fenchel–Rockafellar theorem \cite[Theorem~5.11]{clason2020introduction} gives rise to the necessary and sufficient first-order primal-dual optimality conditions of the form \eqref{eq:fb:generic-problem},
    \begin{gather}
        \nonumber
        0 \in H(z, y) = \begin{pmatrix}
            \subdiff g(z) + f'(z) + K^*y \\
            \subdiff h_*(y) - Kz.
        \end{pmatrix}
        = F'(z, y) + \subdiff G(z, y) + \Xi(z, y),
    \shortintertext{where}
        \label{eq:fb:pdps:functions}
        F(z, y)= f(z),
        \quad
        G(z,y)=g(z) + h_*(y),
        \quad\text{and}\quad
        \Xi=\begin{pmatrix} 0 & K^* \\ -K & 0 \end{pmatrix}.
    \end{gather}
    If $f$ is nonconvex, the necessity can be shown through, e.g., Mordukhovich subdifferentials, and their compatibility with both convex subdifferentials and Fréchet derivatives; see, e.g., \cite{clason2020introduction}.

    Pick step length parameters $\tau, \sigma > 0$.
    With inexact gradients for $f$, the PDPS in Hilbert spaces reads
    \begin{equation}
        \label{eq:pdps:alg}
        \left\{
            \begin{array}{l}
                \nexxt z \defeq \prox_{\tau g}(\this z - \tau \estgrad f(\this z) - \tau K^*\thisy), \\
                \nexty \defeq \prox_{\sigma h_*}(\thisy + \sigma K(2\nexxt z - \this z)).
            \end{array}
        \right.
    \end{equation}
    When $f=j \circ S_u$ for $S_u$ a PDE solution operator, and we compute $\estgrad f$ following \cref{thm:tracking:inner-linear-system-splitting,thm:tracking:reduced-adjoint-splitting}, \eqref{eq:pdps:alg} becomes the algorithm presented in \cite{jensen2022nonsmooth}.

    To extend \eqref{eq:pdps:alg} to general normed spaces, we write it in $X = Z \times Y$ in the implicit form \eqref{eq:fb:implicit} as
    \begin{subequations}
    \label{eq:fb:pdps:implicit}
    \begin{gather}
        0 \in \estdiff F(\thisx) + \subdiff G(\nextx) + \Xi\nextx + M(\nextx-\thisx)
        \quad\text{for}
        \\
        \label{eq:fb:pdps:est-and-m}
        \estdiff F(\this z, \thisy) \defeq \begin{pmatrix}
            \estdiff f(\this z) \\
            0
        \end{pmatrix}
        \quad\text{and}\quad
        M \defeq
        \begin{pmatrix}
            \inv \tau M_z, & -K^* \\
            -K & \inv\sigma M_y
        \end{pmatrix}
    \end{gather}
    \end{subequations}
    for some self-adjoint positive semi-definite $M_z \in \linear(Z; Z^*)$ and $M_y \in \linear(Y; Y^*)$.
    For standard proximal maps in Hilbert spaces, we take $M_z: Z \hookrightarrow Z^*$ and $M_y: Y \hookrightarrow Y^*$ as the standard injections, $z \mapsto \iprod{z}{\freevar}_Z \in Z^*$.
    In that case, $M$ is self-adjoint and positive semi-definite when $\tau\sigma\norm{K}^2 \le 1$, while the treatment of exact forward steps with respect to $f$ requires\footnotemark[2]~$\tau L + \tau\sigma\norm{K}^2 \le 1$ for $L$ the Lipschitz factor of $f'$ \cite{tuomov-proxtest,clason2020introduction,he2012convergence}.
\end{example}

\footnotetext[1]{Several relaxations are possible, include using the relative interior, or the formulas of \cite{attouch1986dualitysumconvex}.}

\footnotetext[2]{This is the requirement for gap estimates; for iterate estimates $L/2$ in place of $L$ is sufficient. In \cite{yan2024improved} an overall factor $4/3$ improvement is shown through an analysis that involves historical iterates.}

\subsection{Inexact growth inequalities}
\label{sec:fb:growth}

We now provide several alternative assumptions that give a precise meaning to the approximate inclusion in \eqref{eq:fb:implicit}.
For this, we first define the \term{Lagrangian gap functional}
\begin{equation}
    \label{eq:fb:general:gap}
    \gap(x; \basex) \defeq [F+G](x) - [F+G](\basex) - \dualprod{\Xi x}{\basex}_{X^*,X}.
\end{equation}

\begin{example}
    \label{ex:fb:fb:gap}
    For forward-backward splitting, $\gap(x; \basex) =  [F+G](x) - [F+G](\basex)$ is simply a function value difference.
\end{example}

\begin{example}
    \label{ex:fb:pdps:gap}
    For the PDPS of \cref{ex:fb:pdps}, with $x=(y,z)$, we obtain the \term{Lagrangian duality gap}
    \[
        \gap(x; \basex)
        =
        \mathcal{L}(z, \basey) - \mathcal{L}(\base z, y)
        \quad\text{for}\quad
        \mathcal{L}(z, y)
        \defeq
        [f+g](z) + \dualprod{Kz}{y} - h_*(y).
    \]
    This is different from the true duality gap that arises from the Fenchel–Rockafellar theorem. For the latter no convergence results exist to our knowledge.
    In the convex case, if $0 \in H(\basex)$, the Lagrangian gap is non-negative, however, it may be zero even if $0 \not\in H(\basex)$, unlike for the true duality gap.
\end{example}

In the assumptions that follow, the factor $\gamma \in \R$ generally models available second-order growth, while $\LL \in \linear(X; X^*)$ is related to the Lipschitz continuity of $F'$; compare \eqref{eq:weaker:descent3}.
We allow $\LL$ to be an operator to model the fact that in, e.g., the PDPS of \cref{ex:fb:pdps}, we take forward steps only in the primal variable.
If $\LL$ were a scalar, the step length condition $\tau L + \tau\sigma\norm{K}^2 \le 1$, where $L$ only multiplies the primal step length $\tau$, would become much stricter.

For subdifferential convergence, we will need an inexact descent inequality, as well as bounds on sums of the gaps.

\begin{assumption}
    \label{ass:fb:descent}
    $M \in \linear(X; X^*)$ is self-adjoint and positive semi-definite.
    Also,
    \begin{enumerate}[label=(\roman*)]
        \item\label{item:fb:descent:xk}
        For a set $\localset \subset X$, $\eta>0$, and $\linear(X; X^*) \ni \LL \le 2(1-\eta)M$, for any $k \in \N$, whenever $\{x^n\}_{n=0}^k \subset \localset$, for some errors $\errDesc{k} \in \R$, we have
        \begin{equation}
            \label{eq:fb:descent:xk}
            \dualprod{\totalest{k+1}}{\nextx-\this x}_{X^*,X}
            \ge
            \gap(\nextx; \thisx)
            - \frac{1}{2}\norm{\nextx - \this x}_{\LL}^2
            - \errDesc{k}.
        \end{equation}
        \item\label{item:fb:descent:error}
        The errors satisfy
        $\sup_{N \in \N} \sum_{k=0}^{N-1} \errDesc{k} \le \rDesc$ for some $\rDesc < \infty$.
        \item\label{item:fb:descent:coercitivity}
        We have $x^0 \in \localset$, and for any $N \ge 1$, $\sum_{k=0}^{N-1}\gap(\nextx; \thisx) \le \rDesc$ implies $x^N \in \localset$.
        \item\label{item:fb:descent:lb}
        For some $0 \le \tilde\eta < \eta$, we have
        \[
            \inf_{N \in \N} \sum_{k=0}^{N-1}\left(
                \gap(\nextx; \thisx)
                + \tilde\eta \norm{\nextx-\thisx}_M^2
            \right) =: - C_\gap > - \infty.
        \]
    \end{enumerate}
\end{assumption}

\begin{remark}
    \label{rem:fb:global}
    If $\localset=X$, convergence will be global.
    In the examples of \cref{sec:inner-adjoint}, $\localset \ne X$ may arise from $S_u$, $G$, or $J$ being only locally Lipschitz continuously differentiable.
\end{remark}

We will also need the approximate subdifferentials $\totalest{k+1}$ to become better as the distance between the iterates shrinks, in the sense of

\begin{assumption}
    \label{ass:fb:approx-cont}
    For $H$ defined in \eqref{eq:fb:generic-problem}, we have
    \[
        \sup_{N \in \N} \sum_{k=0}^{N-1}\norm{\nextx - \thisx}_M^2 < \infty
        \implies
        \lim_{k \to \infty} \inf_{x_{k+1}^* \in H(\nextx)} \norm{x_{k+1}^*-\totalest{k+1}}_{X^*}^2 = 0.
    \]
\end{assumption}

We now introduce the notation
\begin{equation}
    \label{eq:fb:fakenorm}
    \fakenormsq{x}{\Gamma} \defeq \dualprod{\Gamma x}{x}_{X^*, X}
\end{equation}
for when $\Gamma \in \linear(X; X^*)$ may not be positive semi-definite, so that the notation $\norm{x}_\Gamma^2$ is not appropriate.

For function value and iterate convergence, we cannot work with just the iterates: we need to assume properties with respect to a base point $\basex \in X$, usually a solution. For iterate convergence, we assume at a $\basex \in \inv H(0)$ the three-point monotonicity type estimate (compare the proof of \cref{thm:tracking:inner-fb})
\begin{equation}
    \label{eq:fb:three-point-monotonicity:approx}
    \dualprod{\totalest{k+1}}{\nextx - \bar x}_{X^*,X} -\fakenormsq{\thisx - \bar x}{\GF}
    \ge
    \fakenormsq{\nextx - \bar x}{\GG} - \frac{1}{2}\norm{\nextx - \this x}_{\LL}^2 -\errMono{k}(\basex),
\end{equation}
for all $k \in \N$, \emph{whenever $\{x^n\}_{n=0}^k \subset \basexset$} for an open neighbourhood $\basexset$ of $\basex$, errors $\errMono{k}(\basex) \in \R$, self-adjoint $\GF, \GG \in \linear(X; X^*)$, and a positive semi-definite self-adjoint $\LL \in \linear(X; X^*)$.
We recall here that $H(\basex)$ is a set, and the notation \eqref{eq:fb:three-point-monotonicity:approx} means that the inequality holds for all elements of this set.
Note that, for a fixed $k \in \N$, we \emph{do not}, \emph{a priori}, require that $\nextx \in \basexset$.
This will be proved to hold \emph{a posteriori}.

For function value convergence, we need again a descent inequality similar to \eqref{eq:fb:descent:xk}, now instantiated at the base point $\basex$ instead of $\thisx$. That is, for all $k \in \N$, we assume for some errors $\errDescBar{k}(\basex) \in \R$ \emph{whenever $\{x^n\}_{n=0}^k \subset \basexset$} that
\begin{equation}
    \label{eq:fb:three-point-smoothness:approx}
    \dualprod{\totalest{k+1}}{\nextx-\basex}_{X^*,X}
    - \frac{1}{2}\fakenormsq{\thisx - \basex}{\GF}
    \ge
    \gap(\nextx;\basex)
    + \frac{1}{2}\fakenormsq{\nextx - \basex}{\GG}
    - \frac{1}{2}\norm{\nextx - \this x}_{\LL}^2
    - \errDescBar{k}(\basex).
\end{equation}
We write $\errDescBarEmph{k}(\basex) \defeq \errDescBar{k}(\basex)$ when we need to draw a distinction to \eqref{eq:fb:three-point-monotonicity:approx}.
These errors will need to have a finite sum, and we need to initialise close to $\basex$ with respect to the diameter of $\basexset$:

\begin{assumption}
    \label{ass:fb:non-escape}
    Given $\basex \in X$, for some $\eta \ge 0$ and $p \ge 1$ satisfying $0 \le \LL \le (1-\eta) M$, either
    \begin{enumerate}[label=(\alph*)]
        \item \label{item:fb:non-escape:any}
        \eqref{eq:fb:three-point-monotonicity:approx} holds,
        $\basex \in \inv H(0)$, and $M+2\GG \ge p\MF$ with $\MF \defeq M-2\GF \ge 0$; or

        \item \label{item:fb:non-escape:min}
        \eqref{eq:fb:three-point-smoothness:approx} holds,
        $
            \inf_{x \in \basexset} \gap(x; \basex) \ge 0,
        $
        and
        $M+\GG \ge p\MF$ with $\MF \defeq M-\GF \ge 0$.
    \end{enumerate}
    Moreover, $x^0 \in \openball_{\MF}(\basex, \sqrt{\delta^2 - 2r_p})$ and $\openball_{\MF}(\basex, \delta) \subset \basexset$ for some $\delta>0$ with
    \begin{equation}
        \label{eq:fb:local:error-assumption}
        \frac{1}{2}\delta^2 > r_p \defeq \sup_{N \in \N} \sum_{k=0}^{N-1} p^{k-N} \err{k}(\basex) < \infty.
    \end{equation}
\end{assumption}

\begin{remark}
    Recall the three-point descent inequalities \cref{eq:weaker:descent3,eq:weaker:descent3:inexact}.
    Compared to \eqref{eq:fb:three-point-smoothness:approx}, they are missing the $\GF$-term on the left hand side: the second-order growth from $F$ is also on the right hand side, with respect to $\nextx-\basex$ instead of $\thisx-\basex$.
    Such an estimate depends on a costly Young inequality that \cref{eq:fb:three-point-monotonicity:approx,eq:fb:three-point-smoothness:approx} avoid.

    The “transition conditions” of the form $M+\GG \ge p\MF$ in \cref{ass:fb:non-escape} roughly correspond to the basic growth condition $\GG+\GF \ge (p-1)M$ while allowing the chaining of estimates in this more refined approach.
\end{remark}

\begin{example}[Growth conditions for basic forward-backward splitting]
    For basic forward-backward splitting in a Hilbert space,
    \eqref{eq:fb:three-point-monotonicity:approx} is exactly of the form \eqref{eq:inner:fb:growth}, proved in \cref{thm:tracking:inner-fb}.
    The function value counterpart \eqref{eq:fb:three-point-smoothness:approx} can be proved similarly.
    Note that $\GF$ does not, in that case, exactly correspond to the Lipschitz factor of $F$, although is related to it.
\end{example}

\Cref{thm:fb:growth:tracking} will demonstrate how to, more generally, derive \cref{eq:fb:three-point-monotonicity:approx,eq:fb:three-point-smoothness:approx} from individual operator-relative growth and smoothness properties of $F$ and $G$, which are introduced in \cref{sec:operator-reg}.

\subsection{Convergence of subdifferentials and quasi-monotonicity of values}
\label{sec:fb:subdiff}

We first show the potentially global convergence of subdifferentials; see \cref{rem:fb:global}.
When $\Xi=0$, this could be followed by the Kurdyka–{\L}ojasiewicz property to show function value convergence, and, afterwards, either by a growth condition or, in finite dimensions, a finite-length argument based on \eqref{eq:fb:almost-monotonicity} and \cite[proof of Lemma 2.6]{attouch2013convergence} to show iterate convergence. As the property can easily be verified only in finite dimensions (for semi-algebraic functions), we prefer a more direct approach.

\begin{theorem}
    \label{thm:fb:subdiff}
    If \cref{ass:fb:descent} holds, then $\thisx \in \localset$;
    \begin{equation}
        \label{eq:fb:almost-monotonicity}
        \gap(\nextx; \thisx)
        + \eta \norm{\nextx-\thisx}_M^2
        \le  \errDesc{k}
        \quad\text{for all}\quad
        k \in \N;
    \end{equation}
    as well as
    \begin{equation}
        \label{eq:fb:step-sum-bound}
        \sup_{N \in \N} \sum_{k=0}^{N-1}\norm{\nextx-\thisx}_{M}^2 < \frac{C_\gap + \rDesc}{\eta-\tilde\eta}.
    \end{equation}
    If, moreover, \cref{ass:fb:approx-cont} holds, then also $\inf_{x^* \in H(\nextx)} \norm{x^*}_{X^*} \to 0$
\end{theorem}

\begin{proof}
    By the implicit algorithm \eqref{eq:fb:implicit}, the properties of Fenchel conjugates (e.g., \cite[Lemma 5.7]{clason2020introduction})
    and $-M(\nextx-\thisx) =: \totalest{k+1} \in \subdiff\left(\frac{1}{2}\norm{\freevar}_M^2\right)(\nextx-\thisx)$, we have
    \begin{equation}
        \label{eq:fb:subdiff:conjugate-trick}
        (\norm{\freevar}_M^2)^*(2\totalest{k+1})
        =2\left(\frac{1}{2}\norm{\freevar}_M^2\right)^*(\totalest{k+1})
        =\norm{\nextx-\thisx}_M^2
        =-\dualprod{\totalest{k+1}}{\nextx-\thisx}_{X^*,X}.
    \end{equation}
    If $\{x^j\}_{j=0}^{N-1} \subset \localset$, \cref{ass:fb:descent}\,\cref{item:fb:descent:xk} thus yields for all $k=0,\ldots,N-1$ that
    \begin{equation}
        \label{eq:fb:subdiff:mono0}
        \begin{aligned}[t]
        \gap(\nextx; \thisx)
        &
        =
        \gap(\nextx; \thisx)
        -\dualprod{\totalest{k+1}}{\nextx-\thisx}_{X^*,X}
        - \norm{\nextx-\thisx}_{M}^2
        \\
        &
        \le
        \errDesc{k}
        -\frac{1}{2}\norm{\nextx-\thisx}_{2M-\LL}^2
        \le
        \errDesc{k}
         - \eta \norm{\nextx-\thisx}_M^2.
        \end{aligned}
    \end{equation}
    Summing over all such $k$, and using  \cref{ass:fb:descent}\,\cref{item:fb:descent:error}, it follows
    \begin{equation}
        \label{eq:fb:subdiff:mono1}
        \sum_{k=0}^{N-1} \gap(\nextx; \thisx)
        + \sum_{k=0}^{N-1} \eta (\norm{\freevar}_M^2)^*(2\totalest{k+1})
        =
        \sum_{k=0}^{N-1} \gap(\nextx; \thisx)
        + \sum_{k=0}^{N-1}  \eta \norm{\nextx-\thisx}_M^2
        \le
        \rDesc.
    \end{equation}
    From \cref{ass:fb:descent}\,\cref{item:fb:descent:coercitivity}, it now follows that $x^N \in \localset$. Since, by the same assumption, $x^0 \in \localset$, induction establishes  \eqref{eq:fb:almost-monotonicity} and $x^k \in \localset$ for all $k \in \N$.
    Using \cref{ass:fb:descent}\,\cref{item:fb:descent:lb} in \eqref{eq:fb:subdiff:mono1}, we, moreover, deduce \eqref{eq:fb:step-sum-bound} and  $(\norm{\freevar}_M^2)^*(2\totalest{k+1}) \to 0$.
    Let $c \ge \norm{M}_{\linear(X; X^*)}$.
    By $\norm{\freevar}_M^2 \le c\norm{\freevar}_X^2$ and the properties of conjugates (e.g., \cite[Lemmas 5.4 and 5.7]{clason2020introduction}),%
    \[
        \frac{4}{c}\norm{\totalest{k+1}}_{X^*}^2
        = c\norm{2\totalest{k+1}/c}_{X^*}^2
        =(c\norm{\freevar}_X^2)^*(2\totalest{k+1})
        \le (\norm{\freevar}_M^2)^*(2\totalest{k+1}).
    \]
    Thus also $\norm{\totalest{k+1}}_{X^*} \to 0$.
    \Cref{ass:fb:approx-cont} proves that $\inf_{x^* \in H(\nextx)} \norm{\totalest{k+1} - x^*}_{X^*} \to 0$.
    Hence an application of the triangle inequality establishes $\inf_{x^* \in H(\nextx)} \norm{x^*}_{X^*} \to 0$.
\end{proof}

\begin{remark}[Forward-backward splitting]
    For the (inexact) forward-backward splitting of \eqref{eq:fb:fb-inexact}, once we verify the relevant assumptions in \cref{cor:fb:tracking}, the previous theorem establishes the monotonicity of function values, as well as the convergence of subdifferentials to zero, $\inf_{x^* \in \subdiff G(\nextx)} \norm{F'(\nextx)+x^*} \to 0$.
\end{remark}

\subsection{Non-escape, quasi-Féjer monotonicity, linear convergence}
\label{sec:fb:non-escape}

The next lemma is essential for all our strong convergence results. The proof is standard; see, e.g., \cite[Chapter 15]{clason2020introduction} for the case $\err{k}(\basex)=0$ and $\Xi=0$.
Observe that \eqref{eq:fb:local:strong-quasi-fejer} with the triangle inequality may be used to again prove \cref{ass:tracking:main}\,\cref{item:tracking:main:inner-tracking} for multilevel methods.

\begin{lemma}
    \label{lemma:fb:non-escape}
    Suppose \cref{ass:fb:non-escape} holds at $\basex \in X$.
    Then $x^k \in \openball_{\MF}(\basex, \delta) \subset \basexset$ for all $k \in \N$, and the sequence is ($p$-strongly) quasi-Féjer, i.e.,
    \begin{equation}
        \label{eq:fb:local:strong-quasi-fejer}
        \frac{p}{2}\norm{\nextx-\basex}_{\MF}^2
        \le
        \frac{1}{2}\norm{\thisx-\basex}_{\MF}^2
        + \err{k}(\basex)
    \end{equation}
    Moreover, $\sup_{N \in \N} \sum_{k=0}^{N-1} p^{k-N} \norm{\nextx-\thisx}_M^2 \le \delta^2/\eta$ if $\eta>0$.
\end{lemma}

\begin{proof}
    We first treat \cref{ass:fb:non-escape} option \cref{item:fb:non-escape:any}.
    Fix $N \in \N$ and suppose $\{x^j\}_{j=0}^{N-1} \subset \basexset$.
    Observe that $\dualprod{\Xi x}{x}=0$ for all $x \in X$ by the skew-adjointness of $\Xi$.
    Since $0 \in H(\basex)$, using \eqref{eq:fb:three-point-monotonicity:approx} in the implicit algorithm \eqref{eq:fb:implicit}, we thus get
    \[
        \begin{aligned}[t]
        - \fakenormsq{\thisx - \basex}{\GF}
        -\dualprod{M(\nextx-\thisx)}{\nextx-\basex}_{X^*,X}
        &
        \ge
        \fakenormsq{\nextx - \basex}{\GG}
        - \frac{1}{2}\norm{\nextx - \thisx}_{\LL}^2
        - \err{k}(\basex)
        \end{aligned}
    \]
    for all $k \in \{0,\ldots,N-1\}$.
    By the Pythagoras' identity \eqref{eq:norms:pythagoras},
    \[
        \frac{1}{2}\fakenormsq{\thisx - \basex}{M-2\GF}
        \ge
        \frac{1}{2}\fakenormsq{\nextx - \basex}{M+2\GG}
        + \frac{1}{2}\norm{\nextx - \thisx}_{M - \LL}^2
        - \err{k}(\basex).
    \]
    By  $\LL \le (1- \eta) M$ and $M+2\GG \ge p\MF$ for $\MF = M-2\GF \ge 0$, we obtain
    \begin{equation}
        \label{eq:fb:local:step-estimate:0}
        \frac{1}{2}\norm{\thisx-\basex}_{\MF}^2
        \ge
        \frac{\eta}{2}\norm{\nextx-\thisx}_{M}^2
        + \frac{p}{2}\norm{\nextx-\basex}_{\MF}^2
        - \err{k}(\basex).
    \end{equation}
    Multiplying by $p^k$, and summing over $k=0,\ldots,N-1$ yields
    \begin{equation}
        \label{eq:fb:local:step-estimate:0:summed}
        \frac{1}{2}\norm{x^0-\basex}_{\MF}^2
        + \sum_{k=0}^{N-1} p^k \err{k}(\basex)
        \ge
        \sum_{k=0}^{N-1} \frac{\eta p^k}{2}\norm{\nextx-\thisx}_{M}^2
        + \frac{p^N}{2}\norm{x^N-\basex}_{\MF}^2.
    \end{equation}
    Multiplying by $p^{-N} \le 1$ and using $x^0 \in \openball_{\MF}(\basex, \sqrt{\delta^2 - 2r_p})$ and \eqref{eq:fb:local:error-assumption}, it follows
    \begin{equation}
        \label{eq:fb:local:step-estimate:0:summed:est}
        \frac{\delta^2}{2}
        =
            \frac{\delta^2-2r_p}{2}
            + r_p
        >
        \sum_{k=0}^{N-1} \frac{\eta p^{k-N}}{2}\norm{\nextx-\thisx}_{M}^2
        +
        \frac{1}{2}\norm{x^N-\basex}_{\MF}^2.
    \end{equation}
    Hence $x^N \in \openball_{\MF}(\basex, \delta)$.
    Since $x^0 \in \basexset$ by \cref{ass:fb:non-escape}, an inductive argument shows that $\this{x} \in \openball_{\MF}(\basex, \delta) \subset \basexset$ for all $k \in \N$, justifying the above steps.
    Finally, \eqref{eq:fb:local:step-estimate:0} shows \eqref{eq:fb:local:strong-quasi-fejer}, while the claim $\sup_{N \in \N} \sum_{k=0}^{N-1} p^{k-N} \norm{\nextx-\thisx}_{M}^2 < \delta^2/\eta$ follows from \eqref{eq:fb:local:step-estimate:0:summed:est} and $\eta>0$.

    Regarding option \cref{ass:fb:non-escape}\,\cref{item:fb:non-escape:min}, arguing as above with \eqref{eq:fb:three-point-smoothness:approx} in place of  \eqref{eq:fb:three-point-monotonicity:approx}, we get in place of \eqref{eq:fb:local:step-estimate:0} the estimate
    \begin{equation}
        \label{eq:fb:local:step-estimate}
            \frac{1}{2}\norm{\thisx-\basex}_{\MF}^2
            \ge
            \gap(\nextx; \basex)
            + \frac{\eta}{2}\norm{\nextx-\thisx}_{M}^2
            + \frac{1+\gamma}{2}\norm{\nextx-\basex}_{\MF}^2
            - \err{k}(\basex),
    \end{equation}
    where now $\MF = M - \GF \ge 0$.
    Using $\inf_{x \in \openball_{\MF}(\delta, \basex)} \gap(x; \basex) \ge 0$, we proceed as in option \cref{item:fb:non-escape:any} to establish \eqref{eq:fb:local:step-estimate:0:summed:est}, and from there onwards.
\end{proof}

A closer look at \eqref{eq:fb:local:step-estimate:0:summed} yields linear convergence if $p>1$ and we remove $p^{-N}$ from \eqref{eq:fb:local:error-assumption}.

\begin{corollary}
    \label{cor:fb:local:linear}
    Suppose \cref{ass:fb:non-escape} holds at $\basex \in X$ with $p>1$ and the inequality in \eqref{eq:fb:local:error-assumption} strengthened to
    \begin{equation}
        \label{eq:fb:local:error-assumption:strong}
        \frac{1}{2}\delta^2 > \sup_{N \in \N} \sum_{k=0}^{N-1} p^k \err{k}(\basex) < \infty.
    \end{equation}
    Then $\norm{x^N-\basex}_{\MF}^2 \to 0$ at the rate $O(p^{-N})$.
\end{corollary}

\subsection{Local convergence of function values}
\label{sec:fb:values}

We now proceed to function values and duality gaps.
The idea of possibly assuming both \cref{ass:fb:non-escape}\,\cref{item:fb:non-escape:any} and a relaxed version of \cref{item:fb:non-escape:min}, as an alternative to just the latter, is to be able to study descent at non-minimising critical points.
For simplicity, we only treat sublinear convergence.

\begin{theorem}
    \label{thm:fb:local:values}
    Suppose \cref{ass:fb:non-escape} holds at $\basex \in X$ and, for a non-empty set $\hat X \subset X$, \cref{eq:fb:three-point-smoothness:approx} holds for all $\hat x \in \hat X$ with $\LL = \LL_{\hat x} \le M$, $\gamma=\gamma_{\hat x} \ge 0$, and $\Omega_{\hat x} \supset \openball_{\MF}(\basex, \delta)$.
    Then
    \begin{equation}
        \label{eq:fb:local:values:ergodic}
        \sup_{\hat x \in \hat X} \sum_{k=0}^{N-1} \gap(\nextx; \hat x)
        \le
        \sup_{\hat x \in \hat X}\left(
            \frac{1}{2}\norm{x^0-\hat x}_{\MF}^2
            + \sum_{k=0}^{N-1}  \errDescBarEmph{k}(\hat x)
        \right)
        \quad\text{for all}\quad N \in \N.
    \end{equation}

    If $\Xi=0$ and \cref{ass:fb:descent} holds\footnote{Since the proof of the present \cref{thm:fb:local:values} shows that $\thisx \in \openball_M(\optx, \delta)$ for all $k \in \N$, to prove the required \eqref{eq:fb:almost-monotonicity}, it would be enough to assume that just \cref{ass:fb:descent}\,\cref{item:fb:descent:xk} holds with $\localset \supset \openball_M(\optx, \delta)$.}, then, for all $N \in \N$,
    \begin{equation}
        \label{eq:fb:local:values:non-ergodic}
        [F+G](x^N)
        \le
        \inf_{\hat x \in \hat X} [F+G](\hat x)
        + \sup_{\hat x \in \hat X} \left(
            \frac{1}{2N}\norm{x^0-\hat x}_{\MF}^2
            + \sum_{k=0}^{N-1} \left(
                \frac{1}{N}\errDescBarEmph{k}(\hat x)
                +
                \frac{k+1}{N}\errDesc{k}
            \right)
        \right).
    \end{equation}
\end{theorem}

\begin{proof}
    \Cref{lemma:fb:non-escape} shows for all $k \in \N$ that $\thisx \in \openball_{\MF}(\basex, \delta) \subset \Isect_{\hat x \in \hat X} \Omega_{\hat x}$.
    Hence, for any $\hat x \in \hat X$, we may follow the proof of the lemma for case \cref{item:fb:non-escape:min} of \cref{ass:fb:non-escape} to establish \eqref{eq:fb:local:step-estimate} for $\basex=\hat x$.
    To reach this point, the assumption $\inf_{x \in \openball_{\MF}(\delta, \basex)} \gap(x; \basex) \ge 0$ was not yet needed.
    Now, summing \eqref{eq:fb:local:step-estimate} over $k=0,\ldots,N-1$, we obtain
    \begin{equation}
        \label{eq:fb:local:values:main-estimate-in-proof}
        \frac{1}{2}\norm{x^0-\hat x}_{\MF}^2
        + \sum_{k=0}^{N-1}  \errDescBarEmph{k}(\hat x)
        \ge
        \sum_{k=0}^{N-1} \gap(\nextx; \hat x)
        + \frac{1}{2}\norm{x^N-\hat x}_{\MF}^2.
    \end{equation}
    Taking the supremum over $\hat x \in \hat X$, and using $\MF \ge 0$, this establishes \eqref{eq:fb:local:values:ergodic}.

    Suppose then that $\Xi=0$ and \cref{ass:fb:descent} holds.
    \Cref{thm:fb:subdiff} now establishes \eqref{eq:fb:almost-monotonicity}, i.e., the quasi-monotonicity
    $
        [F+G](\nextx) \le [F+G](\thisx) + \errDesc{k}.
    $
    Repeatedly using this and $\gap(\nextx; \hat x)=[F+G](\nextx)-[F+G](\hat x)$ in \eqref{eq:fb:local:values:main-estimate-in-proof}, and dividing by $N$, we obtain \eqref{eq:fb:local:values:non-ergodic}.
\end{proof}

\subsection{Weak convergence}
\label{sec:fb:weak}

We next prove the weak convergence of the iterates.
We call the self-adjoint and positive semi-definite preconditioner $M \in \linear(X; X^*)$ \term{admissible for weak convergence} if $\norm{\thisx}_M \to 0$ implies $M\thisx \to 0$.

\begin{example}
    Let $M=A^*A$ for some $A \in \linear(X; V)$ with $V$ a Hilbert space.
    Then $\norm{\thisx}_M \to 0$ implies $A\thisx \to 0$, and consequently $M\thisx \to 0$.
    Thus, $M$ is weakly admissible.
    In Hilbert spaces, every positive semi-definite self-adjoint operator has such a square root $A$ with $V=X$.
    For a convolution-based construction in the space of Radon measures, see \cite[Theorem 2.4]{tuomov-pointsource}.
\end{example}

\begin{theorem}
    \label{thm:fb:local:weak}
    Suppose \cref{ass:fb:non-escape,ass:fb:approx-cont} hold with $p=1$ and $\eta>0$ at some $\basex=\optx \in \inv H(0)$, and that either \cref{ass:fb:non-escape}\,\cref{item:fb:non-escape:any} or \cref{item:fb:non-escape:min} (only the item, not the entire assumption) holds with $\openball_{\MF}(\basex, \delta) \subset \Omega_{\hat x}$ and $\sum_{k=0}^\infty \err{k}(\hat x) < \infty$ at all $\hat x \in \hat X \defeq \inv H(0) \isect \openball_M(\basex, \delta)$.
    Also suppose that the preconditioner $M$ is admissible for weak convergence, $\MF$ is strictly monotone, and $F$ is either convex or $F'$ is weak-to-strong continuous.
    Then $\thisx \weakto \hat x$ weakly for some $\hat x \in \hat X$.
\end{theorem}

\begin{proof}
    \Cref{lemma:fb:non-escape} proves that $\thisx \in \openball_{\MF}(\basex, \delta)$ for all $k \in \N$, as well as that $\sup_{N \in N} \sum_{k=0}^{N-1} \norm{\nextx-\thisx}_M^2 < \infty$.
    The latter establishes $\norm{\nextx-\thisx}_M \to 0$, and through admissibility for weak convergence, and \eqref{eq:fb:implicit}, that $\totalest{k+1} = -M(\nextx-\thisx) \to 0$ strongly in $X^*$.
    Moreover, \cref{ass:fb:approx-cont} yields $\norm{\totalest{k+1} - x^*_{k+1}}_{X^*} \to 0$ for some $x^*_{k+1} \in H(\nextx)$.
    Consequently $x^*_{k+1} \to 0$.
    Since $\thisx \in \openball_{\MF}(\basex, \delta) \subset \Omega_{\hat x}$, \cref{lemma:fb:non-escape}, shows the quasi-Féjer monotonicity \eqref{eq:fb:local:strong-quasi-fejer} (with $p=1$) for all $\hat x \in \hat X$ and $k \in \N$.

    Suppose then that $x^{k_j+1} \weakto \hat x$ for a subsequence $\{k_j\}_{j \in \N} \subset \N$ and a $\hat x \in X$. We want to show that $\hat x \in \hat X$.
    We consider two cases:
    \begin{enumerate}
        \item If $F$ is convex, $H$ is maximally monotone\footnote{That the additive skew-adjoint term $\Xi$ does not destroy maximal monotonicity, can be proved completely analogously to the Hilbert space case in \cite[Lemma 9.9]{clason2020introduction}.\label{foot:fb:skew-adjoint-max-monotone}}, hence weak-to-strong outer semicontinuous \cite[Lemma 6.10]{clason2020introduction}.
        Now $x^{k_j+1} \weakto \hat x$ and $H(x^{k_j+1}) \ni x^*_{k_j+1} \to 0$ obliges $0 \in H(\hat x)$.

        \item Suppose then that $F'$ is weak-to-strong continuous. Now still $P: x \mapsto \subdiff G(x) + \Xi x$ is maximally monotone\textsuperscript{\ref{foot:fb:skew-adjoint-max-monotone}}, hence weak-to-strong outer semicontinuous.
        We have
        $P(x^{k_j+1}) \ni x^*_{k_j+1} - F'(x^{k_j+1}) \to -F'(\hat x)$ strongly in $X^*$, as well as $x^{k_j+1} \weakto \hat x$, so we must have $-F'(\hat x) \in P(\hat x)$. But this again says $0 \in H(\hat x)$.
    \end{enumerate}
    Thus every weak limiting point $\hat x$ of $\{\thisx\}_{k \in \N}$ satisfies $0 \in H(\hat x)$.
    But, since $\thisx \in \openball_M(\basex, \delta)$ for all $k \in \N$, also $\hat x \in \openball_M(\basex, \delta)$.
    This proves that $\hat x \in \hat X$.
    Since, by assumption, $\sum_{k=0}^\infty \err{k}(\hat x) < \infty$ for all $\hat x \in \hat X$, the quasi-Féjer monotonicity \eqref{eq:fb:local:strong-quasi-fejer} with the quasi-Opial's \cref{lemma:opial} finishes the proof.
\end{proof}

\begin{example}
    In the setting of \cref{sec:tracking} and \cref{thm:weaker:erroneous-lipschitz}, the weak-$*$-to-strong continuity of $F'$ can be achieved, for example, when $F(x)=\frac{1}{2}\norm{S(x)-b}^2$ for a Lipschitz and bounded $S$ with finite-dimensional range.
\end{example}

\begin{remark}
    All of our theory also applies when $X$ is the dual space of a separable normed space $X_*$, and we replace in our definitions $X^*$ by the predual space $X_*$, that is, subdifferentials are subsets of $X_*$, and $M, \Lambda \in \linear(X; X_*)$, etc.
    With this change the theory applies, for example, to $X$ a space of Radon measures, as in \cite{tuomov-pointsource}.
    Then \cref{thm:fb:local:weak} proves the weak-$*$ convergence.
\end{remark}

\section{Operator-relative regularity}
\label{sec:operator-reg}

We now introduce operator-relative smoothness and growth concepts to facilitate the analysis of
\begin{enumerate}[nosep]
    \item primal-dual methods as generalised forward-backward methods, and
    \item the basic forward-backward method for \eqref{eq:fb:problem} when neither $F$ nor $G$ alone provides second-order growth on the whole space $X$, but jointly they do.
\end{enumerate}
We start with the relevant definitions in \cref{sec:operator-reg:def}, and then prove the relevant operator-relative descent inequalities and three-point monotonicity in \cref{sec:opetor-reg:estimates}.

\subsection{Definitions}
\label{sec:operator-reg:def}

For a self-adjoint positive semi-definite $\Lambda \in \linear(X; X^*)$ on a normed space $X$, we say that the Gâteaux derivative $DF$ of $F: X \to \R$ is $\Lambda$-$*$-Lipschitz if
\[
    \norm{DF(z)-DF(x)}_{\Lambda,*} \le \norm{x-z}_\Lambda
    \quad(x, z \in X).
\]
We say that $DF$ is $\Lambda$-$*$-cocoercive, if
\[
    \norm{DF(z)-DF(x)}_{\Lambda,*}^2 \le \dualprod{DF(z)-DF(x)}{z-x}_{X^*,X}.
\]

\begin{remark}
    These properties could be defined for arbitrary seminorms $\norm{\freevar}_\circ$ and corresponding dual support functions $\norm{\freevar}_*$, however, since we will be combining the estimates of this section with the Pythagoras' identity in this the next one, we restrict our attention to operator-generated instances.
\end{remark}

\begin{example}
    \label{ex:operator-reg:hilbert}
    On a Hilbert space $X$, take $\Lambda=L\mathscr{I}$ for the standard injection $\mathscr{I}: X \to X^*$.
    Then $\norm{\freevar}_{\Lambda,*} = L^{-1/2}\norm{\freevar}_{X^*}$, so these concepts reduce to standard $L$-Lipschitz and $\inv L$-cocoercivity properties.
    The two are equivalent \cite[Lemma 7.1]{clason2020introduction}.
\end{example}

The following lemma lists important implications.

\begin{lemma}
    \label{lemma:opeartor-reg:coco-lipschitz}
    $\Lambda$-$*$-cocoercive
    $\implies$ $\Lambda$-$*$-Lipschitz
    $\implies$ $\dualprod{DF(z)-DF(x)}{z-x}_{X^*,X} \le \norm{z-x}_\Lambda^2$.
    Moreover, $\Lambda$-$*$-Lipschitz is equivalent to $\dualprod{DF(z)-DF(x)}{h}_{X^*,X} \le \frac{1}{2}\norm{h}_\Lambda^2 + \frac{1}{2}\norm{z-x}_\Lambda^2$ holding for all $h \in X$, and
    $\Lambda$-$*$-co-coercivity is equivalent to $\dualprod{DF(z)-DF(x)}{2h-(z-x)}_{X^*,X} \le \norm{h}_\Lambda^2$ holding for all $h \in X$.
\end{lemma}

\begin{proof}
    \Cref{lemma:semicon:props}\,\cref{item:semicon:props:young} gives
    $
        \dualprod{DF(z)-DF(x)}{z-x}_{X^*,X}
        \le
        \frac{1}{2}\norm{DF(z)-DF(x)}_{\Lambda,*}^2 + \frac{1}{2}\norm{z-x}_\Lambda^2.
    $
    Using co-coercivity in the left hand side and rearranging gives the first implication.
    Using the $\Lambda$-$*$-Lipschitz property on the right hand side and rearranging gives the second implication.
    The equivalences hold by \cref{lemma:semicon:props}\,\cref{item:semicon:props:conjugacy} and the definition of the Fenchel conjugate.
\end{proof}

One reason for introducing these concepts is to allow functions such as $F$ in \eqref{eq:fb:pdps:functions} to have distinct Lipschitz factors on distinct subspaces.
For the same reason, recalling the notation $\fakenormsq{\freevar}{\Gamma}$ from \eqref{eq:fb:fakenorm}, we call $DF$ locally $\Gamma$-monotone in $\localset \ni \bar x$ for a self-adjoint $\Gamma \in \linear(X; X^*)$ if
\[
    \dualprod{DF(z)-DF(\bar x)}{z-\bar x} \ge \fakenormsq{z-\bar x}{\Gamma}
    \quad
    (z \in \Omega).
\]
We do not at this stage assume $\Gamma$ to be positive semi-definite.
The main reason for allowing non-positive semi-definite $\Gamma$ is to treat sums of functions $F+G$ that satisfy $\Gamma_F+\Gamma_G \ge 0$ while, e.g., $\Gamma_F \not\ge 0$.

Finally, we call a possibly nonsmooth function $G$ $\Gamma$-subdifferentiable and the (convex) subdifferential $\subdiff G$ $\Gamma$-monotone if, respectively,
\begin{equation}
    \label{eq:operator-reg:subdiff}
    G(\tilde x) - G(x) \ge \dualprod{q}{\tilde x - x} + \frac{1}{2}\fakenormsq{\tilde x-x}{\Gamma}
    \quad\text{or}\quad
    \dualprod{\tilde q-q}{\tilde x-x} \ge \fakenormsq{\tilde x-x}{\Gamma}
\end{equation}
for all $q \in \subdiff G(x)$; $\tilde q \in \subdiff \tilde G(x)$, and $x, \tilde x \in X$.
Obviously, the former implies the latter.

\subsection{Estimates}
\label{sec:opetor-reg:estimates}

We first prove a $\Lambda$-$*$-Lipschitz descent lemma, as a generalisation of the basic descent inequality \eqref{eq:weaker:descent}.

\begin{lemma}
    \label{lemma:smoothness:descent}
    On a normed space $X$, suppose $F: X \to \R$ has a  $\Lambda$-$*$-Lipschitz Gâteaux derivative for a self-adjoint positive semi-definite $\Lambda \in \linear(X; X^*)$.
    Then
    \begin{equation}
        \label{eq:smoothness:descent-operator}
         F(x) - F(z) - \dualprod{DF(z)}{x-z}_{X^*,X}
        \le
        \frac{1}{2}\norm{z-x}_\Lambda^2.
    \end{equation}
\end{lemma}

\begin{proof}
    By the mean value theorem and \cref{lemma:opeartor-reg:coco-lipschitz}
    \[
        F(x) - F(z) - \dualprod{DF(z)}{x-z}_{X^*,X}
        = \int_0^1 \dualprod{DF(z+t(x-z))-DF(z)}{x-z} \d t
        \le \int_0^1 t \norm{x-z}_\Lambda^2 \d t.
    \]
    Integrating, the claim follows.
\end{proof}

We now move on to three-point estimates.
The first lemma provides a tool for proving \eqref{eq:fb:three-point-smoothness:approx}, and the second one for proving \eqref{eq:fb:three-point-monotonicity:approx}.
It is important that $x$ ($=\nextx$ in the application to forward steps at $\thisx$) is not, a priori, restricted to the neighbourhood $\localset$ of $\Gamma$-monotonicity at $\bar x$.

For compactness of our overall presentation, besides the smooth function $F$, we include an additional subdifferentiable function $G$ and a skew-adjoint operator $\Xi$, which could always be taken as zero.

\begin{lemma}
    \label{lemma:smoothness:nonconvex}
    On a normed space $X$, let $F: X \to \R$ and suppose $DF$ is $\Lambda$-$*$-Lipschitz and $\Gamma_F$-monotone at $\bar x \in X$ in a convex neighbourhood $\localset \ni \basex$  for some self-adjoint and positive semi-definite $\Lambda \in \linear(X; X^*)$ and a self-adjoint $\Gamma_F \in \linear(X; X^*)$.
    Also suppose that $G: X \to \extR$ is $\Gamma_G$-strongly subdifferentiable for a self-adjoint $\Gamma_G \in \linear(X; X^*)$, and $\Xi \in \linear(X; X^*)$ is skew-adjoint.
    Then, for any $\beta>0$, for all $z \in \localset$, $x \in X$, we have
    \[
        \dualprod{DF(z) + \subdiff G(x) + \Xi x}{x-\bar x}
        - \frac{1}{2}\fakenormsq{\bar x-z}{\Gamma_F}
        \ge
        \gap(x; \bar x)
        + \frac{1}{2}\fakenormsq{x-\bar x}{\Gamma_G}
        - \frac{1}{2}\normsq{x-z}{\Lambda}
    \]
    where $\gap$ is defined in \eqref{eq:fb:general:gap}.
\end{lemma}

\begin{proof}
    Similarly to the proof of the descent inequality in \cref{lemma:smoothness:descent}, the mean value theorem applied to $\phi(t) \defeq F(\bar x + t(z-\bar x))$, followed by the assumed local $\Gamma_F$-monotonicity of $DF$, establishes
    \[
        \begin{aligned}
        F&(\bar x) - F(z) - \dualprod{DF(z)}{\bar x-z}_{X^*,X}
        \\
        &
        = \int_0^1 \dualprod{DF(z+t(\bar x-z))-DF(z)}{\bar x-z} \d t
        \ge \int_0^1 t \fakenormsq{\bar x-z}{\Gamma_F} \d t
        = \frac{1}{2}\fakenormsq{\bar x-z}{\Gamma_F}.
        \end{aligned}
    \]
    Summing this inequality with the descent inequality of \cref{lemma:smoothness:descent}, we obtain
    \[
        \dualprod{DF(z)}{x-\bar x} - \frac{1}{2}\fakenormsq{\bar x-z}{\Gamma_F}
        \ge
        F(x)-F(\bar x)
        - \frac{1}{2}\normsq{x-z}{\Lambda}.
    \]
    By the skew-symmetricity of $\Xi$, we have $\dualprod{\Xi x}{x-\bar x} = \dualprod{\Xi x}{\bar x}$.
    The claim follows from summing this expression with the previous inequality and the first part of \eqref{eq:operator-reg:subdiff} for $G$.
\end{proof}

\begin{lemma}
    \label{lemma:smoothness:monotonicity}
    On a normed space $X$, let $F: X \to \R$ and suppose $DF$ is $\Lambda$-$*$-co-coercive and $\Gamma_F$-monotone in a neighbourhood $\localset \ni \basex$ of some $\basex \in X$ for some self-adjoint and positive semi-definite $\Lambda \in \linear(X; X^*)$ and a self-adjoint $\Gamma_F \in \linear(X; X^*)$.
    Also suppose that $G: X \to \extR$ has a $\Gamma_G$-monotone subdifferential for some self-adjoint $\Gamma_G \in \linear(X; X^*)$, that $\Xi \in \linear(X; X^*)$ skew-adjoint, and that
    \begin{equation}
        \label{eq:smoothness:monotonicity:nonconvex:combined:oc}
        x^* \in DF(\bar x) + \subdiff G(\bar x) + \Xi \bar x.
    \end{equation}
    Then, for any $\beta>0$ and $\zeta \in (0, 1]$, for all $z \in \localset$ and $x \in X$, we have
    \[
        \dualprod{DF(z) + \subdiff G(x) + \Xi x - x^*}{x-\bar x}_{X^*,X}
        - (1-\zeta) \fakenormsq{z-\bar x}{\Gamma_F}
        \ge
        \fakenormsq{x-\bar x}{\Gamma_G}
        - \frac{1}{4\zeta} \normsq{x-z}{\Lambda}.
    \]
\end{lemma}

\begin{proof}
    Interpolating between $\Gamma_F$-monotonicity and $\Lambda$-$*$-co-coercivity, and using \cref{lemma:semicon:props}\,\cref{item:semicon:props:young},
    \[
        \begin{aligned}[t]
        \dualprod{DF(z)-DF(\bar x)&}{x-\bar x}_{X^*,X}
        =
        \dualprod{DF(z)-DF(\bar x)}{z-\bar x}_{X^*,X}
        +
        \dualprod{DF(z)-DF(\bar x)}{x-z}_{X^*,X}
        \\
        &
        \ge
        (1-\zeta) \fakenormsq{z-\bar x}{\Gamma_F}
        +\zeta \norm{DF(z)-DF(\bar x)}_{\Lambda,*}^2
        +\dualprod{DF(z)-DF(\bar x)}{x-z}_{X^*,X}
        \\
        &
        \ge
        (1-\zeta) \fakenormsq{z-\bar x}{\Gamma_F} - \frac{1}{4\zeta}\normsq{x-z}{\Lambda}.
        \end{aligned}
    \]
    We have $\dualprod{\Xi(x-\bar x)}{x-\bar x}=0$.
    The claim follows from summing this expression, the previous inequality, \eqref{eq:smoothness:monotonicity:nonconvex:combined:oc}, and the first part of \eqref{eq:operator-reg:subdiff}.
\end{proof}

\section{Outer algorithms}
\label{sec:outer-examples}

We now explicitly verify \cref{ass:fb:descent,ass:fb:non-escape,ass:fb:approx-cont} for both basic forward-backward splitting and the PDPS, as well as their inexact versions based on the estimation of $F'(\thisx)$ by $\estdiff F(\thisx)$ formed using inner and adjoint algorithms satisfying the tracking theory of \cref{sec:tracking}.
We first provide in \cref{sec:outer-examples:general} a general result for operator-relative forward-backward type methods for \eqref{eq:fb:generic-problem}.
This forms the basis of treatment of both the basic forward-backward splitting in \cref{sec:outer-examples:fb}, and primal-dual proximal splitting in \cref{sec:outer-examples:pdps}.

\subsection{A general result}
\label{sec:outer-examples:general}

Let $\Lambda \in \linear(X; X^*)$ be self-adjoint positive and semi-definite.
To use the tracking theory of \cref{sec:tracking}, recalling \cref{sec:semicon}, we make the choices
\begin{equation}
    \label{eq:fb:distances}
    \dX=\norm{\freevar}_\Lambda,
    \quad\text{and}\quad
    \dXstar=\norm{\freevar}_{\Lambda,*}.
\end{equation}
The main reason for restricting the semi-norms to the operator form, is that we require $\norm{\freevar}_\circ \le c \norm{\freevar}_M$, i.e., $\Lambda \le cM$, for some $c \ge 0$.
We make no restrictions on $\dU$ and $\dW$, which have no direct role in this section.

In brief, the next theorem says that
\begin{enumerate}
    \item \Cref{ass:fb:descent,ass:fb:approx-cont}, used for subdifferential convergence by \cref{thm:fb:subdiff}, require that the initialisation $x^0$ be in the set $\localset$ where the tracking assumptions hold, and that the step lengths (encoded in $M$) be small compared to the operator-relative Lipschitz factor $\Lambda$.
    \item \Cref{ass:fb:non-escape}, used for stronger convergence results by \cref{cor:fb:local:linear,thm:fb:local:values,thm:fb:local:weak}, also requires sufficient strong subdifferentiability.
\end{enumerate}

\begin{theorem}
    \label{thm:fb:growth:tracking}
    On a normed space $X$, let $M, \Lambda \in \linear(X; X^*)$ be self-adjoint and positive semi-definite.
    Suppose $F: X \to \R$ has a $\Lambda$-$*$-Lipschitz Fréchet derivative in $\localset \subset X$, and $G: X \to \extR$ is convex, proper, and lower semicontinuous.
    Given an initial $x^0 \in X$, for all $k \in \N$, construct $\estdiff F(\thisx)$ obeying \cref{ass:tracking:main} (or, more generally, \cref{thm:weaker:smoothness,thm:weaker:erroneous-lipschitz} without \cref{ass:tracking:main}) in $\localset$ for the distances \eqref{eq:fb:distances}.
    Update $\nextx$ by solving
    \begin{equation}
        \label{eq:fb:implicit:thm}
        0 \in \estdiff F(\thisx) + \subdiff G(\nextx) + \Xi\nextx + M(\nextx-\thisx).
    \end{equation}
    Let $\trackingressum$, $\kappa$, $e_{p,k}$, and $\Psi_p$ be as in \cref{thm:weaker:smoothness}.
    Then,
    \begin{enumerate}[label=(\roman*)]
        \item\label{item:fb:growth:tracking:weak}
        \Cref{ass:fb:descent} holds for any $\eta>\tilde\eta\ge0$ and $p \in [1, \kappa)$ with $\errDesc{k}=e_{p,k}/(2\tilde\gamma)$ and $\rDesc = \Psi_p/(2\tilde\gamma)$ for any $\tilde\gamma>0$, provided $\Xi=0$, $\inf [F+G] > -\infty$, $\localset \supset \sublev_{\Psi_p/(2\tilde\gamma)+[F+G](x^0)}(F+G)$, and,
        \[
            0 \le \LL \defeq (1 + \trackingressum^2\inv{\tilde\gamma} + \tilde\gamma)\Lambda \le 2(1-\eta)M.
        \]
        \item\label{item:fb:growth:tracking:approx-cont}
        \Cref{ass:fb:approx-cont} holds if $\Lambda \le cM$ for a $c>0$.
    \end{enumerate}

    Suppose further that $G$ is $\Gamma_G$-strongly subdifferentiable in $X$, and $F'$ is $\Gamma_F$-monotone in $\localset$ for some $\Gamma_F, \Gamma_G \in \linear(X; X^*)$.
    Suppose also that $\openball_{\MF}(\basex, \delta) \subset \localset$ for a base point $\basex \in X$, $\delta>0$, and $\MF$ defined below in \cref{item:fb:growth:tracking:mono-local} or \cref{item:fb:growth:tracking:local}.
    Pick $\tilde\gamma>0$ and $p \in [1, \kappa)$.
    If
    \begin{equation}
        \label{eq:fb:growth:init}
        x^0 \in \openball_{\MF}\left(\basex, \sqrt{\delta^2-\Psi_p/\tilde\gamma}\right)
        \quad\text{with}\quad
        \Psi_p < \delta^2\tilde\gamma,
    \end{equation}
    then, taking $\errMono{k}(\basex)=e_{p,k}/(2\tilde\gamma)$ and $\basexset=\localset$, for any $\eta \ge 0$:
    \begin{enumerate}[resume*]
        \item\label{item:fb:growth:tracking:mono-local}
        \Cref{ass:fb:non-escape} option \cref{item:fb:non-escape:any} holds if $\basex \in \inv H(0)$, $F'$ is $\Lambda$-$*$-cocoercive\footnote{%
            Recall from \cref{lemma:opeartor-reg:coco-lipschitz} that this implies the earlier-assumed $\Lambda$-$*$-Lipschitz property.
        }
        in $\localset$, and, for some $\zeta \in (0, 1]$,%
        \begin{subequations}%
        \label{eq:fb:growth:tracking:mono-local:assumptions}
        \begin{align}
            0 & \le \MF \defeq M - 2(1-\zeta)\Gamma_F,
            \quad
            2\Gamma_G + 2p(1-\zeta)\Gamma_F \ge (p-1) M + \tilde\gamma \Lambda,
            \quad\text{and}
            \\
            0 & \le \LL \defeq (\inv{(2\zeta)}+\trackingressum^2\inv{\tilde\gamma})\Lambda
            \le (1-\eta)M.
        \end{align}
        \end{subequations}

        \item\label{item:fb:growth:tracking:local}
        \Cref{ass:fb:non-escape} option \cref{item:fb:non-escape:min} holds if $\localset$ is convex, $\inf_{x \in \localset} \gap(x; \basex) \ge 0$, and,%
        \begin{subequations}%
        \label{eq:fb:growth:tracking:local:assumptions}
        \begin{align}
            0 & \le \MF \defeq M - \Gamma_F,
            \quad
            \Gamma_G + p\Gamma_F \ge (p-1) M + \tilde\gamma \Lambda
            \quad\text{and}\\
            0 &\le \LL \defeq (1+\trackingressum^2\inv{\tilde\gamma})\Lambda
            \le (1-\eta)M.
        \end{align}
        \end{subequations}
    \end{enumerate}
\end{theorem}

\begin{proof}
    Throughout the proof, $k \in \N$.

    \Cref{item:fb:growth:tracking:weak}:
    By \cref{lemma:smoothness:descent} and the subdifferentiability of $G$, we have
    \[
        \dualprod{F'(\thisx)+ \subdiff G(\nextx)}{\nextx-\thisx}_{X^*,X}
        \ge
        [F+G](\nextx) - [F+G](\thisx)
        -\frac{1}{2}\norm{\nextx-\thisx}_\Lambda^2.
    \]
    Combining this with \cref{thm:weaker:smoothness} for $\basex=\thisx$  establishes
    \[
        \dualprod{\estdiff F(\thisx) + \subdiff G(\nextx)}{\nextx - \thisx}_{X^*,X}
        \ge
        F(\nextx)-F(\thisx)
        - \frac{1}{2} \norm{\nextx - \thisx}_{\LL}^2
        - \frac{1}{2\tilde\gamma}e_{p,k}
    \]
    with $\sup_{N \in \N} \sum_{k=0}^{N-1} p^k e_{p,k} < \Psi_p$ whenever $\{x^n\}_{n=0}^k \subset \localset$.
    This verifies \eqref{eq:fb:descent:xk} and $\sup_{N \in \N} \sum_{k=0}^{N-1} \errDesc{k} \le \rDesc$ with $\errDesc{k}=e_{p,k}/(2\tilde\gamma)$.
    Since we assume $\LL \le 2(1-\eta)M$, \cref{ass:fb:descent}\,\cref{item:fb:descent:xk,item:fb:descent:error} consequently hold.
    Because $\Xi=0$, \cref{item:fb:descent:coercitivity} requires $[F+G](x^N) \le \rDesc + [F+G](x^0)$ to imply $x^N \in \localset$. This holds whenever $\localset \supset \sublev_{\Psi_p/(2\tilde\gamma)+[F+G](x^0)}(F+G)$, as we have assumed.
    Likewise, we prove \cref{item:fb:descent:lb} with the lower bound $\inf [F+G] - [F+G](x^0) > -\infty$.

    \Cref{item:fb:growth:tracking:approx-cont}:
    We have
    $
        \inf_{x_{k+1}^* \in H(\nextx)} \norm{x_{k+1}^*-\totalest{k+1}}_{X^*}
        \le
        \norm{F'(\nextx)-\estdiff F(\thisx)}_{X^*}
    $
    through the choice
    \[
        x_{k+1}^* =  F'(\nextx) - \estdiff F(\thisx) + \totalest{k+1} \in F'(\nextx) + \subdiff G(\nextx) + \Xi\nextx = H(\nextx).
    \]
    Hence, \cref{lemma:semicon:props}\,\cref{item:semicon:props:bound,item:semicon:props:triangle}, followed by the $\Lambda$-$*$-Lipschitz assumption on $F'$ and \cref{thm:weaker:erroneous-lipschitz} establish
    \begin{gather}
        \label{eq:fb:tracking:lipschitz-like}
        \begin{split}
        \inf_{x_{k+1}^* \in H(\nextx)} \norm{x_{k+1}^*-\totalest{k+1}}_{X^*}
        &
        \le
        \norm{\Lambda}_{\linear(X; X^*)}^{1/2} \inf_{x_{k+1}^* \in H(\nextx)} \norm{F'(\nextx)-\estdiff F(\thisx)}_{\Lambda,*}
        \\
        &
        \le
        \norm{\Lambda}_{\linear(X; X^*)}^{1/2}\left(
            \norm{F'(\thisx) - F'(\nextx)}_{\Lambda,*}
            +
            \norm{\estdiff F(\thisx) - F'(\thisx)}_{\Lambda,*}
        \right)
        \\
        &
        \le
        \norm{\Lambda}_{\linear(X; X^*)}^{1/2}\left(
            \norm{\nextx-\thisx}_\Lambda
            +
            \errLip{k}^{1/2}
        \right),
        \end{split}
    \shortintertext{where the $\errLip{k} \ge 0$ satisfy \eqref{eq:weaker:erroneous-lipschitz:error-bound}, hence}
        \label{eq:fb:tracking:m-sum-bound}
        \sum_{n=0}^{k-1} \errLip{n}
        \le \Psi_1 + \trackingressum[1]\sum_{n=0}^{k-1}\norm{x^{n+1}-x^n}_\Lambda
        \le \Psi_1 + c\trackingressum[1]\sum_{n=0}^{k-1}\norm{x^{n+1}-x^n}_M.
    \end{gather}
    Now, $\sup_{N \in \N} \sum_{k=0}^{N-1}\norm{\nextx - \thisx}_M^2 < \infty$
    implies $\lim_{k \to \infty} \inf_{x_{k+1}^* \in H(\nextx)} \norm{x_{k+1}^*-\totalest{k+1}}_{X^*} = 0$ through \cref{eq:fb:tracking:lipschitz-like,eq:fb:tracking:m-sum-bound}.
    This establishes \cref{ass:fb:approx-cont}.

    For the verification of both \cref{item:fb:growth:tracking:mono-local,item:fb:growth:tracking:local}, we observe that by our choice of $\errMono{k}(\basex)=e_{p,k}/(2\tilde\gamma)$, the definition of $r_p$ in \cref{ass:fb:non-escape}, and \cref{thm:weaker:smoothness}, we have
    \begin{equation}
        \label{eq:growth:tracking:rp}
        r_p
        \defeq
        \sup_{N \in \N} p^{-N} \sum_{k=0}^{N-1} \frac{p^k e_{p,k}}{2\tilde\gamma}
        \le
        \sup_{N \in \N} \sum_{k=0}^{N-1} \frac{p^k e_{p,k}}{2\tilde\gamma} < \frac{\Psi_p}{2\tilde\gamma}.
    \end{equation}
    Hence, \eqref{eq:fb:growth:init} verifies \eqref{eq:fb:local:error-assumption} and  $x^0 \in \openball_{\MF}(\basex, \sqrt{\delta^2 - 2r_p})$.
    We have also explicitly assumed the remaining neighbourhood conditions of \cref{ass:fb:non-escape}, as well as $0 \le \LL \le (1-\eta)M$, so only need to verify the respective \eqref{eq:fb:three-point-monotonicity:approx} or \eqref{eq:fb:three-point-smoothness:approx}.

    \Cref{item:fb:growth:tracking:mono-local}:
    Suppose $\{x^n\}_{n=0}^k \subset \localset$.
    By \cref{lemma:smoothness:monotonicity} and \eqref{eq:fb:growth:tracking:mono-local:assumptions}, we have
    \[
        \dualprod{F'(\thisx) + \subdiff G(\nextx) + \Xi\nextx}{\nextx-\basex}_{X^*,X}
        - (1-\zeta)\fakenormsq{\thisx-\basex}{\Gamma_F}
        \ge
         \fakenormsq{\nextx-\basex}{\Gamma_G}
        - \frac{1}{4\zeta}\norm{\nextx - \this x}_{\Lambda}^2.
    \]
    Due to \eqref{eq:fb:implicit:thm}, we have $-M(\nextx-\thisx) = \totalest{k+1} \in \estdiff F(\thisx) + \subdiff G(\nextx) + \Xi\nextx$.
    Therefore, combining the previous inequality with \cref{thm:weaker:smoothness} gives
    \[
        \dualprod{\totalest{k+1}}{\nextx - \basex}_{X^*,X}
        - (1-\zeta)\fakenormsq{\thisx-\basex}{\Gamma_F}
        \ge
        \fakenormsq{\nextx-\basex}{\Gamma_G - (\tilde\gamma/2) \Lambda}
        - \frac{1}{2}\norm{\nextx - \this x}_{\LL}^2
        - \errMono{k}(\basex).
    \]
    This verifies \eqref{eq:fb:three-point-monotonicity:approx}
    with $\GF = (1-\zeta)\Gamma_F$ and $\GG = \Gamma_G - (\tilde\gamma/2) \Lambda$.
    \Cref{ass:fb:non-escape} option \cref{item:fb:non-escape:any} requires
    $M+2\GG \ge p\MF$ for $\MF \defeq M-2\GF \ge 0$.
    That is, $M + 2(\Gamma_G - (\tilde\gamma/2) \Lambda) \ge p(M-2(1-\zeta)\Gamma_F)$
    and $M \ge 2(1-\zeta)\Gamma_F$.
    The former reorganises as $2\Gamma_G + 2p(1-\zeta)\Gamma_F \ge (p-1) M + \tilde\gamma \Lambda$. We have assumed both conditions.

    \Cref{item:fb:growth:tracking:local}:
    Suppose  $\{x^n\}_{n=0}^k \subset \localset$.
    Since $\localset$ is convex, \cref{lemma:smoothness:nonconvex,thm:weaker:smoothness}, and the definition of $\LL$ in \eqref{eq:fb:growth:tracking:local:assumptions} give
    \[
        \begin{split}
        \dualprod{F'(\thisx) + \subdiff G(\nextx) + \Xi\nextx}{\nextx-\basex}_{X^*,X}
        - \frac{1}{2}\fakenormsq{\thisx-\basex}{\Gamma_F}
        &
        \ge
        \gap(\nextx; \basex)
        \\
        \MoveEqLeft[-1]
        + \frac{1}{2}\fakenormsq{\nextx-\basex}{\Gamma_G-\tilde\gamma\Lambda}
        - \frac{1}{2}\norm{\nextx - \this x}_{\LL}^2.
        \end{split}
    \]
    Similarly to the claim \cref{item:fb:growth:tracking:mono-local}, we now verify \eqref{eq:fb:three-point-smoothness:approx} with $\GF = \Gamma_F$ and $\GG = \Gamma_G - \tilde\gamma \Lambda$ by combining this inequality with  \cref{thm:weaker:smoothness}.
    with $\GF = (1-\zeta)\Gamma_F$ and $\GG = \Gamma_G - \tilde\gamma \Lambda$.
    \Cref{ass:fb:non-escape} option \cref{item:fb:non-escape:min} requires
    $M+\GG \ge p\MF$ for $\MF \defeq M-\GF \ge 0$.
    That is, $\GG + p\GF \ge (p-1) M + \tilde\gamma \Lambda$ and $M \ge \Gamma_F$, which we have assumed.
\end{proof}

\begin{remark}[error term]
    Recalling \cref{rem:tracking:error-sum}, we can make $\rDesc \ge 0 $ and $r_p \ge 0$ arbitrarily small by taking high-quality first inner and adjoint steps.
\end{remark}

\begin{remark}[linear convergence]
    From \eqref{eq:growth:tracking:rp} and \cref{thm:weaker:smoothness}, we see that through good-quality first inner and adjoint steps, \eqref{eq:fb:local:error-assumption:strong} can be made to hold.
    Therefore, when \cref{thm:fb:growth:tracking} verifies \cref{ass:fb:non-escape} for $p>1$, the linear convergence \cref{cor:fb:local:linear} is applicable.
\end{remark}

\subsection{Forward-backward splitting}
\label{sec:outer-examples:fb}

We now interpret \cref{thm:fb:growth:tracking} for both standard exact forward-backward splitting in a Hilbert space, as well as outer forward-backward splitting when we construct $\estgrad F$ with inner and adjoint methods that satisfy the tracking theory of \cref{sec:tracking}; in particular, the methods of \cref{sec:inner-adjoint}.
We write $\Inj: X \hookrightarrow X^*$, $x \mapsto \iprod{x}{\freevar}_X$ for the standard injection from the Hilbert space $X$ to its dual. Then $\norm{\freevar}_\Inj = \norm{\freevar}_X$.

For clarity of the statement of the next corollary, we omit any mention of parameters that are not important for the algorithm itself, and that can be deduced from the other choices.

\begin{corollary}[Inexact outer forward-backward splitting on a Hilbert space]
    \label{cor:fb:tracking}
    On a Hilbert space $X$, suppose $F: X \to \R$ has an $L$-Lipschitz Fréchet derivative in $\localset \subset X$, and $G: X \to \extR$ is convex, proper, and lower semicontinuous.
    Pick a step length parameter $\tau>0$, and for all $k \in \N$, construct $\nextestgrad$ obeying \cref{ass:tracking:main} in $\localset$ for the distances
    \begin{equation*}
        \dX=L\norm{\freevar}_{X}
        \quad\text{and}\quad
        \dXstar=\inv L\norm{\freevar}_{X^*}.
    \end{equation*}
    Update
    \begin{equation}
        \label{eq:fb:corollary-inexact-fb}
        \nextx \defeq \prox_{\tau G}(\thisx - \tau\nextestgrad).
    \end{equation}
    Let $\trackingressum$, $\kappa$, $e_{p,k}$, and $\Psi_p$ be as in \cref{thm:weaker:smoothness}.
    Then:
    \begin{enumerate}[label=(\roman*)]
        \item\label{item:fb:tracking:weak}
        \Cref{ass:fb:descent} holds with $\rDesc = \Psi_p/(2\tilde\gamma)$ provided $\inf [F+G] > -\infty$, $\tau(1 + \trackingressum^2\inv{\tilde\gamma} + \tilde\gamma)L < 2$ and
        $
            \localset \supset \sublev_{\Psi_p/(2\tilde\gamma)+[F+G](x^0)}(F+G)
        $
        for some $\tilde\gamma>0$.
        \item\label{item:fb:tracking:approx-cont}
        \Cref{ass:fb:approx-cont} holds.
    \end{enumerate}

    Suppose further that $G$ is $\gamma_G$-strongly subdifferentiable in $X$, and $F'$ is $\gamma_F$-monotone in $\localset$ for some $\gamma_F, \gamma_G \in \R$,
    and that $\openball(\basex, \delta/\bar m) \subset \localset $ for a base point $\basex \in X$, $\delta>0$, and $\bar m$ defined below in \cref{item:fb:tracking:mono-local} or \cref{item:fb:tracking:local}.
    Pick $\tilde\gamma>0$ and $p \in [1, \kappa)$.
    If
    \[
        x^0 \in \openball\left(\basex, \inv{\bar m}\sqrt{\delta^2-\Psi_p/\tilde\gamma}\right)
        \quad\text{with}\quad
        \Psi_p < \delta^2\tilde\gamma,
    \]
    then:
    \begin{enumerate}[resume*]
        \item\label{item:fb:tracking:mono-local}
        \Cref{ass:fb:non-escape} option \cref{item:fb:non-escape:any} holds if $0 \in F'(\basex) + \subdiff G(\basex)$, and, for some $\zeta \in (0, 1]$,
        \[
            2\tau(1-\zeta)\gamma_f < 1,
            \quad
            2\gamma_g + 2p(1-\zeta)\gamma_f \ge (p-1)\inv\tau + \tilde\gamma L,
            \quad\text{and}\quad
            0 \le \tau(\inv{(2\zeta)}+\trackingressum^2\inv{\tilde\gamma})L < 1.
        \]
        In this case $\MF = \bar m \Inj$ for $\bar m = \inv\tau - 2(1-\zeta)\gamma_f > 0$.
        \item\label{item:fb:tracking:local}
        \Cref{ass:fb:non-escape} option \cref{item:fb:non-escape:min} holds if $\localset$ is convex,
        $\inf_{x \in \localset} [F+G](x) \ge [F+G](\basex)$ and,
        \[
            \tau \gamma_f < 1,
            \quad
            \gamma_g + p\gamma_f \ge (p-1)\inv\tau + \tilde\gamma L,
            \quad\text{and}\quad
            0  \le \tau(1+\trackingressum^2\inv{\tilde\gamma})L < 1.
        \]
        In this case $\MF = \bar m \Inj$ for $\bar m = \inv\tau - \gamma_f > 0$.
    \end{enumerate}
\end{corollary}

\begin{proof}
    We apply \cref{thm:fb:growth:tracking}, whose conditions we need to verify.
    In its operator-relative framework, we take $\Xi=0$, $M=\inv\tau \Inj$, $\Lambda=L\Inj$, $\Gamma_F=\gamma_f\Inj$, $\Gamma_G=\gamma_G\Inj$.
    Then the implicit step \eqref{eq:fb:implicit:thm} holds by \eqref{eq:fb:corollary-inexact-fb}, and the condition $\Lambda \le cM$ for a $c>0$ of \cref{thm:fb:growth:tracking}\,\cref{item:fb:growth:tracking:approx-cont} reduces to $\tau L  \le c$, which automatically holds.
    We recall from \cref{ex:operator-reg:hilbert} that in Hilbert spaces with $\Lambda=L\Inj$, both the $\Lambda$-$*$-cocoercivity and $\Lambda$-$*$-Lipschitz properties are equivalent to the basic $L$-Lipschitz property of $f'$.
    Further, when
    \begin{enumerate}[nosep]
        \item
        in \cref{thm:fb:growth:tracking}\,\cref{item:fb:growth:tracking:weak}, we take
        $0 \le \bar\lambda \defeq (1 + \trackingressum^2\inv{\tilde\gamma}+\tilde\gamma)L < 2(1-\eta)\inv\tau$. (This holds for some $\eta>0$ by our step length assumption
        $\tau(1 + \trackingressum^2\inv{\tilde\gamma})L + \tilde\gamma < 2$.)
        \item in \cref{thm:fb:growth:tracking}\,\cref{item:fb:growth:tracking:mono-local} we take $\eta \defeq 1-\tau\bar\lambda>0$ for $0 \le \bar\lambda \defeq (1+\trackingressum^2\inv{\tilde\gamma})L < \inv\tau$;
        \item in \cref{thm:fb:growth:tracking}\,\cref{item:fb:growth:tracking:local} we take $\eta \defeq 1-\tau\bar\lambda>0$ for $0 \le \bar\lambda \defeq (\inv{(2\zeta)}+\trackingressum^2\inv{\tilde\gamma})L < \inv\tau$;
    \end{enumerate}
    and in each case we take $\LL = \bar\lambda\Inj$ and $\tilde\eta=0$, then the conditions of \cref{thm:fb:growth:tracking}\,\cref{item:fb:growth:tracking:weak,item:fb:growth:tracking:local,item:fb:growth:tracking:mono-local} readily translate to the present respective conditions.
\end{proof}

\begin{remark}
    If $p=1$, all the step length conditions in the corollary will hold by taking first $\tilde\gamma>0$ small enough, and then $\tau>0$ small enough.
    If $p>1$ is desired (to use the linear convergence \cref{cor:fb:local:linear}), we need $\gamma_g+\gamma_f>0$, and take also $p>1$ sufficiently small.
\end{remark}

For exact forward-backward splitting with $\nextestdiff = F'(\thisx)$, by taking $p=1$, $\trackingressum=0$, and $e_{p,k}=0$, and then letting $\tilde\gamma\downto 0$ in the previous result, we immediately obtain the following corollary.
In \cref{item:fb:tracking:weak:exact}, we even take $\localset=X$, since the tracking inequalities now trivially work with that choice, and we do not assume any local properties form $F$ and $G$; in \cref{item:fb:tracking:mono-local:exact,item:fb:tracking:local:exact} we do.

\begin{corollary}[Exact outer forward-backward splitting on a Hilbert space]
    \label{cor:fb:exact}
    On a Hilbert space $X$, suppose $F: X \to \R$ has an $L$-Lipschitz Fréchet derivative in $\localset \subset X$, and $G: X \to \extR$ is convex, proper, and lower semicontinuous.
    Pick a step length parameter $\tau>0$.
    Update
    \[
        \nextx \defeq \prox_{\tau G}(\thisx - \tau\grad F(\thisx)).
    \]
    Then:
    \begin{enumerate}[label=(\roman*)]
        \item\label{item:fb:tracking:weak:exact}
        \Cref{ass:fb:descent} holds provided $\inf [F+G] > -\infty$, and
        $
            0 \le \tau L < 2.
        $
        \item\label{item:fb:tracking:approx-cont:exact}
        \Cref{ass:fb:approx-cont} holds.
    \end{enumerate}

    Suppose further that $G$ is $\gamma_G$-strongly subdifferentiable in $X$, and $F'$ is $\gamma_F$-monotone in $\localset$ for some $\gamma_F, \gamma_G \in \R$,
    and that $\openball(\basex, \delta/\bar m) \subset \localset $ for a base point $\basex \in X$, $\delta>0$, and $\bar m$ defined below in \cref{item:fb:tracking:mono-local:exact} or \cref{item:fb:tracking:local:exact}.
    Pick $p \in [1, \kappa)$.
    If $x^0 \in \openball(\basex, \delta/\bar m)$, then:
    \begin{enumerate}[resume*]
        \item\label{item:fb:tracking:mono-local:exact}
        \Cref{ass:fb:non-escape} option \cref{item:fb:non-escape:any} holds if $0 \in F'(\basex) + \subdiff G(\basex)$, and for some $\zeta \in (0, 1]$, we have
        \[
            2\tau(1-\zeta)\gamma_f > 1,
            \quad
            2\gamma_g + 2p(1-\zeta)\gamma_f \ge (p-1)\inv\tau,
            \quad\text{and}\quad
            \tau L < 2\zeta.
        \]
        In this case $\MF = \bar m \Inj$ for $\bar m = \inv\tau - 2(1-\zeta)\gamma_f > 0$.

        \item\label{item:fb:tracking:local:exact}
        \Cref{ass:fb:non-escape} option \cref{item:fb:non-escape:min} holds if $\localset$ is convex, $\inf_{x \in \localset} [F+G](x) \ge [F+G](\basex)$, and,
        \[
            \tau \gamma_f > 1,
            \quad
            \gamma_g + p\gamma_f \ge (p-1)\inv\tau,
            \quad\text{and}\quad
            \tau L < 1.
        \]
        In this case $\MF = \bar m \Inj$ for $\bar m = \inv\tau - \gamma_f > 0$.
    \end{enumerate}
\end{corollary}

Now that we have provided step length and growth conditions that prove \cref{,ass:fb:descent,ass:fb:non-escape,ass:fb:approx-cont} for both exact and single-loop forward-backward splitting for bilevel problems, we can use \cref{thm:fb:subdiff,thm:fb:local:values,thm:fb:local:weak} to prove convergence.

\subsection{Primal-dual proximal splitting}
\label{sec:outer-examples:pdps}

We consider now the problem \eqref{eq:fb:pdps:problem}, i.e.,
\begin{equation}
    \label{eq:pdps:problem}
    \min_{z \in Z}~ f(z) + g(z) + h(Kz)
    =
    \min_{z \in Z} \max_{y \in Y}~ f(z) + g(z) + \dualprod{y}{Kz}_{Y,Y^*} - h_*(y),
\end{equation}
where $Z$ and $Y$ are normed spaces equipped with self-adjoint and positive semi-definite $M_z \in \linear(Z; Z^*)$ and $M_y \in \linear(Z; Z^*)$.
The functions $g: Z \to \extR$, $h_*: Y \to \extR$, and $h=(h_*)^*$ are convex, proper, and lower semicontinuous, and $f: Z \to \R$ possibly non-convex but Fréchet differentiable with $LM_z$-$*$-Lipschitz Fréchet derivative in $\zlocalset \subset Z$ for an $L \ge 0$.
Moreover, $K \in \linear(Z; Y^*)$.

We represent the problem and method in the implicit forms \eqref{eq:fb:pdps:problem}  and \eqref{eq:fb:pdps:implicit} with
\begin{equation}
    \label{eq:pdps:problem-functions}
    F(z, y) \defeq f(z),
    \quad
    G(z,y) \defeq g(z) + h_*(y),
    \quad\text{and}\quad
    \Xi \defeq \begin{pmatrix} 0 & K^* \\ -K & 0 \end{pmatrix}
\end{equation}
while the inexact PDPS becomes an instance of \eqref{eq:fb:implicit:thm} with
\begin{equation}
    \label{eq:pdps:algorithm-functions}
    \estdiff F(\this z, \thisy) \defeq \begin{pmatrix}
        \estdiff f(\this z) \\
        0
    \end{pmatrix}
    \quad\text{and}\quad
    M \defeq
    \begin{pmatrix}
        \inv \tau M_z, & -K^* \\
        -K & \inv\sigma M_y
    \end{pmatrix}
\end{equation}

\begin{remark}
    $M_z$ and $M_y$ generate (semi-)norms in $Z$ and $Y$, respecting the Pythagoras' identity \eqref{eq:norms:pythagoras}.
    In Hilbert spaces, we can simply take $M_z: X \hookrightarrow X^*$ and $M_y: Y \hookrightarrow Y^*$ as the standard injections, to obtain a standard Hilbert space algorithm
    \[
        \left\{
            \begin{array}{l}
                \nexxt z \defeq \prox_{\tau g}(\this z - \tau \estgrad f(\this z) - \tau K^*\thisy), \\
                \nexty \defeq \prox_{\sigma h_*}(\thisy + \sigma K(2\nexxt z - \this z)).
            \end{array}
        \right.
    \]
\end{remark}

We start by extending the standard step length assumption $\tau L + \tau\sigma\norm{K}^2 \le 1$ of the PDPS into normed spaces.
In the next assumption, in the Hilbert space setting with $M_y$ and $M_z$ the standard injections, we can take $K_z=K$ and $K_y=\Id$.

\begin{assumption}[PDPS step length condition]
    \label{ass:pdps:step-length-cond}
    Suppose $K=K_y^* K_z$ for some $K_z \in \linear(Z; V^*)$, $K_y \in \linear(Y; V)$, and a normed space $V$.
    Given $\lambda \ge 0$, the step length parameters $\tau,\sigma>0$ satisfy
    \[
        \norm{K_y \freevar}_{V} \le \norm{\freevar}_{M_y}
        \quad\text{and}\quad
        (\tau\lambda-1)\norm{\freevar}_{M_z}^2 + \tau\sigma\norm{K_z\freevar}_{V^*}^2 \le 0.
    \]
\end{assumption}

\begin{lemma}[PDPS preconditioning operator]
    \label{lemma:pdps:m-estim}
    If \cref{ass:pdps:step-length-cond} holds, then $M$ is positive semi-definite and for any $\gamma_z, \gamma_y \ge 0$ and $\gamma \defeq \min\{\gamma_z\tau,\gamma_y\sigma\}/2$, we have
    \[
        \lambda\diag(M_z,\; 0)
        \le M
        \quad\text{and}\quad
        \gamma M \le
        \diag(\gamma_z M_z,\; \gamma_y M_y).
    \]
\end{lemma}

\begin{proof}
    By a simple application of Young's inequality and \cref{ass:pdps:step-length-cond}, we have
    \[
        \begin{split}
        \norm{(z, y)}_{M}^2
        &
        = \inv\tau \norm{z}_{M_z}^2 + \inv\sigma \norm{y}_{M_y}^2 - 2\dualprod{K_z z}{K_y y}_{V^*,V}
        \\
        &
        \ge
        \left(\inv\tau \norm{z}^2_{M_z} - \sigma\norm{K_z z}_{V^*}^2\right)
        +\inv\sigma\left(\norm{y}^2_{M_y} - \norm{K_y y}_{V}^2\right)
        \ge
        \lambda \norm{z}^2_{M_z}
        \end{split}
    \]
    for any $x=(z,y) \in Z \times Y$.
    This establishes the first claimed inequality.
    The second follows by using Young's inequality and \cref{ass:pdps:step-length-cond} to establish
    \[
        \gamma \norm{(z, y)}_{M}^2
        \le
        \gamma\left(\inv\tau \norm{z}^2_{M_z} + \sigma\norm{K_z z}_{V^*}^2\right)
        +\gamma\inv\sigma\left(\norm{y}^2_{M_y} + \norm{K_y y}_{V}^2\right)
        \le
        \frac{2\gamma}{\tau}\norm{z}^2_{M_z}
        + \frac{2\gamma}{\sigma}\norm{y}^2_{M_y}.
        \qedhere
    \]
\end{proof}

We can now translate \cref{thm:fb:growth:tracking} to the outer PDPS of \cref{ex:fb:pdps}.
It is missing the verification of \cref{ass:fb:descent}\,\cref{item:fb:descent:coercitivity,item:fb:descent:lb} for subdifferential convergence.
Because $\Xi$ is not cyclically monotone (see \cite[Chapter 24]{rockafellar-convex-analysis}), we see no way in general for the PDPS to satisfy that property.\footnote{%
    However, we could try to enforce the conditions, monitoring for convergence failure by setting expected bounds on
    \[
            \sum_{k=0}^{N-1} \gap(\nextx; \thisx)
            = [F+G](x^N) - [F+G](x^0)
            - \sum_{k=0}^{N-1} \dualprod{\Xi\nextx}{\thisx}.
    \]
    In fact, if $\inf F+G > -\infty$, we only need to ensure that the latter sum term sum stays within chosen bounds, without having to calculate potentially costly function values.
}

\begin{theorem}[PDPS with inexact $\estdiff f$; everything else exact]
    \label{thm:pdps:inexact-expansion}
    Assume the setup of \cref{eq:pdps:problem,ass:pdps:step-length-cond} for some $\tau,\sigma,\lambda>0$.
    Suppose that \cref{ass:tracking:main} holds for $f$ in $\zlocalset$ with the distances
    \[
        \dX=\norm{\freevar}_{L M_z}
        \quad\text{and}\quad
        \dXstar=\norm{\freevar}_{L M_z, *}.
    \]
    Then
    \begin{enumerate}[label=(\roman*)]
        \item\label{item:pdps:inexact-expansion:weak}
        \Cref{ass:fb:approx-cont} holds.
    \end{enumerate}

    Suppose further that $g$ and $h_*$ are, respectively, $\gamma_g M_z$-subdifferentiable in $X$ and $\gamma_{h_*} M_y$-subdifferentiable in $Y$, and that $f'$ is $\gamma_f M_z$-subdifferentiable in $\zlocalset$ for some $\gamma_g,\gamma_{h_*},\gamma_f \ge 0$.
    Let $\MF$ and $\bar M_z$ be as defined below in \cref{item:pdps:inexact-expansion:mono} or \cref{item:pdps:inexact-expansion:smoothness}.
    Suppose $\openball_{\bar M_z}(\basez, \delta_z) \subset \zlocalset$ for some primal base point $\basez$ and $\delta_z>0$.
    Let $\basex \in \{\base{z}\} \times \Dom h_*$ and $\localset \defeq \zlocalset \times \Dom h_*$.
    Pick $\tilde\gamma>0$ and  $p \in [1, \kappa)$.
    If
    \begin{equation}
        \label{eq:pdps:init-cond}
        x^0 = (z^0, y^0) \in \openball_{\MF}\left(\basex, \sqrt{\smash[b]{\lambda^2\delta_z^2 - \Psi_p/\tilde\gamma}}\right)
        \quad\text{with}\quad
        \Psi_p < \lambda^2\delta_z^2 \tilde\gamma,
    \end{equation}
    then:
    \begin{enumerate}[resume*]
        \item\label{item:pdps:inexact-expansion:mono}
        \Cref{ass:fb:non-escape} option \cref{item:fb:non-escape:any} holds if $\basex \in \inv H(0)$, $f'$ is $LM_z$-$*$-cocoercive\footnote{%
            Recall from \cref{lemma:opeartor-reg:coco-lipschitz} that this implies the earlier-assumed $\Lambda$-$*$-Lipschitz property.
        }
        in $\zlocalset$, and, for some $\zeta \in (0, 1]$, and $\eta\ge0$,
        \begin{subequations}%
        \label{eq:pdps:inexact-expansion:mono}
        \begin{align}
            \label{eq:pdps:inexact-expansion:mono:gamma-cond}
            \lambda & \le 2(1-\zeta)\gamma_f,
            \quad
            \min\{(2\gamma_g + 2p(1-\zeta)\gamma_f - \tilde\gamma L)\tau,\gamma_{h^*}\sigma\}/2
            \ge p-1,
            \quad\text{and}
            \\
            \label{eq:pdps:inexact-expansion:mono:lambda-cond}
            0 & \le \bar\lambda \defeq L(\inv{(2\zeta)} + \trackingressum^2\inv{\tilde\gamma}))
            \le (1-\eta)\lambda .
        \end{align}
        \end{subequations}
        In this case $\MF = M - 2(1-\zeta)\gamma_f \diag(M_z, 0) \ge 0$ and
        $\bar M_z = (\lambda-2(1-\zeta)\gamma_f)M_z \ge 0$.

        \item\label{item:pdps:inexact-expansion:smoothness}
        \Cref{ass:fb:non-escape} option \cref{item:fb:non-escape:min} holds if $\zlocalset$ is convex, $\inf_{x \in \localset} \gap(x; \basex) \ge 0$, and, for some $\eta\ge0$,
        \begin{subequations}%
        \label{eq:pdps:inexact-expansion:smoothness}
        \begin{align}%
            \label{eq:pdps:inexact-expansion:smoothness:gamma-cond}
            \lambda & \le \gamma_f,
            \quad
            \min\{(\gamma_g + p\gamma_f - \tilde\gamma L)\tau,\gamma_{h^*}\sigma\}/2
            \ge p-1,
            \quad\text{and}
            \\
            \label{eq:pdps:inexact-expansion:smoothness:lambda-cond}
            0 \le \bar\lambda & \defeq L( 1 + \trackingressum^2\inv{\tilde\gamma})
            \le (1-\eta)\lambda.
        \end{align}
        \end{subequations}
        In this case $\MF = M - \gamma_f \diag(M_z, 0) \ge 0$
        and $\bar M_z = (\lambda-\gamma_f)M_z \ge 0$.
    \end{enumerate}
\end{theorem}

\begin{proof}
    Recall the definitions \cref{eq:pdps:problem-functions,eq:pdps:algorithm-functions}.
    Observe that $F'$ is $\Lambda$-$*$-Lipschitz ($\Lambda$-$*$-cocoercive in \cref{item:pdps:inexact-expansion:mono}) and $\Gamma_F$-monotone, and $G$ is $\Gamma_G$-strongly convex for
    \[
        \Lambda \defeq
        \diag(LM_z,\; 0),
        \quad
        \Gamma_F \defeq
        \diag(\gamma_f M_z,\; 0),
        \quad\text{and}\quad
        \Gamma_G \defeq
        \diag(\gamma_g M_z,\; \gamma_{h_*} M_y).
    \]

    To use \cref{thm:fb:growth:tracking}, we directly verify \cref{thm:weaker:smoothness,thm:weaker:erroneous-lipschitz} for $F'$ and $\estdiff F$, instead of proving \cref{ass:tracking:main} for the extended functions.
    By \cref{thm:weaker:smoothness} applied to $f$ and $\estdiff f$, if $\{x^n = (z^n, y^n)\}_{n=0}^k \subset \localset$, we have for $\nexxt x = (\nexxt z, \nexxt y)$ and any $\bar x = (\bar z, \bar y)$ that
    \[
        \begin{split}
        \dualprod{\estdiff F(\this x)-F'(\thisx)}{\nextx - \bar x}_{X^*,X}
        &
        =
        \dualprod{\estdiff f(\this z)-f'(\this z)}{\nexxt z - \bar z}_{Z^*,Z}
        \\
        &
        \ge
        -\frac{\tilde\gamma}{2} \norm{\nexxt z - \bar z}_{L M_z}^2
        - \frac{\trackingressum^2}{2\tilde\gamma}\norm{\nexxt z - \this z}_{LM_z}
        - \frac{1}{2\tilde\gamma}e_{p,k}
        \\
        &
        =
        -\frac{\tilde\gamma}{2} \norm{\nextx-\bar x}_\Lambda^2
        - \frac{\trackingressum^2}{2\tilde\gamma}\norm{\nextx-\thisx}_\Lambda^2
        - \frac{1}{2\tilde\gamma}e_{p,k}.
        \end{split}
    \]
    Recalling the choice of distances \eqref{eq:fb:distances}, this verifies \cref{thm:weaker:smoothness} for $F$ and $\estdiff F$ in $\localset$.

    If $\{x^n = (z^n, y^n)\}_{n=0}^k \subset \localset$, \cref{thm:weaker:erroneous-lipschitz} applied to $f$ and $\estdiff f$ now establishes
    \begin{equation*}
        \norm{\nextestdiff-F'(\thisx)}_{\Lambda,*}^2
        =
        \norm{\estdiff f(\this z) - f'(\this z)}_{L M_z,*}^2
        \le
        \errLip{k},
    \end{equation*}
    where the $\errLip{k}$ satisfy
    $
        \sum_{n=0}^{k-1} \errLip{n}
        \le
        \Psi_1 + \trackingressum[1]\sum_{n=0}^{k-1}\norm{z^{n+1}-z^n}_{L M_z}
        =
        \Psi_1 + \trackingressum[1]\sum_{n=0}^{k-1}\norm{x^{n+1}-x^n}_\Lambda.
    $
    This verifies \cref{thm:weaker:erroneous-lipschitz} for $F$ and $\estdiff F$.

    We proceed with proving our specific claims.
    The all rely on \cref{lemma:pdps:m-estim} proving $M \ge \lambda\diag(M_z,\; 0)$, hence $\Lambda \le (L/\lambda)M$, where  $\Lambda \ge 0$.

    \Cref{item:pdps:inexact-expansion:weak}:
    We use \cref{thm:fb:growth:tracking}\,\cref{item:fb:growth:tracking:approx-cont}.

    \Cref{item:pdps:inexact-expansion:smoothness}:
    We use \cref{thm:fb:growth:tracking}\,\cref{item:fb:growth:tracking:local}, whose specific conditions we need to verify.
    We start with \eqref{eq:fb:growth:tracking:mono-local:assumptions}, i.e.,
    \begin{align*}
        0 & \le \MF \defeq M - 2(1-\zeta)\Gamma_F,
        \quad
        2\Gamma_G + 2p(1-\zeta)\Gamma_F \ge (p-1) M + \tilde\gamma \Lambda,
        \quad\text{and}
        \\
        0 & \le \LL \defeq (\inv{(2\zeta)}+\trackingressum^2\inv{\tilde\gamma})\Lambda
        \le (1-\eta)M.
    \end{align*}
    The first condition follows from our assumption $\lambda \le 2(1-\zeta)\gamma_f$ in \eqref{eq:pdps:inexact-expansion:mono:gamma-cond}.
    The third condition follows, likewise, from  \eqref{eq:pdps:inexact-expansion:mono:lambda-cond}.
    For the second condition, \cref{lemma:pdps:m-estim} proves $M \le \diag(\inv\gamma\gamma_z M_z,\; \inv\gamma\gamma_{^*} M_y)$ for $\gamma_z \defeq \gamma_g + p(1-\zeta)\gamma_f - \tilde\gamma L/2 \ge 0$ and $\gamma \defeq \min\{\gamma_z\tau,\gamma_{h^*}\sigma\}/2$.
    With this, the second condition holds if $\gamma \ge p-1$, which we have assumed in  \eqref{eq:pdps:inexact-expansion:mono:gamma-cond}.
    This proves \eqref{eq:fb:growth:tracking:mono-local:assumptions}.

    Taking $\delta \defeq \lambda\delta_z$, \eqref{eq:pdps:init-cond} implies, as required, $x^0 \in \openball_{\MF}(\basex, \sqrt{\delta^2 - \Psi_p/\tilde\gamma})$ and $\Psi_p < \tilde\gamma\delta^2$.
    By $M \ge \lambda\diag(M_z,\; 0)$\, we have
    $\MF \ge (\lambda-2(1-\zeta)\gamma_f)\diag(M_z, 0)=\diag(\bar M_z, 0)$,
    hence $\openball_{\MF}(\basex, \delta)  \subset \openball_{\bar M_z}(\basez, \delta_z) \times \Dom h_* \subset \zlocalset \times \Dom h_* = \localset$.
    By construction and assumption, we have $\LL \ge 0$.
    The claim now follows from \cref{thm:fb:growth:tracking}\,\cref{item:fb:growth:tracking:local}.

    \Cref{item:pdps:inexact-expansion:mono}: Completely analogous to \cref{item:pdps:inexact-expansion:smoothness}, based on \cref{thm:fb:growth:tracking}\,\cref{item:fb:growth:tracking:mono-local}.
\end{proof}

Now that we have provided step length and growth conditions that prove \cref{,ass:fb:descent,ass:fb:non-escape,ass:fb:approx-cont} for single-loop PDPS for bilevel problems, we can use \cref{thm:fb:subdiff,thm:fb:local:values,thm:fb:local:weak} to prove different forms of convergence.
In fact, further specialising \cref{thm:fb:local:values} to the PDPS, besides inexactness, as a novelty compared to  \cite{tuomov-nlpdhgm-redo,tuomov-nlpdhgm-general,tuomov-nlpdhgm-block,gao2023alternative}, subject to $h_*$ having a bounded domain, we get an estimate on the convex envelope of the objective, i.e., the Fenchel biconjugate.
In non-reflexive spaces, we define the latter as a function in $X$ instead of $X^{**}$ by taking first the conjugate and then the equivalently defined preconjugate: $h^{**} \defeq (h^*)_*$.

\begin{corollary}
    \label{cor:pdps:values}
    Let the assumptions of \cref{thm:pdps:inexact-expansion}\,\cref{item:pdps:inexact-expansion:mono} as well as \eqref{eq:pdps:inexact-expansion:smoothness} hold for $p=1$.
    Also suppose that $\Dom h_*$ is bounded.
    Then, for the \term{ergodic iterates} $\tilde z^N \defeq \frac{1}{N}\sum_{k=0}^{N-1} \this{z}$, for all $N \in \N$, we have
    \[
        [f + g + h \circ K]^{**}(\tilde z^N) \le [f + g + h \circ K](\base{z})
        + \sup_{\hat{y} \in \Dom h_*}
            \frac{1}{2N}\norm{(z^0, y^0)-(\base z, \hat{y})}_{\MF}^2
            + \frac{\sum_{k=0}^{N-1} e_{1,k}}{2\tilde\gamma N}.
    \]
    Here $[f + g + h \circ K](\base{z})=[f + g + h \circ K]^{**}(\base{z})$ if $\base{z}$ is a global minimiser of $f + g + h \circ K$.
\end{corollary}

\begin{proof}
    \Cref{thm:pdps:inexact-expansion}\,\cref{item:pdps:inexact-expansion:mono} with $p=1$ proves \cref{ass:fb:non-escape} option \cref{item:fb:non-escape:any} at $\basex$.
    Likewise, since we have assumed \eqref{eq:pdps:inexact-expansion:smoothness}, the proof of \cref{thm:pdps:inexact-expansion}\,\cref{item:pdps:inexact-expansion:smoothness} shows \eqref{eq:fb:three-point-smoothness:approx} and $\localset \defeq \zlocalset \times \Dom h_* \supset \openball_M(\basex, \delta)$ at any $\hat x \in \hat X \defeq \{\basez\} \times \Dom h_*$ with $\errDescBarEmph{k}(\hat x) = \errDescBar{k}(\hat x) = e_{1,k}/(2\tilde\gamma)$.
    \Cref{thm:fb:local:values} now establishes
    \begin{equation}
        \label{eq:pdps:local:values:ergodic}
        \sup_{\hat x \in \hat X} \sum_{k=0}^{N-1} \gap(\nextx; \hat x)
        \le
        \sup_{\hat x \in \hat X}\left(
            \frac{1}{2}\norm{x^0-\hat x}_{\MF}^2
            + \sum_{k=0}^{N-1}  \frac{e_{1,k}}{2\tilde\gamma}
        \right)
        \quad\text{for all}\quad N \in \N.
    \end{equation}

    Let $\hat x = (\base z, \hat y) \in \hat X$.
    With the expression of \cref{ex:fb:pdps:gap} for the gap, we expand and estimate using the definition of the Fenchel (bi)conjugate and $h^{**}=h$ as well as $[f+g]^{**} \le f+g$ that
    \[
        \begin{aligned}[t]
        \gap(\nextx; \hat x)
        &
        =
        ([f+g](\nexxt{z}) + \dualprod{K\nexxt{z}}{\hat{y}} - h_*(\hat{y})) - ([f+g](\base z) + \dualprod{K\base z}{\nexty} - h_*(\nexty))
        \\
        &
        \ge
        \bigl([f+g]^{**}(\nexxt{z}) + \dualprod{K\nexxt{z}}{\hat{y}} - h_*(\hat{y})\bigr) - [f + g + h \circ K](\base z).
        \end{aligned}
    \]
    Summing over $k \in \{0,\ldots,N-1\}$,
    taking the supremum over $\hat{y} \in \Dom h_*$, and
    using Jensen's inequality, therefore
    \[
        \sup_{\hat{y} \in \Dom h_*}\sum_{k=0}^{N-1}
        \gap(\nextx; \hat x)
        \ge
        N [(f+g)^{**} + h \circ K](\tilde z^N) - N[f + g + h \circ K](\base z).
    \]
    Denoting the infimal convolution by $\infconv$, we have
    \[
        f + g + h \circ K \ge [f+g + h \circ K]^{**} = ((f+g)^* \infconv [h \circ K]^*)^* = (f+g)^{**} + h \circ K.
    \]
    Moreover, the inequality is an equality at a global minimiser (or if $f$ is convex).
    Now the claim follows from \eqref{eq:pdps:local:values:ergodic}.
\end{proof}

\begin{remark}[Dual strong monotonicity not required]
    \label{rem:pdps:no-dual-sc}
    Inexact inner solutions force $\tilde\gamma>0$.
    In this case, $f+g$ has to be locally strongly subdifferentiable to satisfy \eqref{eq:pdps:inexact-expansion:mono:gamma-cond} or \eqref{eq:pdps:inexact-expansion:smoothness:gamma-cond}.
    When $p=1$, as in \cref{cor:pdps:values}, $h^*$, however, does not have to be strongly subdifferentiable.
    Practically, this means that we do not have to apply Moreau–Yosida regularisation to $h$.
    This is a significant improvement over \cite{jensen2022nonsmooth} and even over works on exact nonconvex PDPS; see \cite{tuomov-firstorder}.
    It largely arises from the more optimal analysis based on the splitting of $\GF$ and $\GG$ in \cref{eq:fb:three-point-monotonicity:approx,eq:fb:three-point-smoothness:approx}.
\end{remark}

\begin{remark}
    Taking $p>1$ in the proof of \cref{cor:pdps:values}, linear convergence rates could be obtained as in \cref{cor:fb:local:linear} for the iterates.
\end{remark}

We finally consider adjoint mismatch as in \cite{lorenz2023mismatch}, keeping everything else exact.

\begin{theorem}[PDPS with adjoint mismatch]
    \label{thm:fb:growth:mismatch}
    Assume the setup of \cref{ex:fb:pdps} with $\tau\sigma\norm{K}^2 \le 1$ and, for simplicity, $f=0$ and Hilbert $Z$ and $Y$.
    Suppose $\Dom h_*$ is bounded, and that $g$ and $h_*$ are, respectively, $\gamma_g $- and $\gamma_{h_*}$-strongly convex for some $\gamma_g>0$ and $\gamma_{h_*} \ge 0$.
    Let $\gamma \defeq \min\{\gamma_g\tau/4, \gamma_{h_*}\sigma/2\}$.
    In the PDPS \eqref{eq:pdps:alg}, replace $K^*$ with a “mismatched” adjoint $K^{*\approx}$.
    Then, for any $\basex \in Z \times Y$ and $p \in (1, 1+2\gamma]$, \cref{ass:fb:non-escape}\,\cref{item:fb:non-escape:any} holds with $\LL=0$, $\basexset=Z \times Y$, $\delta=\infty$, $r_p \le \varepsilon/(1-p)$, and
    \[
        \errMono{k}(\basex)=\frac{1}{2\gamma_g}\norm{(K^{*\approx}-K^*)\thisy}_{Z}^2 \le \varepsilon \defeq \frac{1}{2\gamma_g}(\norm{K^{*\approx}-K^*}\diam\Dom h_*)^2.
    \]
\end{theorem}

\begin{proof}
    With $M$, $G$, and $F$ given by \cref{ex:fb:pdps}, the abstract algorithm \eqref{eq:fb:implicit} reads
    \[
        -M(\nextx-\thisx)
        =:
        \totalest{k+1}
        =
        x_{k+1}^* + ((K^{*\approx}-K^*)\thisy, 0)
        \quad\text{for a}\quad x_{k+1}^* \in H(\nextx).
    \]
    Here $H$ is defined in \eqref{eq:fb:generic-problem}.
    Let $\basex \in \inv H(0)$.
    Using \cref{lemma:pdps:m-estim} in the final step, we estimate
    \[
        \begin{aligned}
        \dualprod{\totalest{k+1} -  H(\basex)}{\nextx - \bar x}_{X^*,X}
        &
        =
        \dualprod{\totalest{k+1} -  x_{k+1}^*}{\nextx - \bar x}_{X^*,X}
        + \dualprod{x_{k+1}^* -  H(\basex)}{\nextx - \bar x}_{X^*,X}
        \\
        &
        \ge
        \iprod{(K^{*\approx}-K^*)\thisy}{\nexxt{z}-\bar z}
        + \gamma_g \norm{\nexxt{z} - \bar z}_{Z}^2
        + \gamma_{h_*} \norm{\nexxt{y} - \bar y}_{Y}^2
        \\
        &
        \ge
        \frac{\gamma_g}{2} \norm{\nexxt{z} - \bar z}_{Z}^2
        + \gamma_{h_*} \norm{\nexxt{y} - \bar y}_{Y}^2
        - \frac{1}{2\gamma_g}\norm{(K^{*\approx}-K^*)\thisy}_{Z}^2
        \\[-0.5ex]
        &
        \ge
        \gamma \norm{\nextx - \bar x}_{M}^2
        - \errMono{k}(\basex).
        \end{aligned}
    \]
    Therefore, \eqref{eq:fb:three-point-monotonicity:approx} holds with $\GG = \gamma M$ and $\GF=0$.
    Moreover, we have
    $
        \sum_{k=0}^{N-1} p^{k-N}
        \le
        1/(p-1)
    $
    for any $p \in (1,1+2\gamma]$, verifying \eqref{eq:fb:local:error-assumption} and consequently \cref{ass:fb:non-escape}\,\cref{item:fb:non-escape:any}.
\end{proof}

\section{Numerical illustrations and application examples}
\label{sec:numerical}

We now treat two application examples: electrical impedance tomography and, as a demonstration that crosses the boundaries between PDE-constrained and bilevel optimisation, minimal surface control.

\subfile{eit}

\subfile{minimal-surface}

\appendix

\section{Scalar tracking results}
\label{sec:scalar-tracking}

We prove a a simple scalar tracking result, which will be used to establish the results of \cref{sec:tracking}.
The following is a scalar version of \cref{ass:tracking:main}.

\begin{assumption}
    \label{ass:scalar-tracking:main}
    For a given $k \ge 0$ and scalars $\distU{0}, \ldots, \distU{k+1}, \thisDistW[0], \ldots,  \thisDistW[k+1], \distX{1}, \ldots, \distX{k} \ge 0$, and $\scalarTrackingError{0}, \ldots, \scalarTrackingError{k} \in \R$, there exist $\pi_u, \pi_w, \primaldifffact, \alpha_w,\alpha_u>0$, such that
    \begin{align}
        \label{item:scalar-tracking:main:inner-tracking}
        \tag{i}
        \kappa_u \distU{j+1}
        &
        \le
        \distU{j}
        + \pi_u\distX{j},
        &&
        \quad\text{for all}\quad j=1,\ldots,k,
        \\
        \label{item:scalar-tracking:main:adjoint-tracking}
        \tag{ii}
        \kappa_w \distW{j+1}
        &
        \le
        \thisDistW[j]
        + {\primaldifffact} \nextDistU[j]
        + \pi_w\distX{j},
        &&
        \quad\text{for all}\quad j=1,\ldots,k,
        \quad\text{and}
        \\
        \label{item:scalar-tracking:main:differential-transformation}
        \tag{iii}
        \scalarTrackingError{j}
        &
        \le
        \alpha_u \distU{j+1}
        + \alpha_w \distW{j+1}
        &&
        \quad\text{for all}\quad j=0,\ldots,k.
    \end{align}
\end{assumption}

We wish to develop simple bounds for $\scalarTrackingErrorThis$ that can readily be used in convergence proofs when \cref{ass:scalar-tracking:main} is instantiated as \cref{ass:tracking:main}.
These core estimates then allow us to isolate the contributions of initialisation and update errors, and thereby quantify the impact of inexact inner and adjoint solutions over multiple iterations on the differential approximations.
We start with a result that unrolls the recursion in \cref{ass:scalar-tracking:main}.

\begin{lemma}
    \label{lemma:scalar-tracking:generic-recursion}
    Let \cref{ass:scalar-tracking:main}\,\cref{item:scalar-tracking:main:inner-tracking,item:scalar-tracking:main:adjoint-tracking} hold for a $k \ge 0$.
    Then, letting $\iota_k \defeq \sum_{m=1}^{k}  \kappa_u^{-m}\kappa_w^{-(k+1-m)}$ (understanding that $\iota_0=0$), we have
    \begin{equation}
        \label{eq:scalar-tracking:generic-recursion:claim}
        \begin{aligned}[t]
        \nexxt R(\alpha_u, \alpha_w)
        \defeq
        \alpha_u \aaa{k+1} + \alpha_w \bbb{k+1}
        &
        \le
        (\alpha_u\kappa_u^{-k} + \alpha_w \iota_k {\primaldifffact}) \aaa{1}
        +
        \alpha_w\kappa_w^{-k} \bbb{1}
        \\
        \MoveEqLeft[-1]
        + \sum_{j=0}^{k-1} \bigl(\alpha_u\kappa_u^{-(k-j)}\pi_u
        + \alpha_w[\iota_{k-j}{\primaldifffact}\pi_u + \kappa_w^{-(k-j)}\pi_w] \bigr) \ccc{j+1}.
        \end{aligned}
    \end{equation}
\end{lemma}

\begin{proof}
    For $k=0$, the right hand side of the inequality in  \eqref{eq:scalar-tracking:generic-recursion:claim} is equal to the left hand side.
    For $k=1$,
    we have $\iota_1 = \inv\kappa_u\inv\kappa_w$
    and, by assumption,
    $
        \aaa{2}
        \le
        \inv\kappa_u  \aaa{1} + \inv\kappa_u \pi_u \ccc{1}
        \text{ and }
        \bbb{2}
        \le
        \inv\kappa_w \bbb{1} + \inv\kappa_w {\primaldifffact} \aaa{2} + \inv\kappa_w \pi_w \ccc{1}.
    $
    Multiplying the former by $\alpha_u+\alpha_w\inv\kappa_w{\primaldifffact}$ and the latter by $\alpha_w$, then summing up, observing to cancel the two instances of $\alpha_w\inv\kappa_w{\primaldifffact}\aaa{2}$, establishes \eqref{eq:scalar-tracking:generic-recursion:claim}.

    We then take $k=n+1$, and proceed by induction, assuming \eqref{eq:scalar-tracking:generic-recursion:claim} to hold for $k = n$.
    Again,
    $
        \aaa{n+2}
        \le
        \inv\kappa_u  \aaa{n+1} + \inv\kappa_u \pi_u \ccc{n+1}
        \text{ and }
        \bbb{n+2}
        \le
        \inv\kappa_w \bbb{n+1} + \inv\kappa_w {\primaldifffact} \aaa{n+2} + \inv\kappa_w \pi_w \ccc{n+1}
    $
    by assumption.
    As in the case $k=1$, multiplying the former by $\alpha_u+\alpha_w\inv\kappa_w{\primaldifffact}$ and the latter by $\alpha_w$, and then summing up, yields
    \[
        \begin{aligned}[t]
            R^{n+2}(\alpha_u, \alpha_w)
            =
            \alpha_u \aaa{n+2} + \alpha_w \bbb{n+2}
            &
            \le
            (\alpha_u\inv\kappa_u + \alpha_w\inv\kappa_w\inv\kappa_u{\primaldifffact}) \aaa{n+1}
            +
            \alpha_w\inv\kappa_w \bbb{n+1}
            \\
            \MoveEqLeft[-1]
            +
            (\alpha_u\inv\kappa_u\pi_u + \alpha_w[\inv\kappa_w\inv\kappa_u\pi_u  {\primaldifffact} + \inv\kappa_w\pi_w]) \ccc{n+1}.
        \end{aligned}
    \]
    The first two terms on the right-hand side equal $R^{n+1}(\alpha_u\inv\kappa_u + \alpha_w\inv\kappa_w\inv\kappa_u{\primaldifffact}, \alpha_w\inv\kappa_w)$, so using \eqref{eq:scalar-tracking:generic-recursion:claim} for $k=n$, we continue
    \[
        \begin{aligned}[t]
            R^{n+2}(\alpha_u, \alpha_w)
            &
            \le
            ((\alpha_u\inv\kappa_u + \alpha_w\inv\kappa_w\inv\kappa_u{\primaldifffact})\kappa_u^{-n} + \alpha_w\inv\kappa_w\iota_n {\primaldifffact}) \aaa{1}
            +
            \alpha_w\inv\kappa_w\kappa_w^{-n} \bbb{1}
            \\
            \MoveEqLeft[-1]
            +
            \sum_{j=0}^{n-1} \bigl(
                (\alpha_u\inv\kappa_u + \alpha_w\inv\kappa_w\inv\kappa_u{\primaldifffact})\kappa_u^{-(n-j)}\pi_u
                + \alpha_w\inv\kappa_w[\iota_{n-j}{\primaldifffact}\pi_u + \kappa_w^{-(n-j)}\pi_w]
            \bigr) \ccc{j+1}
            \\
            \MoveEqLeft[-1]
            +
            (\alpha_u\inv\kappa_u\pi_u + \alpha_w[\inv\kappa_w\inv\kappa_u\pi_u  {\primaldifffact} + \inv\kappa_w\pi_w]) \ccc{n+1}\\
            &
            =
            (\alpha_u\kappa_u^{-(n+1)} + \alpha_w \primaldifffact(\inv\kappa_w\kappa_u^{-(n+1)} + \inv\kappa_w \iota_n)) \aaa{1}
            +
            \alpha_w\kappa_w^{-(n+1)} \bbb{1}
            \\
            \MoveEqLeft[-1]
            +
            \sum_{j=0}^n \bigl(
                \alpha_u\kappa_u^{-(n+1-j)}\pi_u  + \alpha_w[ (\inv\kappa_w\kappa_u^{-(n+1-j)}
                + \inv\kappa_w\iota_{n-j}){\primaldifffact}\pi_u + \kappa_w^{-(n+1-j)}\pi_w]
            \bigr) \ccc{j+1}.
        \end{aligned}
    \]
    Here $\inv\kappa_w\kappa_u^{-(n+1-j)}+ \inv\kappa_w\iota_{n-j} = \iota_{n+1 - j}$, as by the definition of $\iota_{n+1}$, for any $n \ge 0$,
    \begin{equation}
        \label{eq:scalar-tracking:iota-recursion}
        \iota_{n+1}
        =
        \sum_{m=1}^{n+1}  \kappa_u^{-m}\kappa_w^{-(n+2-m)}
        =
        \inv\kappa_w\kappa_u^{-(n+1)}  + \sum_{m=1}^{n}  \kappa_u^{-m}\kappa_w^{-(n+2-m)}
        =
        \inv\kappa_w\kappa_u^{-(n+1)} + \inv\kappa_w \iota_n,
    \end{equation}
    Thus we obtain \eqref{eq:scalar-tracking:generic-recursion:claim} for $k=n+1$.
\end{proof}

The next three lemmas form our core estimates.
To simplify the estimates, recalling that $\kappa_u, \kappa_w>1$, we observe that
\begin{equation}
    \label{eq:scalar-tracking:iota-est}
    p^k \iota_k
    \le
    \inv p k (\kappa/p)^{-(k+1)}
    \quad\text{for}\quad
    \kappa \defeq \min (\kappa_u, \kappa_w) > 1
    \text{ and any } p \in (0, \kappa).
\end{equation}
Thus, by sum formulae for arithmetic-geometric progressions \cite[formula 0.113]{gradshteyn2014table},
\begin{equation}
    \label{eq:scalar-tracking:iota-est-sum}
    \sum_{k=0}^{n-1} p^k \iota_k
    \le
    \sum_{k=0}^{\infty} p^k \iota_k
    \le
    p^{-1}(\kappa/p-1)^{-2}
    =p(\kappa-p)^{-2}
    \quad \text{for all } n \in \N.
\end{equation}

The next lemma bounds the distance between the true differential and estimate at iteration $k$ through an error term. In the two lemmas that follow, we further estimate the error term.
We use the first lemma directly in \cref{thm:weaker:erroneous-lipschitz}.

\begin{lemma}
    \label{lemma:scalar-tracking:inner-product-error-estimate}
    Suppose \cref{ass:scalar-tracking:main} holds for a $k \ge 0$.
    Then for any $p \in (0, \kappa)$, we have
    \begin{equation}
        \label{eq:scalar-tracking:inner-product-error-estimate}
        \scalarTrackingErrorThisSq
        \le
        (\alpha_u \nextDistU + \alpha_w \nextDistW)^2
        \le
        \breve e_{p,k}.
    \end{equation}
    where, for
    $
        \trackingres_j
        \defeq \alpha_u\kappa_u^{-j}\pi_u + \alpha_w[\iota_j{\primaldifffact}\pi_u + \kappa_w^{-j}\pi_w]
    $
    and
    $\overline\kappa \defeq \max\{\kappa_u, \kappa_w\}$, we set
    {\allowdisplaybreaks\begin{align}
        \label{eq:scalar-tracking:ressum}
        \trackingressum
        &
        \defeq
        \frac{\overline\kappa}{p}
        \sum_{j=0}^\infty p^j \trackingres_j
        \le
        \frac{(\alpha_u\pi_u+\alpha_w\pi_w)\kappa\overline\kappa}{p(\kappa-p)}
        + \frac{\alpha_w {\primaldifffact}\pi_u\overline\kappa}{p^2(\kappa-p)^2}
        \quad\text{and}
        \\
        \label{eq:scalar-tracking:ek:0}
        \breve e_{p,k}
        &
        \defeq
            \frac{\trackingressum(\alpha_u\kappa_u^{-k} + \alpha_w \iota_k {\primaldifffact})}{\pi_u p^k}\initDistUsq
        +
        \frac{\trackingressum\alpha_w\kappa_w^{-k}}{\pi_w p^k}\initDistWsq
        +
        \sum_{j=0}^{k-1} \frac{\trackingressum\trackingres_{k-j}}{ p^{k-j}} \nextDistXthisSq[j].
    \end{align}}
    The second inequality of \eqref{eq:scalar-tracking:inner-product-error-estimate} holds even if \cref{ass:scalar-tracking:main}\,\cref{item:scalar-tracking:main:differential-transformation} does not.
\end{lemma}

\begin{proof}
    Since $\{x^n\}_{n=0}^k \subset \localset$, invoking the inner and adjoint tracking \cref{ass:scalar-tracking:main}\,\cref{item:scalar-tracking:main:inner-tracking,item:scalar-tracking:main:adjoint-tracking} and \cref{lemma:scalar-tracking:generic-recursion},
    we obtain
    \[
        \nexxt R
        \defeq
        \alpha_u \nextDistU + \alpha_w \nextDistW
        \le
        (\alpha_u\kappa_u^{-k} + \alpha_w \iota_k {\primaldifffact})\initDistU
        +
        \alpha_w\kappa_w^{-k}\initDistW
        + \sum_{j=0}^{k-1} \trackingres_{k-j} \nextDistXthis[j].
    \]
    Using  Young's inequality several times here, we deduce for any $\theta_k^u,\theta_k^w,\theta_{k,j},s>0$ that
    \begin{equation}
        \label{eq:scalar-tracking:sr-est}
        \begin{split}
        4s \nexxt R
        &
        \le
            \frac{(\alpha_u\kappa_u^{-k} + \alpha_w \iota_k {\primaldifffact})^2}{\theta_k^u}\initDistUsq
        +
        \frac{(\alpha_w\kappa_w^{-k})^2}{\theta_k^w}\initDistWsq
        \\
        \MoveEqLeft[-1]
        +
        \sum_{j=0}^{k-1} \frac{\trackingres_{k-j}^2}{\theta_{k,j}} \nextDistXthisSq[j]
        +
        4\left(\theta_k^u + \theta_k^w + \sum_{j=0}^{k-1} \theta_{k,j}\right)s^2.
        \end{split}
    \end{equation}
    Take $\theta_k^u = p^k \trackingressum^{-1}\pi_u(\alpha_u\kappa_u^{-k} + \alpha_w \iota_k {\primaldifffact})$,
    $\theta_k^w = p^k \trackingressum^{-1}\pi_w\alpha_w\kappa_w^{-k}$,
    and $\theta_{k,j} = \trackingressum^{-1}p^{k-j}\trackingres_{k-j}$.
    Observe from \eqref{eq:scalar-tracking:iota-recursion} that $\iota_k \le \kappa_w \iota_{k+1}$.
    Hence $p^k \iota_k \le (\kappa_w/p) p^{k+1} \iota_{k+1}$, and further, $p^k\trackingres_k \le (\overline\kappa/p)p^{k+1}\trackingres_{k+1}$, where $\overline\kappa/p > 1$. Now
    \[
        \theta_k^u + \theta_k^w + \sum_{j=0}^{k-1} \theta_{k,j}
        = \frac{1}{\trackingressum}\left(p^k \trackingres_k + \sum_{j=1}^k p^j \trackingres_j\right)
        \le \frac{\overline\kappa}{\trackingressum p}\sum_{j=0}^{k+1} p^j \trackingres_j
        \le 1.
    \]
    Inserting this estimate and the choices of $\theta_k^w$, $\theta_k^u $, and $\theta_{k,j}$ into \eqref{eq:scalar-tracking:sr-est}, establishes
    $
        4s \nexxt R
        \le
        \breve e_{p,k} + 4s^2.
    $
    Taking $s=\nexxt R/2$ (which maximises $4s \nexxt R - 4s^2$) yields the second inequality of \cref{eq:scalar-tracking:inner-product-error-estimate}.
    The first inequality is simply \cref{ass:scalar-tracking:main}\,\cref{item:scalar-tracking:main:differential-transformation}.

    Finally, the bound in \eqref{eq:scalar-tracking:ressum} on $\trackingressum$ follows from \cref{eq:scalar-tracking:iota-est-sum} and $\sum_{j=0}^{\infty} (p/\kappa)^j = 1/(1-p/\kappa) = \kappa/(\kappa-p)$.
\end{proof}

We use the next result in \cref{thm:weaker:smoothness}.

\begin{lemma}
    \label{lemma:scalar-tracking:error-sum}
    Suppose \cref{ass:scalar-tracking:main} holds for a $k \ge 0$.
    Then for any $p \in [1, \kappa)$, we have
    \begin{equation}
        \label{eq:scalar-tracking:inner-product-error-estimate:x}
        \scalarTrackingErrorThisSq
        \le
        \trackingressum^2 \nextDistXthisSq
        +
        e_{p,k},
    \end{equation}
    where, for $\breve e_{p,k}$ defined \eqref{eq:scalar-tracking:ek:0},
    \begin{equation}
        \label{eq:scalar-tracking:ek}
        e_{p,k} \defeq \breve e_{p,k} - \trackingressum^2 \nextDistXthisSq
    \end{equation}
    satisfies
    \begin{equation}
        \label{eq:scalar-tracking:ressum:sm}
        \begin{split}
        \sum_{n=0}^k p^n e_{p,n}
        &
        \le
        \Psi_p \defeq
            \frac{\initDistUsq}{\pi_u} \bigg(\frac{\trackingressum\alpha_u\kappa}{\kappa-1} + \frac{\trackingressum\alpha_w{\primaldifffact}}{(\kappa-1)^2}\bigg)
        +
        \frac{\initDistWsq}{\pi_w} \bigg(\frac{\trackingressum\alpha_w\kappa}{\kappa-1}\bigg).
        \end{split}
    \end{equation}
\end{lemma}

\begin{proof}
    We only need to prove \eqref{eq:scalar-tracking:ressum:sm}; the rest is immediate from \cref{lemma:scalar-tracking:inner-product-error-estimate}. Let
    \[
        \begin{aligned}
        A_n & \defeq
        \frac{\trackingressum(\alpha_u\kappa_u^{-n} + \alpha_w \iota_n {\primaldifffact})}{\pi_u}\initDistUsq,
        &
        B_n & \defeq \frac{\trackingressum\alpha_w\kappa_w^{-n}}{\pi_w}\initDistWsq,
        \\
        C_n & \defeq \sum_{j=0}^{n-1} p^j\trackingressum\trackingres_{n-j} \nextDistXthisSq[j],
        \quad\text{and}
        &
        D_n & \defeq p^n \trackingressum^2 \nextDistXthisSq[n].
        \end{aligned}
    \]
    Then $p^n e_{p,n} =:  A_n + B_n + C_n - D_n$ from \cref{eq:scalar-tracking:ek,eq:scalar-tracking:ek:0}.
    By the assumption $p \in [1, \kappa)$, we have $\frac{\overline\kappa}{p}p^j \ge 1$.
    Hence from \eqref{eq:scalar-tracking:ressum} we have that
    $
        \trackingressum
        =
        \frac{\overline\kappa}{p}
        \sum_{j=0}^\infty p^j \trackingres_j
        \ge
        \sum_{j=0}^\infty \trackingres_j.
    $
    Now
    \[
        \begin{split}
        \sum_{n=0}^{k} C_n
        &
        =
        \sum_{n=0}^{k}  \sum_{j=0}^{n-1} p^j \trackingressum\trackingres_{n-j} \nextDistXthisSq[j]
        =
        \trackingressum
        \sum_{j=0}^{k-1} p^j \sum_{n=j+1}^{k} \trackingres_{n-j} \nextDistXthisSq[j]
        \\
        &
        =
        \trackingressum
        \sum_{j=0}^{k-1} p^j \sum_{\ell=0}^{k-1-j} \trackingres_{\ell+1} \nextDistXthisSq[j]
        \le
        \sum_{j=0}^{k-1} p^j \trackingressum^2 \nextDistXthisSq[j]
        \le
        \sum_{j=0}^{k} D_j.
        \end{split}
    \]
    Moreover, using \eqref{eq:scalar-tracking:iota-est-sum} and the sum formula for geometric series, we estimate that $\sum_{n=0}^{k} (A_n+B_n)$ is less than the right-hand side of \eqref{eq:scalar-tracking:ressum:sm}.
\end{proof}

 We use the next lemma in \cref{thm:weaker:erroneous-lipschitz}.

\begin{lemma}
    \label{lemma:scalar-tracking:error-one}
    Suppose \cref{ass:scalar-tracking:main} holds for a $k \in \N$.
    Then, for $\breve e_{1,n}$ given in \eqref{eq:scalar-tracking:ressum}, we have
    $
        \sum_{n=0}^{k-1} \breve e_{1,n}
        \le \Psi_1 + \trackingressum[1]^2\sum_{n=0}^{k-1}\nextDistXthisSq[n].
    $
\end{lemma}

\begin{proof}
    We have $\breve e_{1,n} = e_{1,n} + \trackingressum[1]^2 \nextDistXthisSq[n]$,
    where \eqref{eq:scalar-tracking:ressum:sm} bounds $\sum_{n=0}^{k-1} e_{1,n} \le \Psi_1$.
\end{proof}

\section{Opial's lemma for quasi-Féjer monotonicity}
\label{sec:opial}

Here we prove a generalisation of Opial's lemma \cite{opial1967weak} for quasi-Féjer monotonicity, i.e, Féjer monotonicity with an additive error term. We prove it in normed spaces for Bregman divergences \eqref{eq:fb:bregman}, as they add no extra difficulties.
In an even more general variable-metric framework, a similar result is also proved in \cite[Proposition 2.7]{vannguyen2016variable}.
Our simplified proof follows the outline of that in \cite{clason2020introduction}, and is nearly identical to the one in \cite{tuomov-pointsource}, where the errors took a more specific form.

For the proof, we recall the following deterministic version of the results of \cite{robbins1971convergence}:

\begin{lemma}
    \label{lemma:convergence-inequality}
    Let $\{a_k\}_{k \in \N}$, $\{b_k\}_{k \in \N}$, $\{c_k\}_{k \in \N}$, and $\{d_k\}_{k \in \N}$ be non-negative and
    $
        a_{k+1} \le a_k(1+b_k) + c_k - d_k
    $
    for all $k \in \N$.
    If $\sum_{k=0}^\infty b_k < \infty$ and $\sum_{k=0}^\infty c_k < \infty$, then
    \begin{enumerate*}[label=(\roman*)]
        \item $\lim_{k \to \infty} a_k$ exists and is finite; and
        \item $\sum_{k=0}^\infty d_k < \infty$.
    \end{enumerate*}
\end{lemma}

\begin{lemma}
    \label{lemma:opial}
    Let either $X$ be the dual space of a corresponding separable normed space $X_*$, or, alternatively, let $X$ be reflexive.
    Also let $M : X \to \R$ be convex, proper, and Gâteaux differentiable with $M': X \to X_*$ weak-$*$-to-weak continuous.
    Finally, let $\hat X \subset X$ be non-empty and $\{e_k(\bar x)\}_{k \in \N} \in \R$ for all $\bar x \in \hat X$. If
    \begin{enumerate}[label=(\roman*)]
        \item\label{item:opial-limit} all weak-$*$ limit points of $\{\thisx\}_{k \in \N}$ belong $\hat X$;
        \item\label{item:opial-nonincreasing} $\bregmann_M(\nextx, \bar x) \le \bregmann_M(\thisx, \bar x) + e_k(\bar x)$ for some $e_k(\bar x) \ge 0$ for all $\bar x \in \hat X$ and $k \in \N$; and
        \item\label{item:opial-epsilon-bound} $\sum_{k=0}^\infty e_k(\bar x) <\infty$ for all $\bar x \in \hat X$;
    \end{enumerate}
    then all weak-$*$ limit points of $\{\this x\}_{k \in \N}$ satisfy $\hat x, \bar x \in \hat X$ and
    \begin{equation}
        \label{eq:opial:mprime-zero}
        \dualprod{M'(\hat x) - M'(\bar x)}{\hat x - \bar x} = 0.
    \end{equation}
    If $\{\thisx\}_{k \in \N} \subset X$ is bounded, then such a limit point exists.
    If, in addition to all the previous assumptions, \eqref{eq:opial:mprime-zero} implies $\hat x = \bar x$ (such as when $M$ is strictly monotone), then $\thisx \weaktostar \hat x$ weakly-$*$ in $X$ for some $\hat x \in \hat X$.
\end{lemma}

\begin{proof}
    Let $\bar x$ and $\hat x$ be weak-$*$ limit points of $\{\thisx\}_{k \in \N}$.
    Since Bregman divergences $\bregmann_M \ge 0$ for convex $M$, the conditions \cref{item:opial-epsilon-bound,item:opial-nonincreasing} establish the assumptions of \cref{lemma:convergence-inequality} for $a_k = \bregmann_M(\thisx; \bar x)$, $b_k=0$, $c_k=e_k(\bar x)$, and $d_k=0$. It follows that $\{\bregmann_M(x^k; \bar x)\}_{k\in\N}$ is convergent.
    Likewise we establish that $\{\bregmann_M(x^k; \hat x)\}_{k\in\N}$ is convergent.
    Therefore, by the obvious three-point identity for Bregman divergences (see, e.g., \cite{tuomov-firstorder}),
    \[
         \dualprod{M'(x^k)-M'(\hat x)}{\bar x - \hat x}
            = \bregmann_M(x^k; \hat x) - \bregmann_M(x^k; \bar x) + \bregmann_M(\hat x; \bar x)
            \to c \in \R.
    \]
    Since $\bar x$ and $\hat x$ are a weak-$*$ limit point, there exist subsequences $\{x^{k_n}\}_{n\in \N}$ and $\{x^{k_m}\}_{m\in \N}$ with $x^{k_n}\weakto \bar x$ and $x^{k_m}\weakto \hat x$.
    By the weak-$*$-to-weak continuity of $M': X \to X_*$, \eqref{eq:opial:mprime-zero} follows from
    \[
        \dualprod{M'(\bar x)-M'(\hat x)}{\bar x - \hat x}
        =
        \lim_{n \to \infty}
        \dualprod{M'(x^{k_n})-M'(\hat x)}{\bar x - \hat x}
        =
        c
        =
        \lim_{m \to \infty}
        \dualprod{M'(x^{k_m})-M'(\hat x)}{\bar x - \hat x}
        =
        0.
    \]

    If $\{\thisx\}_{k \in \N}$ is bounded, and $X$ is the dual space of some separable normed space $X_*$, it contains a \mbox{weakly-$*$} convergent subsequence by the Banach--Alaoglu theorem, so a limit point exists as claimed. If $X$ is reflexive, the Eberlein–\u{S}mulyan theorem establishes the same result.
    Hence, if \eqref{eq:opial:mprime-zero} implies $\bar x =\hat x$, then every convergent subsequence of  $\{\thisx\}_{k \in \N}$ has the same weak limit.
    It lies in $\hat X$ by \cref{item:opial-limit}.
    The final claim now follows from a standard subsequence--subsequence argument: Assume to the contrary that there exists a subsequence of $\{x^k\}_{k\in\N}$ not convergent to $\hat x$. Then the above argument provides a further subsequence converging to $\hat x$.
    This contradicts the fact that any subsequence of a convergent sequence converges to the same limit.
\end{proof}

\input{tracking.bbl}
\bibliographystyle{jnsao}

\end{document}

%% file: symbols.tex
\newcommand{\term}{\emph}

\newcommand{\field}[1]{\mathbb{#1}}
\newcommand{\N}{\mathbb{N}}

\newcommand{\R}{\field{R}}
\newcommand{\extR}{\overline \R}

\newcommand{\norm}[1]{\|#1\|}

\newcommand{\adaptabs}[1]{\left|#1\right|}

\newcommand{\inv}[1]{#1^{-1}}

\newcommand{\grad}{\nabla}
\newcommand{\freevar}{\,\boldsymbol\cdot\,}
\newcommand{\Union}\bigcup
\newcommand{\Isect}\bigcap
\newcommand{\union}\cup
\newcommand{\isect}\cap
\newcommand{\bigunion}\bigcup
\newcommand{\bigisect}\bigcap

\newcommand{\defeq}{:=}

\newcommand{\subdiff}{\partial}

\DeclareRobustCommand{\downto}{{{\mathchoice%
            {\rotatebox[origin=c]{-20}{$\to$}}
            {\rotatebox[origin=c]{-20}{$\to$}}
            {\rotatebox[origin=c]{-20}{\scalebox{0.75}{$\to$}}}
            {\rotatebox[origin=c]{-20}{\scalebox{0.6}{$\to$}}}
}}}

\DeclareMathOperator*{\argmin}{arg\,min}

\DeclareMathOperator{\interior}{int}

\DeclareMathOperator{\Dom}{dom}

\DeclareMathOperator{\sublev}{sub}

\DeclareMathOperator{\diam}{diam}

\DeclareMathOperator{\tr}{tr}
\DeclareMathOperator{\diag}{diag}

\DeclareMathOperator{\divergence}{div}

\DeclareMathOperator{\prox}{prox}

\DeclareMathOperator{\TV}{TV}
\DeclareMathOperator{\range}{ran}

\newcommand{\iprod}[2]{\langle #1,#2\rangle}
\newcommand{\adaptiprod}[2]{\left\langle #1,#2\right\rangle}

\newcommand{\weakto}{\mathrel{\rightharpoonup}}
\def \weaktostarSym{\setbox0=\hbox{$\rightharpoonup$}\rlap{\hbox
        to\wd0{\hss\raise1ex\hbox{$\scriptscriptstyle{*\,}$}\hss}}\box0}
    \def \weaktostar    {\mathrel{\weaktostarSym}}

\def\linear{\mathbb{L}}

\def\extR{\overline \R}

\def\ccc{\,d}

\def\dualprod#1#2{\langle #1|#2\rangle}

\def\infconv{\mathop{\Box}}

\let\phi=\varphi
\let\epsilon=\varepsilon

\DeclareMathOperator{\Id}{Id}

\def\BVspace{\mathop{\mathrm{BV}}}

\def\nexxt#1{#1^{k+1}}
\def\this#1{#1^k}
\def\prev#1{#1^{k-1}}
\let\opt\hat

\def\estgrad#1{\widetilde{\grad #1}}
\def\estdiff#1{\widetilde{#1'}}

\def\nextx{\nexxt x}
\def\thisx{\this x}
\def\prevx{\prev x}
\def\optx{\opt x}

\def\nexty{\nexxt y}
\def\thisy{\this y}

\def\nextu{\nexxt u}
\def\thisu{\this u}
\def\optu{\opt u}
\let\base\bar
\def\basex{\base x}
\def\basey{\base y}
\def\basez{\base z}
\def\nextz{\nexxt z}
\def\thisz{\this z}

\def\totalest#1{\tilde\partial_{#1}}

\def\nextestdiff{\estdiff F(\thisx)}
\def\nextestgrad{\estgrad F(\thisx)}

\def\diffwrt#1#2{#1^{(#2)}}

\def\d{\,\mathrm{d}}
\def\gap{\mathcal{G}}

\def\errDesc#1{\varepsilon_{\mathrm{desc},#1}}
\def\errLip#1{e_{\mathrm{lip},#1}}
\def\errMono#1{\varepsilon_{#1}}
\def\errDescBar#1{\varepsilon_{#1}}
\def\errDescBarEmph#1{\varepsilon_{\mathrm{desc},#1}}
\def\err#1{\varepsilon_{#1}}
\def\rDesc{r_{\mathrm{desc}}}

\def \approxinSym{\setbox0=\hbox{$\in$}\rlap{\hbox
        to\wd0{\hss\raise1ex\hbox{$\sim$}\hss}}\box0}
\def\approxin{\mathrel{\approxinSym}}

\def\openball{\mathbb{O}}

\def\bregmann{B}
\def\localset{\Omega_X}
\def\zlocalset{\Omega_Z}
\def\ulocalset{\Omega_U}
\def\basexset{\Omega_{\basex}}
\def\lipSu{\pi_u}
\def\lipSw{\pi_w}

\usepackage{ifthen}

\newcommand{\dis}[4][]{%
    d_{#2}%
    \ifthenelse{\equal{#1}{2}}{^2}{}%
    (#3, #4)%
}

\newcommand{\adaptdis}[4][]{%
    d_{#2}%
    \ifthenelse{\equal{#1}{2}}{^2}{}%
    \left(#3, #4\right)%
}

\def\primaldifffact{\mu_u}

\def\trackingres{\psi}
\newcommand{\trackingressum}[1][p]{\varsigma_{#1}}

\def\nocolor#1{}

\def\distU#1{{\nocolor{SteelBlue!70}d^u_{#1}}}
\def\distUsq#1{{\nocolor{SteelBlue!70}(d^u_{#1})^2}}
\newcommand{\nextDistU}[1][k]{\distU{#1+1}}
\newcommand{\thisDistU}[1][k]{\distU{#1}}
\def\initDistU{\distU{1}}
\def\initDistUsq{\distUsq{1}}
\def\dU{{\nocolor{SteelBlue!70}d_U}}
\def\aaa#1{\distU{#1}}
\def\InstnextDistU{{\nocolor{SteelBlue!70}d_U(\nextu, S_u(\thisx))}}
\def\InstthisDistU{{\nocolor{SteelBlue!70}d_U(\thisu, S_u(\prev x))}}
\def\InstinitDistU{{\nocolor{SteelBlue!70}d_U(u^1, S_u(x^0))}}
\def\InstinitDistUsq{{\nocolor{SteelBlue!70}d_U^2(u^1, S_u(x^0))}}

\def\distW#1{{\nocolor{ForestGreen!70}d^w_{#1}}}
\def\distWsq#1{{\nocolor{ForestGreen!70}(d^w_{#1})^2}}
\newcommand{\nextDistW}[1][k]{\distW{#1+1}}
\newcommand{\thisDistW}[1][k]{\distW{#1}}
\def\initDistW{\distW{1}}
\def\initDistWsq{\distWsq{1}}
\def\dW{{\nocolor{ForestGreen!70}d_W}}
\def\bbb#1{\distW{#1}}
\def\InstnextDistW{{\nocolor{ForestGreen!70}d_W(\nexxt w, S_w(\thisx))}}
\def\InstthisDistW{{\nocolor{ForestGreen!70}d_W(\this w, S_w(\prev x))}}
\def\InstinitDistW{{\nocolor{ForestGreen!70}d_W(w^1, S_w(x^0))}}
\def\InstinitDistWsq{{\nocolor{ForestGreen!70}d_W^2(w^1, S_w(x^0))}}

\newcommand{\distX}[1]{{\nocolor{orange!70}\varrho_{#1}}}
\newcommand{\thisDistXprev}[1][k]{\distX{#1}}
\newcommand{\nextDistXthis}[1][k]{\distX{#1+1}}
\newcommand{\nextDistXthisSq}[1][k]{{\nocolor{orange!70}\varrho_{#1+1}^2}}
\newcommand{\InstthisDistXprev}[1][k]{{\nocolor{orange!70}\norm{x^{#1} - x^{#1-1}}_\circ}}
\newcommand{\InstnextDistXthis}[1][k]{{\nocolor{orange!70}\norm{x^{#1+1} - x^{#1}}_\circ}}
\newcommand{\InstnextDistXthisSq}[1][k]{{\nocolor{orange!70}\norm{x^{#1+1} - x^{#1}}_\circ^2}}

\def\bX{{\nocolor{orange!70}\norm{\freevar}_\circ}}
\def\ccc#1{{\nocolor{orange!70}\varrho_{#1}}}

\def\scalarTrackingError#1{{\nocolor{magenta!70}\widetilde d_{#1}}}
\def\scalarTrackingErrorSq#1{{\nocolor{magenta!70}\widetilde d_{#1}^2}}
\def\scalarTrackingErrorThis{\scalarTrackingError{k}}
\def\scalarTrackingErrorThisSq{\scalarTrackingErrorSq{k}}
\def\trackingErrorThis{{\nocolor{magenta!70}\norm{\nextestdiff-F'(\thisx)}_*}}
\def\trackingErrorThisSq{{\nocolor{magenta!70}\norm{\nextestdiff-F'(\thisx)}_*^2}}
\def\trackingErrorSq#1#2{{\nocolor{magenta!70}\norm{#1-#2}_*^2}}
\def\trackingError#1#2{{\nocolor{magenta!70}\norm{#1-#2}_*}}
\def\dXstar{{\nocolor{magenta!70}\norm{\freevar}_*}}

\def\normalDistXsq#1#2{{\nocolor{red!70}\norm{#1-#2}_\circ^2}}
\def\normalDistX#1#2{{\nocolor{red!70}\norm{#1-#2}_\circ}}
\def\demoDistXsq#1#2{{\nocolor{red!70}\norm{#1-#2}_X^2}}

\def\dX{{\nocolor{red!70}\norm{\freevar}_\circ}}

\def\MF{\bar M}
\def\GF{\bar\Gamma_F}
\def\GG{\bar\Gamma_G}
\def\LL{\bar\Lambda}

\newcommand{\fakenormsq}[2]{\llbracket #1 \rrbracket_{#2}^2}
\newcommand{\normsq}[2]{\| #1 \|_{#2}^2}
\def\Inj{\mathscr{I}}

\def\hstar{\star}

%% file: eit.tex
\pgfplotsset{
    every axis label/.append style = {font = \scriptsize},
    every tick label/.append style = {font = \scriptsize},
    ignore legend/.style={
        every axis legend/.code={\let\addlegendentry\ignorelegendentry}
    },
    log x ticks with fixed point/.style={
        xticklabel={
            \pgfkeys{/pgf/fpu=true}
            \pgfmathparse{exp(\tick)}%
            \pgfmathprintnumber[fixed relative, precision=3]{\pgfmathresult}
            \pgfkeys{/pgf/fpu=false}
        },
    },
    log y ticks with fixed point/.style={
        yticklabel={
            \pgfkeys{/pgf/fpu=true}
            \pgfmathparse{exp(\tick)}%
            \pgfmathprintnumber[fixed relative, precision=3]{\pgfmathresult}
            \pgfkeys{/pgf/fpu=false}
        },
    },
    compat=1.18
}

\def\conductivity{z}
\def\conductivityLow{\underline\conductivity}
\def\conductivityHigh{\overline\conductivity}
\def\WOp{\Sigma^{-1/2}}
\def\EITmeas{\mathscr{I}}
\def\Hs{\mathcal{H}}
\def\dif{\,\text{d}}
\def\Lmeas{\xi}
\def\BoundMeas{s}
\def\measindex{m}
\def\Nmeas{N}
\def\Nelec{E}

\ifSubfilesClassLoaded{
    \section*{Electrical Impedance Tomography}
}{
    \subsection{Electrical Impedance Tomography}
    \label{sec:eit}
}

\paragraph{Problem formulation}

We start with Electrical Impedance Tomography (EIT).
We seek to construct an electrical conductivity $\conductivity \in \BVspace(\Omega)$ inside a bounded Lipschitz domain $\Omega \subset \R^2$ from boundary measurements of currents at a number $\Nelec \ge 1$ of electrodes.
The same electrodes are excited with prescribed potentials.
This measurement scheme is repeated for multiple combinations of potentials at the different electrodes.
With
$
    \Hs \defeq H^1(\Omega) \times \R^{\Nelec},
$
write
\[
    \bar u
    \defeq
    (\bar u_1, \ldots, \bar u_\Nmeas)
    \defeq
    ((u_1, I_1),\ldots,(u_{\Nmeas}, I_\Nmeas)) \in U \defeq \Hs^{\Nmeas}
\]
for a vector of inner potentials $\bar u_\measindex \in H^1(\Omega)$ and electrode currents $I_\measindex \in \R^{\Nelec}$ over multiple measurements $\measindex=1,\ldots,\Nmeas$.
Imposing total variation regularisation on the conductivity, our problem then is
\begin{equation}
    \label{eq:eit:problem}
    \min_{\conductivity \in \BVspace(\Omega)}~ \frac12\sum_{\measindex=1}^{\Nmeas} \norm{I_\measindex - \EITmeas_\measindex}_{\inv\Sigma}^2 + \delta_{[\conductivityLow,\conductivityHigh]}(\conductivity) + \alpha\TV(\conductivity)
\end{equation}
with $\bar u$ and $\conductivity$ subject to the Complete Electrode Model (CEM) \cite{cheng1989electrode}.
In weak form, this is (see, e.g., \cite[§3]{tuomov2024online-eit})
\begin{subequations}
\label{eq:eit:cem}
\begin{gather}
    \label{eq:eit:bilin}
    B_\conductivity(\bar u_\measindex,\bar v_\measindex) = L_\measindex(\bar v_\measindex)
    \quad\text{for all}\quad
    \bar v_\measindex = (v_\measindex, V_\measindex) \in \Hs
\shortintertext{for the bilinear form $B_\conductivity: \Hs \times \Hs \to \R$,}
    \label{eq:eit:bilin:form}
    B_\conductivity(\bar u_\measindex, \bar v_\measindex)
    \defeq
    \int_{\Omega} \conductivity \grad u_\measindex \cdot \grad v_\measindex \dif\Lmeas + \sum_{i=1}^{\Nelec} \frac{1}{\zeta_i}\int_{\partial \Omega_{i}} u_\measindex (v_\measindex - V_{\measindex,i})\dif\BoundMeas
    + \sum_{i=1}^{\Nelec} I_{\measindex,i} V_{\measindex,i},
\shortintertext{and the linear form $L_\measindex \in \Hs^*$,}
    L_\measindex(\bar v_\measindex)
    \defeq
    \sum_{i=1}^{\Nelec} \frac{1}{\zeta_i} \int_{\partial \Omega_{i}} U_{\measindex,i}(v_\measindex-V_{\measindex,i}) \dif\BoundMeas.
\end{gather}
\end{subequations}
Here $\partial\Omega_{i}$ are the electrode surfaces, $\zeta_i>0$ corresponding contact impedances, and $\nu$ is the outward unit normal.
The electrode potentials $U_{\measindex,i}$ for each measurement  $\measindex=1,\ldots,\Nmeas$ and electrode $i=1,\ldots,\Nelec$ are prescribed, while the resulting electrode currents $I_{\measindex,i}$ are determined by the model.

The upper and lower bounds $0<\conductivityLow<\conductivityHigh$ make the PDE \eqref{eq:eit:cem} well-posed and enforce $\conductivity \in L^\infty(\Omega)$.
The symmetric positive definite matrix $\Sigma \in \R^{\Nelec \times \Nelec}$ models noise levels and data imprecision in the measurements $\EITmeas_\measindex$.
Together with the regularisation parameter $\alpha>0$, the distributional derivative $D\conductivity$ of the conductivity is penalised by the isotropic total variation regulariser
\begin{equation}
    \label{eq:eit:tv}
    \begin{split}
    \TV(\conductivity)
    &
    \defeq
    \sup\left\{
        \int \iprod{\divergence y(\xi)}{\conductivity(\xi)} \d \xi
        \,\middle|\,
        y \in C_c^\infty(\Omega; \R^2),\, \sup_{x \in \Omega} \norm{y(x)}_2 \le 1
    \right\}
    \\
    &
    =
    \sup\left\{
        \int \iprod{\divergence y(\xi)}{\conductivity(\xi)} \d \xi
        \,\middle|\,
        y \in H_0(\divergence, \Omega),\, \sup_{x \in \Omega} \norm{y(x)}_2 \le 1
    \right\}.
    \end{split}
\end{equation}
The latter form follows from density arguments.
By definition, $\TV(z) < \infty$  for $z \in \BVspace(\Omega)$.
The idea of total variation regularisation is to promote sparsity of gradients of the conductivity field: they should be concentrated on object boundaries.
Similar problem formulations for EIT have been studied in \cite{voss2018imaging,voss2019three,jauhiainen2021nonplanar,tuomov2024online-eit}.

Let $Z=\BVspace(\Omega) \isect L^\infty(\Omega)$ and $Y=H_0(\divergence, \Omega)$.
We can write \eqref{eq:eit:cem} over all $m \in \{1,\ldots,\Nmeas\}$ as $T(\bar u, \conductivity)=0$, i.e., \eqref{eq:intro:tsu}, for
$
    T: \Hs^\Nmeas \times Z \to (\Hs^\Nmeas)^*
$
defined by
\[
    T(\bar u, \conductivity) \defeq (B_\conductivity(\bar u_1, \freevar) - L_1, \ldots, B_\conductivity(\bar u_\Nmeas, \freevar) - L_\Nmeas).
\]
We denote the corresponding solution mapping by $S_{\bar u}: Z \to \Hs^\Nmeas$.
Also setting
\begin{gather*}
    f(\conductivity) \defeq j(S_{\bar u}(\conductivity)),
    \quad
    j(\bar u) \defeq \frac12\sum_{\measindex=1}^{\Nmeas} \norm{I_\measindex - \EITmeas_{\measindex}}_{\inv\Sigma}^2,
    \quad
    g(\conductivity) \defeq \delta_{[\conductivityLow,\conductivityHigh]}(\conductivity),
    \\
    K = -\divergence^* \in \linear(Z; Y^*),
    \quad\text{and}\quad
    h=(h_*)^*
    \quad\text{for}\quad
    h_*(y) = \begin{cases}
        0, & \sup_{x \in \Omega} \norm{y(x)}_2 \le \alpha, \\
        \infty, & \text{otherwise,}
    \end{cases}
\end{gather*}
the problem \eqref{eq:eit:problem}\&\eqref{eq:eit:cem} then reads
\begin{equation}
    \label{eq:eit:problem:fg}
    \min_{\conductivity \in Z}~f(\conductivity) + g(\conductivity) + h(Kz)
\end{equation}

\paragraph{Outer algorithm}

To avoid smoothing of the total variation term, it is most convenient to use the primal-dual method of \cref{ex:fb:pdps,sec:outer-examples:pdps} that, based on \eqref{eq:eit:tv}, reformulates the total variation as a dual ball constraint and a bilinear term.
Using a proximal reformulation of primal-dual optimality conditions, also (damped) semismooth Newton (SSN) should be practically---if not theoretically---applicable.
However, due to the complexity of the resulting second-order system, its efficient implementation is outside the scope of the present work.
We will compare our methods against the SSN on the minimal surface problem.

We, thus, apply the primal-dual method \eqref{eq:fb:pdps:implicit} to \eqref{eq:eit:problem:fg}, endowing $Z=\BVspace(\Omega) \isect L^\infty(\Omega)$ with the norm $\norm{u}_Z \defeq \norm{u}_{\BVspace(\Omega)} + \norm{u}_{L^\infty(\Omega)}$.
Then the trivial injection $M_z: Z \hookrightarrow Z^*$, $z \mapsto \iprod{z}{\freevar}_{L^2(\Omega)}$ is continuous.\footnote{This follows from the $L^\infty$ topology. Without it, Poincaré's inequality should be used.}
Likewise, we can trivially embed $H_0(\divergence, \Omega)$ into $H_0(\divergence, \Omega)^*$ to obtain $M_y \in \linear(Y; Y^*)$, $y \mapsto \iprod{y}{\freevar}_{L^2(\Omega; \R^2)}$.
However, this option, which has $\norm{y}_{M_y}^2=\int_\Omega \norm{y(\xi)}^2 \d\xi$, and which we will use in numerical practise, will not---in infinite dimensions---satisfy \cref{ass:pdps:step-length-cond}, required by the convergence \cref{cor:pdps:values}.

An alternative that satisfies \cref{ass:pdps:step-length-cond} for appropriate step lengths, $K_z: Z \hookrightarrow L^2(\Omega)^*$ the trivial injection, and $K_y=-\divergence \in \linear(H_0(\divergence, \Omega); L^2(\Omega))$, is $M_y=\iprod{\divergence \freevar}{\divergence \freevar}_{L^2(\Omega)}$.
An alternative that satisfies \cref{ass:pdps:step-length-cond} for appropriate step lengths, $K_z: Z \hookrightarrow L^2(\Omega)^*$ the trivial embedding, and $K_y=-\divergence \in \linear(H_0(\divergence, \Omega); L^2(\Omega))$, is $M_y=\iprod{\divergence \freevar}{\divergence \freevar}_{L^2(\Omega)}$. 
This choice will not, however, produce an easily computable dual proximal step from the implicit algorithm \eqref{eq:fb:pdps:implicit}: it involves solving a problem of the form $\min_{y \in H_0(\divergence, \Omega)} \norm{\divergence(y-\tilde y)}_{L^2(\Omega)}^2$ subject to $\sup_{\xi \in \Omega} \norm{y(\xi)}_2 \le \alpha$.
We, therefore, use the simpler $M_y$ in finite element spaces. Then the proximal map is a simple projection, and the overall algorithm \eqref{eq:fb:pdps:implicit} reduces to the Hilbert space form \cref{eq:pdps:alg}.

\begin{remark}
    We may not have to solve the expensive proximal map exactly.
    Following our general theme, we could possibly only take a single step of forward-backward splitting in $L^2(\Omega)$ towards the solution of the proximal map, and then use the general inexact convergence theory of \cref{sec:fb}.
    However, we leave the specifics of proximal map approximation outside the scope of this work.
\end{remark}

\paragraph{Differential estimation}

To estimate $f'(\this\conductivity)$ in \eqref{eq:fb:pdps:implicit}, with a small modification, we follow the general forward PDE splitting approach of \cref{thm:tracking:inner-linear-system-splitting}, and the adjoint splitting approach of \cref{thm:tracking:reduced-adjoint-splitting}, both combined with either the Gauss–Seidel splitting of \cref{ex:splitting:Gauss–Seidel}, or with exact solution.
Other alternatives are also imaginable.

Let $N_k+M_k$ be an admissible splitting (\cref{ass:primal-admissible-splitting}) of $B_{\this\conductivity}$.
For the forward PDE, we then get from \cref{thm:tracking:inner-linear-system-splitting} the \textbf{inner step}
\begin{equation}
    \label{eq:eit:inner-step}
    \nexxt{\bar u_\measindex} = \inv N_k (L_\measindex - M_k \this{\bar u_\measindex})
    \quad\text{for all}\quad
    \measindex=1,\ldots,\Nmeas.
\end{equation}

Regarding the adjoint, for $\bar u = S_{\bar u}(z)$, and any $\bar w = (\bar w_1,\ldots,\bar w_\Nmeas) \in (\Hs^\Nmeas)^{**}=\Hs^\Nmeas$,  $\bar h = (\bar h_1, \ldots, \bar h_\Nmeas) \in \Hs^\Nmeas$, and $h_z \in Z$, we have
\begin{gather}
    \label{eq:eit:wTz}
    \bar w \diffwrt{T}{\conductivity}(\bar u, \conductivity)
    =
    \sum_{\measindex=1}^\Nmeas
    \diffwrt{B_\conductivity}{\conductivity}(\bar u_\measindex, \bar w_\measindex),
    \quad\text{i.e.,}\quad
    \bar w \diffwrt{T}{\conductivity}(\bar u, \conductivity)h_z
    =
    \int_{\Omega} h_z \grad u_\measindex \cdot \grad w_\measindex \dif\Lmeas,
\shortintertext{and}
    \nonumber
    \bar w \diffwrt{T}{\bar u}(\bar u, \conductivity)\bar h
    =
    (B_z(\bar h_1, \bar w_1)), \ldots, B_z(\bar h_\Nmeas, \bar w_\Nmeas)).
\end{gather}
Hence, the \textbf{reduced adjoint equation} \eqref{eq:tracking:reduced-adjoint} that defines $S_{\bar w}(\conductivity)=\bar w$, reads
\begin{equation}
    \label{eq:eit:reduced-adjoint}
    B_\conductivity(\freevar, \bar w_\measindex) + \diffwrt{j}{\bar u_\measindex}(\bar u) = 0
    \quad\text{for all}\quad
    \measindex=1,\ldots,\Nmeas.
\end{equation}
From \eqref{eq:tracking:reduced-adjoint-diff-transformation}, we then obtain
\begin{equation}
    \label{eq:eit:f-prime}
    f'(\conductivity)
    =
    \bar w \diffwrt{T}{\conductivity}(\bar u, \conductivity)
    =
    \sum_{\measindex=1}^\Nmeas \diffwrt{B_\conductivity}{\conductivity}(\bar u_\measindex, \bar w_\measindex),
    \quad\text{i.e.,}\quad
    f'(z)h_z = \sum_{\measindex=1}^\Nmeas \int_{\Omega} h_z \grad u_\measindex \cdot \grad w_\measindex \dif\Lmeas.
\end{equation}

Observe that $\diffwrt{j}{\bar u_\measindex}(\bar u) = (0, \diffwrt{j}{I_\measindex}(\bar u))=(0, \inv\Sigma(I_\measindex-\EITmeas_\measindex)) \in H^1(\Omega) \times \R^\Nelec$.
Hence, instead of solving \eqref{eq:eit:reduced-adjoint} for all $i=1,\ldots,\Nmeas$, we take a basis $\{e_1,\ldots,e_\Nelec\}$ of $\R^\Nelec$, and consider for the unknown $\tilde w_i \in \Hs$ the \textbf{modified reduced adjoint equation}
\[
    B_\conductivity(\freevar, \tilde w_i) + (0, e_i) = 0,
    \quad
    \text{for all}
    \quad
    i=1,\ldots,\Nelec.
\]
Representing $\diffwrt{j}{\bar u_\measindex}(\bar u) = (0, \sum_{i=1}^\Nelec e_i \diffwrt{j}{I_\measindex}(\bar u)e_i)$, we can then solve $\bar w_\measindex = \sum_{i=1}^\Nelec \tilde w_i \diffwrt{j}{I_\measindex}(\bar u)e_i$.
This helps to stabilise the Gauss–Seidel approach without any additional computational cost when the number of electrodes $\Nelec$ is not greater than the number of measurements $\Nmeas$.

With this modification, we follow the splitting approach of \cref{thm:tracking:reduced-adjoint-splitting}.
For $N_k+M_k$ an adjoint admissible splitting (\cref{ass:adjoint-admissible-splitting}) of $B_{\this\conductivity}$, it gives the \textbf{adjoint step}
\begin{equation}
    \label{eq:eit:adjoint-step}
    \nexxt{\tilde w}_i \defeq -\inv N_k((0, e_i)+M_k\this{\tilde w_i})
    \quad
    \text{for all}
    \quad
    i=1,\ldots,\Nelec,
\end{equation}
and the \textbf{differential transformation}
\begin{equation}
    \label{eq:eit:diff-transform}
    \begin{split}
    \estdiff f(\this\conductivity)
    =
    \sum_{\measindex=1}^{\Nmeas} \diffwrt{B_{\this\conductivity}}{\conductivity}(\nexxt{\bar u}_\measindex, \nexxt{\bar w}_\measindex)
    &
    =
    \sum_{\measindex=1}^{\Nmeas} \sum_{i=1}^\Nelec
    \iprod{\grad_{I_\measindex} j(\nexxt{\bar u})}{e_i}
    \diffwrt{B_{\this\conductivity}}{\conductivity}(\nexxt{\bar u}_\measindex, \nexxt{\tilde w}_i)
    \\
    &
    =
    \sum_{\measindex=1}^{\Nmeas} \sum_{i=1}^\Nelec
    \iprod{\nexxt{I_\measindex} - \EITmeas_\measindex}{e_i}_{\inv\Sigma}
    \diffwrt{B_{\this\conductivity}}{\conductivity}(\nexxt{\bar u}_\measindex, \nexxt{\tilde w}_i).
    \end{split}
\end{equation}

Our \textbf{overall algorithm}, thus, consists of iterating for $k \in \N$ the steps
\begin{enumerate}[nosep]
    \item Compute $\estdiff f(\this\conductivity)$ through \cref{eq:eit:inner-step,eq:eit:adjoint-step,eq:eit:diff-transform}.
    \item Form $\nexxt\conductivity$ and the dual variable $\nexty$ by solving \eqref{eq:fb:pdps:implicit} (in finite element subspaces, \eqref{eq:pdps:alg}).
\end{enumerate}

\paragraph{Convergence}

We now explore, what would be required to prove the convergence of the above method using \cref{cor:pdps:values}.
We have already constructed our problem in agreement with \cref{eq:pdps:problem}, and have discussed the satisfaction of \cref{ass:pdps:step-length-cond} for the outer algorithm.
Beyond step length, growth, and local initialisation conditions---which are usually difficult to verify for nonconvex problems---we, therefore, need to verify \cref{ass:tracking:main} for the  inner, adjoint, and differential transformation steps \cref{eq:eit:inner-step,eq:eit:adjoint-step,eq:eit:diff-transform}, and we need to verify that $f'$ is Lipschitz.
The verification of \cref{ass:tracking:main} is done using \cref{thm:tracking:inner-linear-system-splitting,thm:tracking:reduced-adjoint-splitting}.
For the Gauss–Seidel option, the admissible splitting \cref{ass:primal-admissible-splitting,ass:adjoint-admissible-splitting}, required by these results, can, in principle, be verified with the help of \cref{ex:splitting:Gauss–Seidel}, after passing to a finite element subspace.

\textbf{Lipschitz properties (challenge, requires \emph{some} finite-dimensionality):}
The Lipschitz continuity of the solution mapping $S_{\bar u}$ and its derivative $S_{\bar u}'$ as functionals on $L^\infty(\Omega)$ are verified in \cite[§3]{tuomov2024online-eit}.
This is why we endow $Z$ with the joint $\BVspace$ and $L^\infty$ norm.
However, \cref{cor:pdps:values} requires the continuity to be with respect to $\norm{\freevar}_{M_z}$, i.e., with respect to the $L^2$ norm.
We know two ways to achieve this:
\begin{enumerate}[label=(\alph*),nosep]
    \item Restrict $\conductivity$ to a finite-element subspace, and use the equivalence of norms.
    \item Restrict $u$ and $w$ to a finite-element subspace, so that $\conductivity$ does not require the $L^\infty$ topology.
\end{enumerate}
The source of these difficulties is the tri-linear term $ \int_{\Omega} \conductivity \grad u_\measindex \cdot \grad v_\measindex \dif\Lmeas$ in \eqref{eq:eit:bilin:form}, where we can use Hölder's inequality in $L^2$ only once.
If we could prove additional regularity of the solutions to the EIT problem, another approach could be possible.

\Cref{thm:tracking:reduced-adjoint-splitting} also demands that the differentials of $T$ be Lipschitz (and bounded).
While the Lipschitz continuity of $\bar u \mapsto \diffwrt{T}{\bar u}(\bar u, z)$ is not difficult to verify due to the bounds $\underline z \le z \le \overline z$, that of $\bar u \mapsto \diffwrt{T}{z}(\bar u, z)$ faces the same difficulties as the Lipschitz continuity of $S_u$.
It has to be Lipschitz between the distances $\norm{\freevar}_{\Hs^\Nmeas}$ and $\norm{\freevar}_{M_z,L\Hs^\Nmeas}$, as defined in \cref{sec:semicon}.
Recalling \eqref{eq:eit:wTz}, we, thus, need for all $\bar v, \bar u \in \Hs^\Nmeas$ that
\[
    \sup_{\norm{h_z}_{L^2(\Omega)} \le 1} \sup_{\norm{\bar w}_{\Hs^\Nmeas} \le 1}
    \sum_{\measindex=1}^\Nmeas  \int_{\Omega} h_z \grad (u_\measindex -v_\measindex) \cdot \grad w_\measindex \dif\Lmeas \le L_{\diffwrt{T}{z};u} \norm{\bar u - \bar v}_{\Hs^\Nmeas}.
\]
Here, $h_\conductivity$ has lost the $L^\infty$ topology of $\conductivity$.
Practically, again, this means that either $h_\conductivity$ (hence, $\conductivity$) or both $u_\measindex$ and $v_\measindex$ must live in a finite-dimensional subspace.

Similar considerations apply to $C_{\diffwrt{T}{\conductivity}} \defeq \sup\{ \norm{\diffwrt{T}{\conductivity}(\nexxt{\bar u}, \this\conductivity)}_{M_z,\Hs^\Nmeas} \mid k \in \N \} < \infty$.
Moreover, we require here a bound on $\norm{\nexxt{\bar u}}_{\Hs^\Nmeas}$.
The $L^\infty$-Lipschitz continuity of $S_u$ and the bounds $\underline z \le z \le \overline z$ bound $\norm{S_u(\thisx)}_{\Hs^\Nmeas}$. If we solve $\nexxt{\bar u}$ exactly, we have our bound.
With the Gauss–Seidel approach, we \emph{may} have to take multiple inner iterations to ensure any given bound.

\begin{remark}
    \label{rem:eit:bounds}
    In fact, \cref{thm:weaker:erroneous-lipschitz} gives an a posteriori bound on $\norm{\nextu-S_u(\thisx)}$, and we have bounded $\norm{S_u(\thisx)}$.
    We, therefore, know that there exists \emph{some} bound on $\sup_{k \in \N} \norm{\nextu}$, we just do not know its exact magnitude.
    Indeed, this a posteriori bound depends on still unconstructed factors through $\Psi_p$, as well as on $\sum_{n=0}^{k-1} \norm{x^{n+1}-x^n}^2$.
    The latter is bounded by the convergence results of \cref{sec:fb} together with \cref{thm:fb:growth:tracking}.
    With some care, we expect to be able to reduce the dependence on the unknown factors.
    However, since we need to be careful to avoid circular reasoning, we have left the refinement of this route to future work.
\end{remark}

\textbf{Second-order growth (likely missing, but workarounds exist):}
We also have another challenge: \cref{cor:pdps:values} requires the local strong subdifferentiability of $f+g$ to satisfy \eqref{eq:pdps:inexact-expansion:mono:gamma-cond}.
However, recalling \cref{rem:pdps:no-dual-sc}, $h^*$, importantly, does not have to be strongly subdifferentiable, \emph{we do not need to smoothen the total variation}.
The primal strong subdifferentiability, in contrast, can be difficult to verify, and may not hold in a function space setting or if $\Nmeas \cdot \Nelec$ is much less than the number of nodes in a finite element grid for $\conductivity$; compare \cite[§3.3]{tuomov2024online-eit}.
However, due to the nonlinearity of $S_{\bar u}$, and the property only having to be local, the noncompliance is not clear-cut.
In the case of primal-only algorithms, the total variation term can compensate for the lack of growth of the data term through overall metric subregularity, see \cite[Appendix A]{jauhiainen2020relaxed} and \cite[§4.3]{tuomov-regtheory}.
Such studies have not been made for primal-dual methods, and would be far more challenging due to the dualisation of total variation.
Moreover, no growth properties would be required by the subdifferential convergence \cref{thm:fb:subdiff}, but this result is currently not applicable to the PDPS, only to basic forward-backward splitting.

A workaround is to add to the problem an additional squared norm of the conductivity. In practise, we have observed no need for it.

\textbf{Existence of solutions (not guaranteed in infinite dimensions):}
It should also be observed that there does not necessarily exist a dual solution: $y$ achieving the supremum in \eqref{eq:eit:tv}.
This requires additional regularity from the primal solution.

\paragraph{Numerical results}

We perform the experiments on P1 finite element grids of 5039 nodes (“fine grid”) and 2917 nodes (“coarse grid”). There are $\Nelec=16$ evenly spaced electrodes, and we make $\Nmeas=16$ measurements by setting $U_{\measindex,\measindex}=1$ and $U_{\measindex,i}=0$ for $i \ne \measindex$.
We illustrate the ground-truth conductivity in \cref{fig:eit:reco}, along with the reconstructions from boundary measurements.
The reconstructions are not perfect due to EIT being a \emph{highly} ill-posed inverse problem.
The measurement data  $\{\EITmeas_\measindex\}_{\measindex=1}^{\Nmeas}$ is generated from the ground-truth data by applying the forward PDE on the finer grid to simulate $I_\measindex$, and then adding a low level of noise. Further details can be found in our numerical implementation \cite{tuomov-tracking-codes} or in \cite{tuomov2024online-eit}, whose setup our demonstration mirrors, aside from that work including an additional dynamical aspect.

As the initial primal iterate we take $\conductivity \equiv 1$, and as the initial dual iterate $y \equiv 0$.
In the Gauss–Seidel alternative of the algorithm, for each outer primal-dual iteration, we take 7 (inner) Gauss–Seidel steps for the forward PDE, and 1 Gauss–Seidel step for the adjoint PDE. These 7 steps are required in practise for $\trackingressum$ (see \eqref{eq:tracking:ressum}) to be small enough for \eqref{eq:pdps:inexact-expansion:mono} or \eqref{eq:pdps:inexact-expansion:smoothness} to hold without making $\tau$ very small.

\Cref{fig:eit:performance} shows the algorithm performance in terms of both the function value relative to the initial iterate---which is very descriptive in this type of problems---and as a relative distance to an approximate solution to \eqref{eq:eit:problem}. The latter is formed by taking $T=100000$ iterations of our Gauss–Seidel based method.

In summary, our Gauss–Seidel approach cuts the CPU footprint in one sixth, confirming and significantly further improving upon the results of \cite{jensen2022nonsmooth} on a simpler related problem.

\begin{figure}[t]
    \centering
    \begin{subfigure}{0.3\textwidth}\centering
        \includegraphics[width=0.7\linewidth]{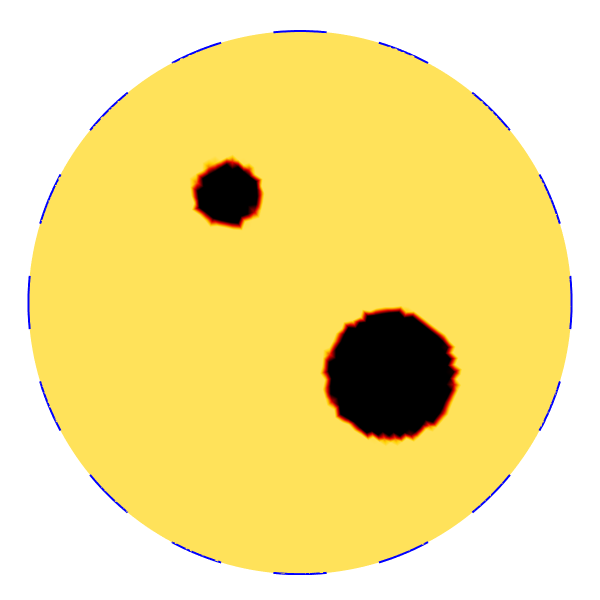}
        \caption{Ground truth}
    \end{subfigure}
    \begin{subfigure}{0.3\textwidth}\centering
        \includegraphics[width=0.7\linewidth]{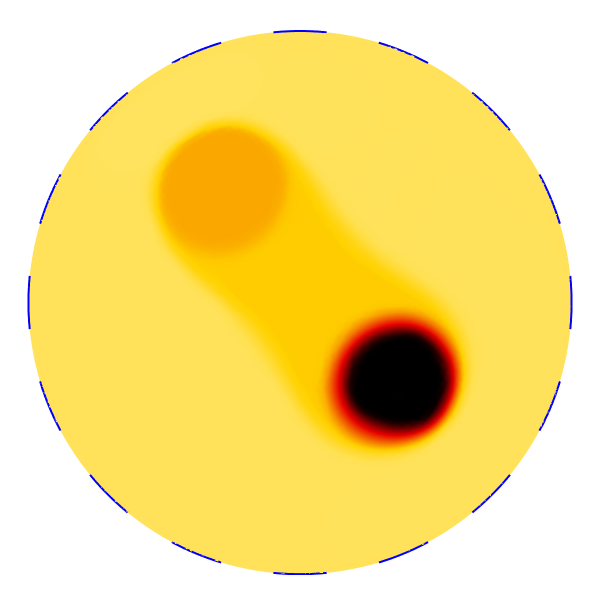}
        \caption{Reconstruction/exact}
    \end{subfigure}
    \begin{subfigure}{0.3\textwidth}\centering
        \includegraphics[width=0.7\linewidth]{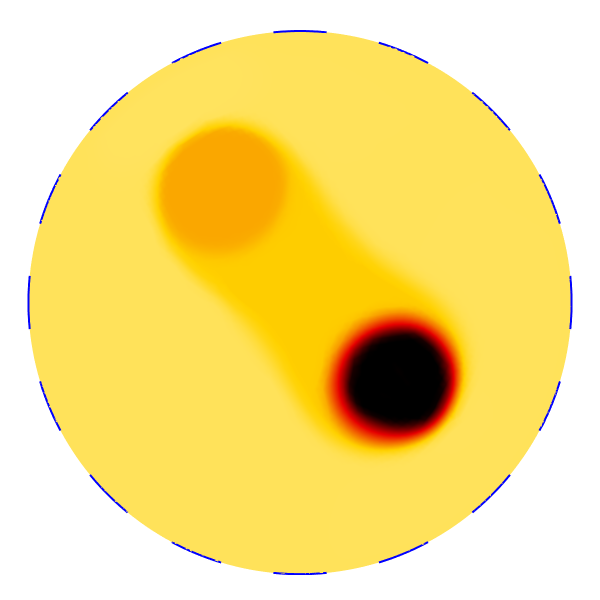}
        \caption{Reconstruction/GS}
    \end{subfigure}
    \caption{The ground-truth data for the EIT demonstration along with the reconstruction with both exact solutions of the PDE and its adjoint on each step, and with Gauss–Seidel (7 forward PDE steps, 1 adjoint step). The electrodes are also visible on the boundaries.}
    \label{fig:eit:reco}
\end{figure}

\begin{figure}[t]
    \centering
    \begin{tikzpicture}
        \begin{groupplot}[
            group style={
                group size=2 by 2,
                horizontal sep = 10ex,
                vertical sep = 13ex,
            },
            width=0.5\linewidth,
            height=0.3\linewidth,
            scaled x ticks=false,
            xminorticks=true,
            yminorticks=true,
            axis x line=bottom,
            axis y line=left,
            ymode=log,
            xmode=log,
            grid=major,
            label style={
                font=\small,
            },
            title style={
                at = {(0.0, 1.0)},
                anchor = south west
            },
            legend style={
                at = {(1.0, 1.0)},
                yshift=3pt,
                anchor = south east,
                inner sep=1pt,
                outer sep=1pt,
                draw=none,
                fill=none,
                font=\normalfont,
                legend columns=-1,
            },
            xlabel={CPU time (s) (logarithmic)},
        ]

            \nextgroupplot[
                title={Coarse grid},
                ylabel={Relative value},
                ignore legend,
                log basis y={10},
                ytickten={0,-1,-1.4},
            ]

            \addplot [color=Set2-A, line width=1pt]
                table [x=cpu_time, y=rel_value, col sep=comma]
                {eit/eit_gs_lowres.csv};
            \addlegendentry{Gauss--Seidel}

            \addplot [color=Set2-B, dashed, line width=1pt]
                table [x=cpu_time, y=rel_value, col sep=comma]
                {eit/eit_exact_lowres.csv};
            \addlegendentry{Exact}

            \nextgroupplot[
                ylabel={Relative error},
                log basis y={10},
                ytickten={0,-1,-2,-2.8},
                xmode=normal,
                xlabel={CPU time (s)},
            ]

            \addplot [color=Set2-A, line width=1pt]
                table [x=cpu_time, y=rel_error, col sep=comma]
                {eit/eit_gs_lowres.csv};
            \addlegendentry{Gauss--Seidel}

            \addplot [color=Set2-B, dashed, line width=1pt]
                table [x=cpu_time, y=rel_error, col sep=comma]
                {eit/eit_exact_lowres.csv};
            \addlegendentry{Exact}

            \nextgroupplot[
                title={Fine grid},
                ylabel={Relative value},
                ignore legend,
                log basis y={10},
                ytickten={0,-1,-1.4},
            ]

            \addplot [color=Set2-A, line width=1pt]
                table [x=cpu_time, y=rel_value, col sep=comma]
                {eit/eit_gs_highres.csv};
            \addlegendentry{Gauss--Seidel}

            \addplot [color=Set2-B, dashed, line width=1pt]
                table [x=cpu_time, y=rel_value, col sep=comma]
                {eit/eit_exact_highres.csv};
            \addlegendentry{Exact}

            \nextgroupplot[
                ylabel={Relative error},
                log basis y={10},
                ytickten={0,-1,-1.7},
                xmode=normal,
                xlabel={CPU time (s)},
            ]

            \addplot [color=Set2-A, line width=1pt]
                table [x=cpu_time, y=rel_error, col sep=comma]
                {eit/eit_gs_highres.csv};
            \addlegendentry{Gauss--Seidel}

            \addplot [color=Set2-B, dashed, line width=1pt]
                table [x=cpu_time, y=rel_error, col sep=comma]
                {eit/eit_exact_highres.csv};
            \addlegendentry{Exact}
        \end{groupplot}
    \end{tikzpicture}

    \caption{EIT algorithm performance. Top row: coarse grid, bottom row: fine grid. On the left, relative value $V(\conductivity^k)/V(\conductivity^0)$, where $V(\conductivity)=\frac12\sum_{\measindex=1}^{\Nmeas} \norm{I_\measindex - \EITmeas_\measindex}_{\inv\Sigma}^2 + \alpha\TV(\conductivity)$, and on the right, the relative error $\norm{\conductivity^k-\conductivity^T}/\norm{\conductivity^T}$ for $T=100000$ simulating the exact solution. The data shows 10000 iterations of both approaches.
    }
    \label{fig:eit:performance}
\end{figure}

\ifSubfilesClassLoaded{
    \bibliographystyle{jnsao}
    \bibliography{abbrevs.bib,tracking.bib}
}{}

\end{document}

%% file: minimal-surface.tex
\def\uspace{H^1(\Omega)}
\def\uspaceDual{H^1(\Omega)^*}
\def\dirichletSpace{H_0^1(\Omega)}
\def\dirichletSpaceDual{H^{-1}(\Omega)}
\def\tracespace{H^{1/2}(\partial\Omega)}
\def\tracespaceDual{H^{-1/2}(\partial\Omega)}
\def\trinv{\tr^\dagger}

\ifSubfilesClassLoaded{
    \section*{Minimal surface control}
}{
    \subsection{Minimal surface control}
    \label{sec:minsurf}
}

\paragraph{Problem formulation}

We now consider a somewhat more academic example, intended to demonstrate the application of our methods to both bilevel optimisation and nonlinear PDEs, while permitting comparison with Newton-type methods.
On a bounded Lipschitz domain $\Omega \subset \R^d$, we want to solve
\begin{subequations}
\label{eq:minsurf}
\begin{gather}
    \label{eq:minsurf:outer}
    \min_{u \in \uspace,\, x \in \tracespace} J(S_u(x)) + G(x)
    \shortintertext{subject to the minimal surface problem}
    \label{eq:minsurf:inner}
    S_u(x) = \argmin_{u \in \uspace} \int_\Omega \sqrt{1+\norm{\grad u(\xi)}^2} \d \xi + \delta_{\{x\}}(\tr u),
\end{gather}
\end{subequations}
where the trace $\tr \in \linear(\uspace; \tracespace)$ and the indicator function $\delta_{\{x\}}$ model the boundary condition $\tr u=x$.

For $J$ and $G$ we make somewhat arbitrary choices.
We impose zero boundary values on a subset $\Gamma$ of the boundary, and values between $0$ and $h_{\max}>0$ in $\Gamma^c \defeq \partial\Omega \setminus\Gamma$, by taking
\begin{equation}
    \label{eq:minsurf:r}
    G(x)=\delta_{C}(x)
    \quad\text{for}\quad
    C = \{ x \in \tracespace \mid x=0 \text{ on } \Gamma,\, 0 \le x \le h\max \text{ on } \Gamma^c\}.
\end{equation}
We then seek to maximise the volume under the surface and impose an “opening” of area $a$ in the complement of $\Gamma^c$, by setting, for some regularisation parameter $\lambda>0$,
\begin{equation}
    \label{eq:minsurf:j}
    J(u) = \frac{\lambda}{2}\adaptabs{\int_{\Gamma^c} \tr u(\xi)\d\xi -a} -\int_\Omega u(\xi) \d\xi + \frac{\epsilon}{2}\norm{u}_{\uspace}^2.
\end{equation}
The final coercion term for a small $\varepsilon>0$ ensures the existence of solutions to \eqref{eq:minsurf}.\footnote{%
    The proof of existence follows the same lines as \cite[Theorem 2.2]{delosreyes2014learning}.
}
Existence or second-order growth of the outer problem is, however, not required when we derive convergence of an outer forward-backward method from \cref{thm:fb:subdiff}.

\begin{remark}
    It would be more tasteful, and avoid the regularisation parameter $\lambda$, to exactly impose the constraint $\int_{\Gamma^c} x \xi \d\xi = a$ in the nonsmooth $G$. The proximal map of $G$ within $\Gamma^c$ would then be a projection into a scaled probability simplex. This is feasible, see e.g., \cite{angerhausen2022stochastic}, and could be employed with our forward-backward method. The choice would, however, exclude a comparison to the semismooth Newton's method, as developing the Newton derivative of such a $G$ is outside the scope of this work.
\end{remark}

\paragraph{Inner problem and adjoint}

Let $f: \uspace \to \R$ and $g: \uspace \times \tracespace \to \extR$,
\[
    f(u) \defeq \int p(\grad u)(\xi) \d\xi,
    \quad
    p(z)(\xi) \defeq \sqrt{1+\norm{z(\xi)}^2},
    \quad\text{and}\quad
    g(u, x) \defeq \delta_{\{x\}}(\tr u).
\]
Then the inner problem \eqref{eq:minsurf:inner} reads as $S_u(x) \in \argmin_{u \in \uspace}~ f(u) + g(u; x)$.
For any $\theta>0$, let
\[
    P: \uspace \times \tracespace \to \uspace,
    \quad
    P(u, x) \defeq \prox_{\theta g(\freevar; x)}(u - \theta \grad_u f(u)).
\]
Then $u = P(u, x)$ characterises the solutions of the inner problem \cite[Theorem 4.2 \& Corollary 6.22]{clason2020introduction}.

We have
\begin{equation}
    \label{eq:minsurf:fu}
    f'(u) = p'(\grad u)\grad \in \uspaceDual,
\end{equation}
and
\[
    \subdiff_u g(u, x) \subset \uspaceDual,
    \quad
    \subdiff_u g(u, x) = \tracespaceDual \tr
    \quad\text{when}\quad \tr u=x.
\]
Moreover, $p'(z) \in L^2(\Omega; \R^d)^*$ and $p''(z) \in [L^2(\Omega; \R^d)^*]^{\otimes 2}$ with the superposition representations
\begin{equation}
    \label{eq:minsurf:pprime}
    p'(z)h \defeq \int_\Omega \adaptiprod{\frac{z(\xi)}{\sqrt{1+\norm{z(\xi)}^2}}}{h(\xi)} \d\xi,
    \quad
    p''(z)(h_1, h_2) \defeq
    \int_{\Omega}
    \adaptiprod{h_1(\xi)}{\frac{\Id - \frac{z(\xi) \otimes z(\xi)}{1+\norm{z(\xi)}^2 }}{\sqrt{1+\norm{z(\xi)}}}h_2(\xi)}\d\xi.
\end{equation}

The proximal map of $g$ is
$
    \prox_{\theta g(\freevar; x)}(\tilde u) = \argmin_u \frac{1}{2}\norm{u-\tilde u}_{\uspace}^2 + \theta\delta_{\{x\}}(\tr u),
$
whose solutions $u^+$ are characterised by
\begin{equation}
    \label{eq:minsurf:step-oc}
    \iprod{u^+-\tilde u}{\freevar}_{\uspace} + \lambda\tr = 0,
    \quad \lambda \in \tracespaceDual,
    \quad
    \tr u^+=x.
\end{equation}
On the other hand, the forward step $\tilde u=u - \theta \grad_u f(u)$ equivalently reads
\[
    \iprod{\tilde u - u}{\freevar}_{\uspace} + \theta\diffwrt{f}{u}(u, x) = 0.
\]
Together with \eqref{eq:minsurf:step-oc} and \eqref{eq:minsurf:fu}, this gives
\begin{equation}
    \label{eq:minsurf:step-oc:full}
    \iprod{u^+-u}{\freevar}_{\uspace} + \theta p'(\grad u)\grad + \lambda\tr = 0,
    \quad \lambda \in \tracespaceDual,
    \quad
    \tr u^+=x.
\end{equation}
We equip $\uspace$ with the trace inner product\footnote{Equivalence of the induced norm to the standard one follows by combining the Poincaré inequality \cite[Corollary 9.19]{brezis2011functional} with the existence of a bounded right inverse of the trace operator, e.g., \cite[Theorem 3.37]{mclean2000strongly}.} $\iprod{u}{v}_{\uspace} \defeq \iprod{\grad u}{\grad v}_{L^2(\Omega; \R^d)} + \iprod{\tr u}{\tr v}_{L^2(\partial\Omega)}$.
Then, in a standard fashion, restricting the test space to $\dirichletSpace$ completely determines $u^+$, and the variable $\lambda$ becomes superfluous.
We can thus rewrite \eqref{eq:minsurf:step-oc:full} as
\[
    \iprod{\grad(u^+ - u)}{\grad_0 v}_{L^2(\Omega; \R^d)}
    + \theta p'(\grad u)\grad_0 v = 0
    \quad\text{for all}\quad v \in \dirichletSpace,
    \quad
    \tr u=x.
\]
This fully determines $u^+ = P(u, x)$.
In particular, with $u=\thisu$ and $\nextu = u^+$, this determines a forward-backward update for the inner variable.
Setting $u^+=u$, we also see that the inner solutions $u=S_u(x)$ are characterised by $0=T(u, x)$ for
\[
    T: \uspace \times \tracespace \to \dirichletSpaceDual \times \tracespace,
    \quad
    T(u, x) =
    \begin{pmatrix}
        p'(\grad u)\grad_0,
        &
        \tr u - x
    \end{pmatrix}.
\]
Then $\diffwrt{T}{u}(u, x) \in \linear(\uspace; \dirichletSpaceDual \times \tracespace)$
and  $\diffwrt{T}{x}(u, x) \in \linear(\tracespace; \dirichletSpaceDual \times \tracespace)$,
\[
    \diffwrt{T}{u}(u, x) h_u = \begin{pmatrix}
        p''(\grad u)(\grad_0 \freevar, \grad h_u),
        &
        \tr h_u
    \end{pmatrix}
    \quad\text{and}\quad
    \diffwrt{T}{x}(u, x) h_x = \begin{pmatrix}
        0,
        &
        -h_x
    \end{pmatrix}.
\]
Thus, the reduced adjoint equation \eqref{eq:tracking:reduced-adjoint} gives $S_w(x)=(w_\Omega, w_{\partial\Omega}) \in \dirichletSpace \times \tracespaceDual$ as the solution of
$
    p''(\grad u)(\grad_0 w_\Omega, \grad \freevar) + w_{\partial\Omega}\tr + J'(u) = 0
$
for $u=S_u(x)$.
This is to say that for all test functions $v \in \uspace$, we have
\begin{equation}
    \label{eq:minsurf:reduced-adjoint}
    p''(\grad u)(\grad_0 w_\Omega, \grad v) + \dualprod{w_{\partial\Omega}}{\tr v}_{\tracespaceDual,\tracespace} + J'(u)v = 0.
\end{equation}
Then, by \eqref{eq:tracking:reduced-adjoint-diff-transformation},
\begin{equation}
    \label{eq:minsurf:reduced-adjoint-diff-transformation}
    F'(x)
    = [J \circ S_u]'(x)
    = J'(S_u(x))S_u'(x)
    = S_w(x) \diffwrt{T}{x}(S_u(x), x)
    = -w_{\partial\Omega}.
\end{equation}

\paragraph{Overall algorithm}
Given initial iterates $x^0 \in \tracespace$ and $u^0 \in \uspace$, and writing  $F=J \circ S_u$, our overall inexact forward-backward algorithm,
$
    \nextx = \prox_{\tau G}(\thisx - \tau \nextestgrad),
$
decomposes into:
\begin{enumerate}
    \item \textbf{Inner forward-backward step:} For an inner step length parameter $\theta \in (0, \tau L)$, where $L$ is a Lipschitz factor of $f'$, solve $\nextu \in \uspace$ from the linear system
    \[
        \iprod{\grad(\nextu - \thisu)}{\grad_0 v}_{L^2(\Omega; \R^d)}
        + \theta p'(\grad \thisu)\grad_0 v = 0
        \quad\text{for all}\quad v \in \dirichletSpace,
        \quad
        \tr \nextu=\thisx.
    \]
    After discretisation, this step is efficient to do exactly, as the stiffness matrix can be pre-factorised.
    \item \textbf{Adjoint step:} find $(\nexxt w_\Omega, \nexxt w_{\partial\Omega}) \in \dirichletSpace \times \tracespaceDual$ by (a) solving or (b) taking a single step of Gauss–Seidel splitting on a discretisation of the linear system
    \[
        p''(\grad \nextu)(\grad_0 \nexxt w_\Omega, \grad v) + \dualprod{\nexxt w_{\partial\Omega}}{\tr v}_{\tracespaceDual,\tracespace} + J'(\nextu) v = 0
        \quad\text{for all}\quad v \in \uspace.
    \]
    \item \textbf{Outer forward-backward step:} Update
    \begin{equation}
        \label{eq:minsurf:outer-update}
        \nextx \defeq \argmin_{x \in \tracespace} G(x) + F(\thisx) - \dualprod{\nexxt w_{\partial\Omega}}{x-\thisx}_{\tracespaceDual,\tracespace} + \frac{1}{2}\norm{x-\thisx}_M^2,
    \end{equation}
    where $M=\inv\tau \mathscr{I}$ for the injection $\mathscr{I}: H^{1/2}(\partial\Omega) \hookrightarrow L^2(\partial\Omega)\hookrightarrow H^{-1/2}(\Omega)$, $x \mapsto \iprod{x}{\freevar}_{L^2(\Omega)}$, and an outer step length parameter $\tau>0$.
    In numerical practise, with also $\nexxt w_{\partial\Omega} \in L^2(\partial\Omega)$, this is simply the standard forward-backward update $\nextx \defeq \prox_{\tau G}(\thisx + \tau\nexxt w_{\partial\Omega})$.
\end{enumerate}

This works for any $J$ and $G$ that satisfy the assumptions of our general theory.
For the specific choices \cref{eq:minsurf:r,eq:minsurf:j}, we have $J'(u)v=\int_\Omega v(\xi) d\xi + (\int_{\Gamma^c} \tr u(\xi) d\xi -a)\int_{\Gamma^c} \tr v(\xi) d\xi$, and denoting $|\Gamma^c|=\int_{\Gamma^c} \d\xi$ and assuming $\tilde x \in L^2(\partial\Omega)$,
\[
    \prox_{\tau G}(\tilde x)(\xi) =
    \begin{cases}
        0, & \xi \in \Gamma, \\
        \min\{\max\{\tilde x(\xi), 0\}, h_{\max}\}, & \xi \in \Gamma^c.
    \end{cases}
\]

The mapping $M$ is self-adjoint and positive (semi-)definite in the sense of \cref{sec:intro}.
The update \eqref{eq:minsurf:outer-update} can be written in the implicit form \eqref{eq:fb:implicit} with $\Xi=0$, i.e.,
\begin{equation}
    \label{eq:minsurf:outer-fb-implicit}
    0 \in \subdiff G(\nextx) - w_\Omega + M(\nextx-\thisx).
\end{equation}

We next sketch how convergence of the method could be obtained. Rigorous verification of all conditions is, again, outside the scope of the present work, as that would demand a detailed sensitivity analysis of the minimal surface problem.

\begin{claim}
    For small enough $\theta, \tau>0$, the above method satisfies $\inf_{x^* \in \subdiff G(\thisx)} \norm{x^* + F'(\thisx)} \to 0$.
\end{claim}

\begin{proof}[Sketch of proof]
    To prove the inner tracking \cref{ass:tracking:main}\,\cref{item:tracking:main:inner-tracking}, the first idea is to use the exemplary \cref{thm:tracking:inner-fb}.
    However, since $f$ and $g(\freevar; x)$ are only “partially” strongly convex, and $f$ only on bounded sets, we need to combine their contributions, and ensure that we work on bounded sets of $u$.

    \textbf{Bounded and Lipschitz solutions maps:}
    The trace operator has a bounded right-inverse $\tr^\dagger$ \cite[Theorem 3.37]{mclean2000strongly}.
    With the definition of $f$ and $S_u(\thisx)$ this gives for some constants $C_1, C_2>0$ that
    \[
        \norm{\grad[S_u(\thisx)]}_{L^2(\Omega; \R^d)}
        \le f(S_u(\thisx))
        \le f(\tr^\dagger \thisx)
        \le C_1 + \norm{\grad\tr^\dagger\thisx}_{L^2(\Omega; \R^d)}
        \le C_1 + C_2\norm{\thisx}_{\tracespace}.
    \]
    This and Poincaré's inequality \cite[Corollary 9.19]{brezis2011functional} bound $\norm{S_u(\thisx)}_{\uspace}$ as a function of $\norm{\thisx}_{\tracespace}$.

    To prove that $S_u$ is Lipschitz (from $\Dom R \subset \tracespace$ to $\uspace$) for some $\pi_u$, we can use \cite[Theorem 4.51]{bonnans2000perturbation}.\footnote{It is worth noting that the critical cone condition (3.147) in \cite{bonnans2000perturbation} essentially does the $H^1 \rightsquigarrow H_0^1$ replacement. The upper Lipschitz condition on (4.116) is essentially the existence of a bounded inverse of the trace operator.}
    We can then, in principle, apply the implicit function theorem to the reduced adjoint equation \eqref{eq:minsurf:reduced-adjoint} to show that $S_w$ is $L$-Lipschitz for some $L>0$ (on the bounded set $\Dom R$).
    Then $F'$ given by \eqref{eq:minsurf:reduced-adjoint-diff-transformation} is also $L$-Lipschitz (on $\Dom R$).

    \textbf{Inner tracking:}
    Suppose $\thisu \in \ulocalset$ for some bounded $\ulocalset \subset \uspace$ that also contains $S_u(\thisx)$.
    Write $\grad^\hstar\grad \defeq \iprod{\grad \freevar}{\grad \freevar}_{L^2(\Omega; \R^d)}$ and $\tr^\hstar\tr \defeq \iprod{\tr \freevar}{\tr \freevar}_{L^2(\partial \Omega)}$.
    Both operators are in $\linear(\uspace;\uspaceDual)$.
    It follows from the expression for $p''$ in \eqref{eq:minsurf:pprime} that $f'$ is $\gamma_f\nabla^\hstar\nabla$-monotone \emph{on the bounded set} $\ulocalset$ for some $\gamma_f>0$ in the sense of \cref{sec:operator-reg:def}.
    Similarly, $f'$ can be verified to be $2\nabla^\hstar\nabla$-$*$-co-coercive through the mean value theorem and the equivalence of \cref{lemma:opeartor-reg:coco-lipschitz}.
    Moreover, $g(\freevar; x)$ is $\gamma_g\tr^\hstar\tr$-subdifferentiable for \emph{any} $\gamma_g>0$.
    For simplicity take $\gamma_g=\gamma_f/(2(1+\gamma_f\theta))$, where $\theta \in (0, \inv \ell)$ is the inner step length parameter.
    Let the injection $\Inj = \grad^\hstar\grad u + \tr^\hstar\tr: \uspace \hookrightarrow \uspaceDual$.
    Let $\theta>0$ be small enough that $A \defeq \inv\theta\Inj - \gamma_f\grad^\hstar\grad \ge \epsilon\Inj$ for some $\epsilon>0$.
    We then have
    \[
        \inv\theta\Inj + 2\gamma_g\tr^\hstar\tr
        =
        A
        + 2\gamma_g(\tr^\hstar\tr + \grad^\hstar\grad)
        - (\gamma_f-2\gamma_g)\grad^\hstar\grad
        =
        (1+\gamma_f)A.
    \]
    Abbreviating $\this{\bar u} = S_u(\thisx)$, we exploit the operator-relative \cref{lemma:smoothness:monotonicity} with $\zeta=1/2$ to obtain
    \[
        \dualprod{f'(\thisu) + \subdiff g(\nextu; \thisx)}{\nextu-\this{\bar u}}_{\uspaceDual, \uspace}
        -\frac{\gamma_f}{2}\norm{\thisu-\this{\bar u}}_{\nabla^\hstar\nabla}^2
        \ge
        \gamma_g\norm{\nextu-\this{\bar u}}_{\tr^\hstar\tr}^2
        - \frac{\ell}{2}\norm{\nextu-\thisu}_{\nabla^\hstar\nabla}^2.
    \]

    Combining with $f'(\thisu) + \subdiff g(\nextu; \thisx)=-\inv\theta\Inj(\nextu-\thisu)$ and the Pythagoras' identity \eqref{eq:norms:pythagoras} gives
    \begin{equation}
        \label{eq:minsurf:contractivity}
        \frac{1}{2}\norm{\thisu-\this{\bar u}}_{A}^2
        \ge
        \frac{1}{2}\norm{\nextu-\this{\bar u}}_{\inv\theta\Inj + 2\gamma_g\tr^*\tr}^2
        \ge
        \frac{1+\gamma_f}{2}\norm{\nextu-\this{\bar u}}_{A}^2.
    \end{equation}
    Now using the Lipschitz continuity of $S_u$, and $\norm{\freevar}_{A}$ being equivalent to the standard norm in $H^1(\Omega)$, we obtain \cref{ass:tracking:main}\,\cref{item:tracking:main:inner-tracking}.

    \textbf{Bounded inner iterates (challenge):}
    We did not need $\nextu \in \ulocalset$ to use \cref{lemma:smoothness:monotonicity}, only $\thisu,S_u(\thisx) \in \ulocalset$. To pass to the subsequent step, we now need to ensure that.
    Obtaining such an \emph{a posteriori} bound is, however, challenging (nevertheless, see  \cref{rem:eit:bounds}).
    It is easier to obtain an \emph{a priori} bound with controlled escape:
    Indeed, \eqref{eq:minsurf:contractivity} and $S_u(\thisx) \in \ulocalset$ show that $\nextu$ belongs to some enlarged bounded set $\ulocalset'$.
    Performing multiple inner iterations $j$, \emph{if necessary}, $\nextu=u^{k+1,j}$ can, consequently, be returned to within a set distance of $S_u(\thisx)$, hence to some bounded set $\ulocalset$. In practise, we have observed no need for this.

    \textbf{Adjoint tracking and differential transformation:}
    \Cref{ass:tracking:main}\,\cref{item:tracking:main:adjoint-tracking,item:tracking:main:differential-transformation} follow from \cref{thm:tracking:reduced-adjoint-splitting} (and \cref{ex:splitting:Gauss–Seidel} for the Gauss–Seidel adjoint steps, after passing to a finite element subspace).
    The required Lipschitz continuities and bounds of $\diffwrt{T}{u}$ and $\diffwrt{T}{x}$ are easily verified; unlike in the case of EIT, $\norm{\diffwrt{T}{z}(u, z)}_{M,\dirichletSpaceDual \times \tracespaceDual}$ becoming an $L^2$ operator norm, causes no difficulty.

    \textbf{Outer step:}
    Switching \eqref{eq:minsurf:outer-fb-implicit} to its Riesz representation (or we could again use \cref{thm:fb:growth:tracking}), we now verify \cref{ass:fb:descent,ass:fb:approx-cont} for the outer step through \cref{cor:fb:tracking}\,\cref{item:fb:tracking:weak,item:fb:tracking:approx-cont}.
    This requires small enough $\tau>0$.
    Convergence then follows from \cref{thm:fb:subdiff}.
\end{proof}

\paragraph{Semismooth Newton's method (SSN)}

Writing $\mathscr{R}w_{\partial\Omega}$ for the Riesz representation of $w_{\partial\Omega}$, we consider the optimality conditions derived above,
\begin{gather*}
    0 = H(x, u, w_\Omega, w_{\partial\Omega})
    \defeq
    \begin{pmatrix}
    x - \prox_{\tau G}(x + \tau\mathscr{R} w_{\partial\Omega})
    \\
    T(u, x),
    \\
    p''(\grad u)(\grad_0 w_\Omega, \grad \freevar)  + w_{\partial\Omega}\tr  + J'(u)
    \end{pmatrix},
    \\
    H: \tracespace \times \uspace \times \dirichletSpace \times \tracespaceDual:
    \to \tracespace \times (\dirichletSpaceDual \times \tracespace) \times \uspaceDual.
\end{gather*}
We Newton-differentiate
\[
     D_N H(x, u, w_\Omega, w_{\partial\Omega})
     =
     \begin{pmatrix}
      \Id - A & 0 & 0 &  -\tau A
     \\
     \diffwrt{T}{x}(u, x) & \diffwrt{T}{u}(u, x) & 0 & 0
     \\
     0
     &
     p'''(\grad u)(\grad_0 w_\Omega, \grad \freevar, \grad \freevar)
     + J''(u)
     &
     p''(\grad u)(\grad_0 \freevar, \grad \freevar)
     &
     \tr^*
     \end{pmatrix},
\]
where $A \defeq D_N \prox_{\tau G}(x + \tau \mathscr{R} w_{\partial\Omega})$ is the Newton derivative of the proximal map; see, e.g., \cite[Chapter 14]{clason2020introduction}.
Due to the symmetricity of $p'''(\grad u)$, the order of application of the parameters does not matter while $p''(\grad u)$ is first applied in the $\grad_0$ component.
Write $\this\eta \defeq (\thisx, \thisu, \this w_\Omega, \this w_{\partial\Omega})$ Now the SSN computes on each iteration a $\this s$ such that $H'(\this\eta) \this s = -H(\this\eta)$, and updates $\nexxt\eta \defeq \this\eta + \this s$.

In practise, $H'(\this v)$ is poorly conditioned (or even not invertible).
We therefore use the damped variant $[H'(\this\eta) + \vartheta \Id]\this s = -H(\this\eta)$ for a damping parameter $\vartheta>0$. We can then expect only linear convergence \cite[Theorem 14.2]{clason2020introduction}.
We also tried to only dampen a reduced system in conjunction with an active-set strategy, but had no success with it.

\paragraph{Numerical results}

\begin{figure}[t!]
    \centering
    \begin{subfigure}{0.45\textwidth}\centering
        \includegraphics[width=0.7\linewidth]{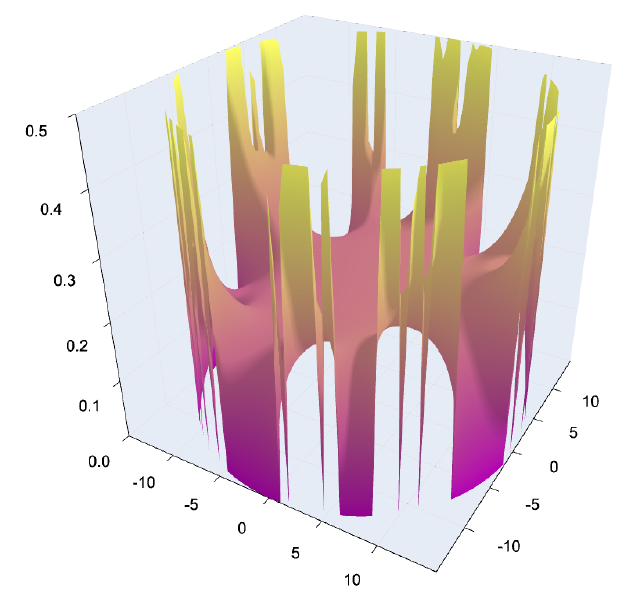}
        \caption{Forward-backward + Gauss–Seidel}
    \end{subfigure}
    \begin{subfigure}{0.45\textwidth}\centering
        \includegraphics[width=0.7\linewidth]{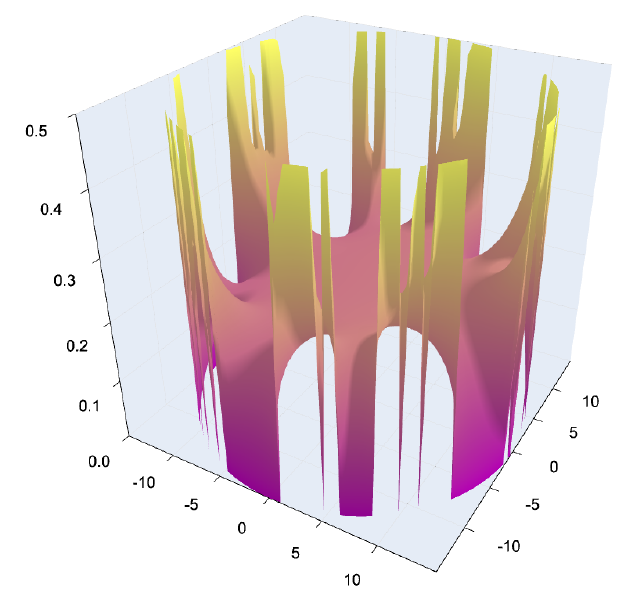}
        \caption{Semismooth Newton}
    \end{subfigure}
    \caption{Numerical solutions for the minimal surface control demonstration.}
    \label{fig:minsurf:solution}
\end{figure}

\pgfplotsset{
    log shift y/.style={
        y coord trafo/.code={
            \pgfmathparse{ln(##1-#1)/ln(10)}
        },
        y coord inv trafo/.code={
            \pgfmathparse{10^(##1)+#1}
        },
    }
}

\begin{figure}[t!]
    \centering
    \begin{tikzpicture}
        \begin{groupplot}[
            group style={
                group size=2 by 3,
                horizontal sep = 10ex,
                vertical sep = 13ex,
            },
            width=0.5\linewidth,
            height=0.3\linewidth,
            scaled x ticks=false,
            xminorticks=true,
            yminorticks=true,
            axis x line=bottom,
            axis y line=left,
            xmode=log,
            grid=major,
            label style={
                font=\small,
            },
            title style={
                at = {(0.0, 1.0)},
                anchor = south west
            },
            legend style={
                at = {(1.0, 1.0)},
                yshift=3pt,
                anchor = south east,
                inner sep=1pt,
                outer sep=1pt,
                draw=none,
                fill=none,
                font=\normalfont,
                legend columns=-1,
            },
            xlabel={CPU time (s) (logarithmic)},
        ]
            \nextgroupplot[
                title={Coarse grid},
                ylabel={Relative value$^*$},
                ignore legend,
                log shift y = 0.0156,
                ytick = {0.0157, 0.02, 0.1},
                yticklabels = {$0.0157$, $0.02$, $0.1$},
            ]
            \addplot [color=Set2-A, line width=1pt]
                table [x=cpu_time, y=rel_value, col sep=comma]
                {minsurf/minsurf_fbfb_gs_lowres_eps.csv};
            \addlegendentry{FB + Gauss--Seidel}

            \addplot [color=Set2-B, dashed, line width=1pt]
                table [x=cpu_time, y=rel_value, col sep=comma]
                {minsurf/minsurf_fbfb_exact_lowres_eps.csv};
            \addlegendentry{FB + Exact}

            \addplot [color=Set2-C, dashdotted, line width=1pt]
                table [x=cpu_time, y=rel_value, col sep=comma]
                {minsurf/minsurf_newton_lowres_eps.csv};
            \addlegendentry{damped SSN}

            \nextgroupplot[
                ylabel={Relative discrepancy},
                ytickten={-2,-5,-8,-11,-14,-17},
                ymode=log
            ]

            \addplot [color=Set2-A, line width=1pt]
                table [x=cpu_time, y=rel_discrepancy, col sep=comma]
                {minsurf/minsurf_fbfb_gs_lowres_eps.csv};
            \addlegendentry{FB + Gauss--Seidel}

            \addplot [color=Set2-B, dashed, line width=1pt]
                table [x=cpu_time, y=rel_discrepancy, col sep=comma]
                {minsurf/minsurf_fbfb_exact_lowres_eps.csv};
            \addlegendentry{FB + Exact}

            \addplot [color=Set2-C, dashdotted, line width=1pt]
                table [x=cpu_time, y=rel_discrepancy, col sep=comma]
                {minsurf/minsurf_newton_lowres_eps.csv};
            \addlegendentry{damped SSN}

            \nextgroupplot[
                title={Fine grid},
                ylabel={Relative value$^*$},
                ignore legend,
                log shift y = 0.0170,
                ytick = {0.0171, 0.02, 0.1},
                yticklabels = {$0.0171$, $0.02$, $0.1$},
            ]

            \addplot [color=Set2-A, line width=1pt]
                table [x=cpu_time, y=rel_value, col sep=comma]
                {minsurf/minsurf_fbfb_gs_highres_eps.csv};
            \addlegendentry{FB + Gauss--Seidel}

            \addplot [color=Set2-B, dashed, line width=1pt]
                table [x=cpu_time, y=rel_value, col sep=comma]
                {minsurf/minsurf_fbfb_exact_highres_eps.csv};
            \addlegendentry{FB + Exact}

            \addplot [color=Set2-C, dashdotted, line width=1pt]
                table [x=cpu_time, y=rel_value, col sep=comma]
                {minsurf/minsurf_newton_highres_eps.csv};
            \addlegendentry{damped SSN}

            \nextgroupplot[
                ylabel={Relative discrepancy},
                ytickten={-2,-5,-8,-11,-14,-17},
                ymode=log,
            ]

            \addplot [color=Set2-A, line width=1pt]
                table [x=cpu_time, y=rel_discrepancy, col sep=comma]
                {minsurf/minsurf_fbfb_gs_highres_eps.csv};
            \addlegendentry{FB + Gauss--Seidel}

            \addplot [color=Set2-B, dashed, line width=1pt]
                table [x=cpu_time, y=rel_discrepancy, col sep=comma]
                {minsurf/minsurf_fbfb_exact_highres_eps.csv};
            \addlegendentry{FB + Exact}

            \addplot [color=Set2-C, dashdotted, line width=1pt]
                table [x=cpu_time, y=rel_discrepancy, col sep=comma]
                {minsurf/minsurf_newton_highres_eps.csv};
            \addlegendentry{damped SSN}

            \nextgroupplot[
                title={Extra-fine grid},
                ylabel={Relative value$^*$},
                ignore legend,
                log shift y = 0.0163,
                ytick = {0.0164, 0.02, 0.1},
                yticklabels = {$0.0164$, $0.02$, $0.1$},
            ]

            \addplot [color=Set2-A, line width=1pt]
                table [x=cpu_time, y=rel_value, col sep=comma]
                {minsurf/minsurf_fbfb_gs_extrares_eps.csv};
            \addlegendentry{FB + Gauss--Seidel}

            \addplot [color=Set2-B, dashed, line width=1pt]
                table [x=cpu_time, y=rel_value, col sep=comma]
                {minsurf/minsurf_fbfb_exact_extrares_eps.csv};
            \addlegendentry{FB + Exact}

            \addplot [color=Set2-C, dashdotted, line width=1pt]
                table [x=cpu_time, y=rel_value, col sep=comma]
                {minsurf/minsurf_newton_extrares_eps.csv};
            \addlegendentry{damped SSN}

            \nextgroupplot[
                ylabel={Relative discrepancy},
                ytickten={-2,-5,-8,-11,-14,-17},
                ymode=log
            ]

            \addplot [color=Set2-A, line width=1pt]
                table [x=cpu_time, y=rel_discrepancy, col sep=comma]
                {minsurf/minsurf_fbfb_gs_extrares_eps.csv};
            \addlegendentry{FB + Gauss--Seidel}

            \addplot [color=Set2-B, dashed, line width=1pt]
                table [x=cpu_time, y=rel_discrepancy, col sep=comma]
                {minsurf/minsurf_fbfb_exact_extrares_eps.csv};
            \addlegendentry{FB + Exact}

            \addplot [color=Set2-C, dashdotted, line width=1pt]
                table [x=cpu_time, y=rel_discrepancy, col sep=comma]
                {minsurf/minsurf_newton_extrares_eps.csv};
            \addlegendentry{damped SSN}

        \end{groupplot}
    \end{tikzpicture}
    \caption{Minimal surface control algorithm performance. Top row: coarse grid, middle row: fine grid, bottom row: extra-fine grid. On the left, relative value $V(x^k)/V(x^0)$, for ($^*$) the positivity-shifted objective function $V(x)=[F+G](x) + h_{\max} \int_\Omega d \xi$.
    The relative value is, moreover, displayed on a shifted logaritmic scale $t \mapsto \log_{10}(t - t_0)$, for a low value $t_0$ on which to focus.
    On the right, the relative discrepancy $d_k/d_0$ for $d_k \defeq \norm{\thisx - \prox_{\tau G}(\thisx + \tau \grad F(\thisx))}$.
    The data shows 20000 iteration of the forward-backward approaches, and 2000 iterations of SSN.}
    \label{fig:minsurf:performance}
\end{figure}

We perform the experiments on the same two grids of $\Omega=B(0, 15)$ as in the EIT demonstration of \cref{sec:eit}, as well as an “extra-fine” grid of 17281 nodes.
We arbitrarily take $\Gamma = \{(x, y) \in \partial\Omega \mid \cos(5\phi) + 1.2\sin(9\phi) < 0 \}$, where $\phi=\arctan(y, x)$. We take $a=15$, $h_{\max}=0.5$, and $\epsilon=0.1$.
The resulting surface is shown in \cref{fig:minsurf:solution}.
The initial iterate is $x^0=\prox_R(0.1\chi_\Omega)$, $u^0 \equiv 0$, and $w^0 \equiv 0$; the projection is performed for the performance graphs in \cref{fig:minsurf:performance} to display meaningful relative values. Moreover, the function values in the graphs are shifted to be positive, for the logarithmic scale to be meaningful.

In the “FB + Gauss–Seidel” variant of our algorithm, on each outer iteration, we only take a single Gauss–Seidel step for the adjoint equation.
In the “FB + Exact” variant, we solve the adjoint equation exactly (up to numerical precision).
For all algorithms, the outer step length parameter $\tau=0.0001$, while the inner step length is $\theta=0.9/L$, where $L$ is an estimate of the Lipschitz factor of $f'(\freevar, x)$.
For the outer step length, the SSN is not dependent on a condition such as $\tau L \le 1$, but we were unable to significantly increase the value without destabilising the algorithm.
To stabilise the SSN, we, moreover, had to manually fine-tune the dampening parameter to $\vartheta=0.05$ for all grids.
Further details can be found in our numerical implementation \cite{tuomov-tracking-codes}.

The (damped) semismooth Newton's method is surprisingly slow. Not only does each step require a long time to factorise the system matrix, it also requires an unusually high number of iterations: \cref{fig:minsurf:performance} shows 2000 iterations of the SSN.
Nevertheless, on the fine grid (but not the extra-fine or coarse grid), it eventually reaches much lower values for the discrepancy, than the forward-backward variants within their maximum iteration count of 20000.
Unlike in the EIT experiments, we did not include graphs of the relative distance to a high quality solution, because the SSN did not converge to that solution. It seemed to get stuck in a different critical point of a lower function value.

In summary, the variant of our method that uses Gauss–Seidel splitting for the adjoint equation, significantly outperforms the alternatives.
Readers skeptical of these observations are invited to verify them by experimenting with our implementation \cite{tuomov-tracking-codes}, and possibly improving it.
It is very important here that we are treating a problem where the parametrisation affects the system matrix: we would not expect our algorithm to outperform the SSN (in applicable cases) if the matrix could be prefactorised.

\ifSubfilesClassLoaded{
    \bibliographystyle{jnsao}
    \bibliography{abbrevs.bib,tracking.bib}
}{}
\end{document}